\documentclass[
11pt,                          
english                        
]{article}

\usepackage[english]{babel}    
\usepackage{amsmath}           
\usepackage[utf8]{inputenc}    
\usepackage[T1]{fontenc}       
\usepackage{exscale}           
\usepackage[sort]{cite}        
\usepackage{eucal}
\usepackage{array}             
\usepackage[a4paper]{geometry} 
\usepackage{xspace}            
\usepackage{tikz}              
\usepackage{tikz-cd}
\usepackage{amssymb}           
\usepackage{bbm}               
\usepackage{mathtools}         
\usepackage[amsmath,thmmarks,hyperref]{ntheorem} 
\usepackage{mathrsfs}          
\usepackage{stmaryrd}          
\usepackage{paralist}          
\usepackage{aliascnt}          
\usepackage{enumitem}
\usepackage{lmodern}           
\usepackage{xcolor}            
\usepackage[expansion=false    
           ]{microtype}        
\usepackage[nottoc]{tocbibind} 
\usepackage[
           final=true,         
           pdfpagelabels       
           ]{hyperref}         
           \usepackage[all, cmtip]{xy}

\author{
  \textbf{Chiara Esposito}\thanks{\texttt{chesposito@unisa.it}},
  \\[0.1cm]
  Dipartimento di Matematica\\
  Universit\`a degli Studi di Salerno\\
  via Giovanni Paolo II, 123\\
  84084 Fisciano (SA)\\
  Italy    
  \\[0.3cm]
  \textbf{Alfonso Giuseppe Tortorella}\thanks{\texttt{alfonsogiuseppe.tortorella@kuleuven.be}},
  \\[0.1cm]
  Department of Mathematics\\
  Katholieke Universiteit (KU) Leuven\\
  Celestijnenlaan 200B box 2400\\
  3001 Leuven\\
  Belgium
  \\[0.3cm]
  \textbf{Luca Vitagliano}\thanks{\texttt{lvitagliano@unisa.it }}
  \\[0.1cm]
  Dipartimento di Matematica\\
  Universit\`a degli Studi di Salerno\\
  via Giovanni Paolo II, 123\\
  84084 Fisciano (SA)\\
  Italy    
  \\[0.3cm]
}

\newcommand{\refitem}[1] {\textit{\ref{#1}.)}}

\numberwithin{equation}{section}

\allowdisplaybreaks
\let\originalleft\left
\let\originalright\right
\renewcommand{\left}{\mathopen{}\mathclose\bgroup\originalleft}
\renewcommand{\right}{\aftergroup\egroup\originalright}

\theoremheaderfont{\normalfont\bfseries}
\theorembodyfont{\itshape}
\newtheorem{lemma}{Lemma}[section]
\newtheorem{proposition}[lemma]{Proposition}
\newtheorem{theorem}[lemma]{Theorem}
\newtheorem{corollary}[lemma]{Corollary}

\theorembodyfont{\rmfamily}
\newtheorem{definition}[lemma]{Definition}
\newtheorem{example}[lemma]{Example}
\newtheorem{remark}[lemma]{Remark}

\makeatletter
\def\theorem@checkbold{}
\makeatother

\theoremheaderfont{\scshape}
\theorembodyfont{\normalfont}
\theoremstyle{nonumberplain}
\theoremseparator{:}
\theoremsymbol{\hbox{$\boxempty$}}
\newtheorem{proof}{Proof}

\newenvironment{propositionlist}{\begin{compactenum}[\itshape i.)]}{\end{compactenum}}
\newenvironment{definitionlist}{\begin{compactenum}[\itshape i.)]}{\end{compactenum}}

\newcommand{\D}              {\mathop{}\!\mathrm{d}}

\newcommand{\Fun}[1][k]      {C^{#1}}
\newcommand{\Cinfty}         {\Fun[\infty]}
\newcommand{\Hom}            {\operatorname{\mathsf{Hom}}}
\newcommand{\End}            {\operatorname{\mathsf{End}}}

\newcommand{\der}     {\mathsf{D}}

\title{Infinitesimal Automorphisms of VB-Groupoids and Algebroids}

\date{ }

\begin{document}

\setlength{\leftmargini}{4em}

%
%

\maketitle

%
%

\begin{abstract}
  VB-groupoids and algebroids are \emph{vector bundle objects} in the
  categories of Lie groupoids and Lie algebroids respectively, and
  they are related via the \emph{Lie functor}.  VB-groupoids and
  algebroids play a prominent role in Poisson and related
  geometries. Additionally, they can be seen as models for vector
  bundles over singular spaces.  In this paper we study their
  infinitesimal automorphisms, i.e.~vector fields on them generating a
  flow by diffeomorphisms preserving both the linear and the
  groupoid/algebroid structures. For a special class of
  VB-groupoids/algebroids coming from representations of Lie
  groupoids/algebroids, we prove that infinitesimal automorphisms are
  the same as multiplicative sections of a certain derivation
  VB-groupoid/algebroid.
\end{abstract}

%
%
\newpage
\tableofcontents
\newpage

%
%

\section*{Introduction}
\label{sec:Introduction}

This is the first in a series of papers aiming at studying
(multi)differential operators (DO) and jets on Lie groupoids.  Here we
study a class of first order DOs, namely \emph{derivations of vector
  bundles} (also known as \emph{covariant differential
  operators}). Derivations can be conveniently regarded as
infinitesimal automorphisms of the vector bundle structure. Thus, in
the realm of Lie groupoids, it is natural to look at infinitesimal
automorphisms of VB-groupoids and their infinitesimal counterparts,
VB-algebroids.

VB-groupoids are \emph{groupoid objects in the category of vector
  bundles} or, equivalently, \emph{vector bundle objects in the
  category of Lie groupoids} \cite{Mackenzie2005, bursztyn2016,
  Gracia-Saz2010}.  They play an important role in Poisson geometry
and, more generally, in Lie groupoid theory. The tangent and the
cotangent bundle of a Lie groupoid provide canonical examples. More
generally, VB-groupoids can be seen as intrinsic models of two-term
representations up to homotopy of Lie groupoids
\cite{Gracia-Saz2010}. For instance, the whole information about the
\emph{adjoint representation} (\emph{up to homotopy}) of a Lie
groupoid \cite{CA2013} is encoded by its tangent
VB-groupoid. Accordingly, deformations of a Lie groupoid $G$ are
controlled by the \emph{VB-groupoid complex} of $T G$
\cite{crainic2015deformations}.

The infinitesimal counterparts of VB-groupoids are VB-algebroids,
i.e.~\emph{Lie algebroid objects in the category of vector bundles}
or, equivalently, \emph{vector bundle objects in the category of Lie
  algebroids} \cite{Mackenzie2005, bursztyn2016, gracia2010lie,
  Drummond2015}.  The tangent and the cotangent bundle of a Lie
algebroid provide canonical examples.  Similarly to VB-groupoids,
VB-algebroids can be seen as intrinsic models of two-term
representations up to homotopy of Lie algebroids \cite{gracia2010lie}.
The \emph{adjoint representation} (\emph{up to homotopy}) of a Lie
algebroid is encoded by its tangent VB-algebroid and deformations of a
Lie algebroid $A$ are controlled by the \emph{VB-algebroid complex} of
$TA$ \cite{Crainic2008}. Additionally, the deformation complex of a
Lie groupoid $G$ and the deformation complex of its Lie algebroid $A$
are intertwined by a Van Est map \cite{WX1991, C2003} (see also
\cite{CD2017, BGV2017} for \emph{homogeneous versions} of the Van Est
theorem), which is yet another manifestation of the action of the
\emph{Lie functor} \cite{crainic2015deformations}.

Because of their importance in Lie groupoid/algebroid theory, it is
natural to look at infinitesimal automorphisms of VB-groupoids and
VB-algebroids and at how the Lie functor intertwines them. Notice that
infinitesimal automorphisms of a Lie groupoid are the same as
\emph{multiplicative vector fields} on it
\cite{Mackenzie1997}. Similarly, infinitesimal automorphisms of a
VB-groupoid are (equivalent to) its \emph{multiplicative
  derivations}. Multiplicative vector fields on Lie groupoids are
important objects. In particular, their theory includes the classical
liftings of a vector field on the tangent and the cotangent bundle,
which are of general interest in differential geometry
\cite{Mackenzie1997}. We expect multiplicative derivations to have
similar importance. In this respect, this paper is also part of the
research line aiming at studying \emph{multiplicative structures} on
Lie groupoids and their infinitesimal counterparts (see \cite{KS2015}
for a recent review), i.e.~those geometric data which are compatible
with the groupoid, resp.~algebroid, structure.

The study of multiplicative structures was originally motivated by
Poisson-Lie groups \cite{Drinfeld1983} and symplectic groupoids
\cite{CDW1987, W1987, Z1990, Z1990b} and was pursued by several
authors. Recent research has concentrated on multiplicative structures
of tensorial type: multiplicative differential forms
\cite{bursztyn2012multiplicative} and multivectors
\cite{Iglesias-Ponte2012}, multiplicative (1,1) tensors \cite{SX2007},
multiplicative generalized complex structures \cite{JSX2016} (see also
\cite{C2011, BG2016}), and, encompassing all the previous ones,
generic multiplicative tensors \cite{BD2017}, with few exceptions:
multiplicative foliations \cite{J2012, JO2014}, multiplicative Dirac
structures \cite{O2008, O2013, J2014}, multiplicative vector-bundle
valued differential forms \cite{CSS2015, DE2019}.  Non-tensorial
multiplicative structures (other than the few examples listed above)
have not been considered yet.  In particular, a systematic
investigation (including a definition) of multiplicative (multi)DOs,
multiplicative jets, and their infinitesimal counterparts, is still
unavailable. 

There are two additional motivations to study multiplicative
(multi)DOs and jets.  One comes from Jacobi geometry and is
represented, on one hand, by Jacobi-Lie groups \cite{IM2003}, a wide
generalization of Poisson-Lie groups and, on another hand, contact
groupoids (i.e.~Lie groupoids equipped with a multiplicative contact
structure \cite{KS1993, L1993, D1995}). A Jacobi structure on a
manifold $M$ is a line bundle $L \to M$ together with a Lie bracket on
its sections which is also a first order biDO (see, e.g.,
\cite{M1991}). Jacobi structures encompass Poisson, contact, and
locally conformal symplectic structures as particular cases. The case
when the line bundle $L \to M$ is equipped with a global
trivialization has often been an object of study.  In particular
Iglesias and Marrero adopt this approach to their definition of a
Jacobi groupoid, i.e.~a Lie groupoid with a compatible Jacobi
structure \cite{IM2001, IM2003a}. However, probably, this is not the
best choice from a conceptual point of view (see, e.g.,
\cite{BGG2015}). Not to mention that there are important Jacobi
structures whose underlying line bundle is not at all trivial. In
\cite{BGG2015} the authors study Jacobi manifolds intrinsically,
viewing them as homogeneous Poisson manifolds, in particular they
define Jacobi groupoids as \emph{homogeneous Poisson groupoids}. There
is an alternative, more algebraic approach, requiring the technology
of derivations of vector bundles (see, e.g.~\cite{LOTV14,
  vitagliano2015dirac}). In particular, in order to apply the latter
approach to a Jacobi groupoid, one needs to understand multiplicative
(multi)derivations of VB-groupoids first. Here, we begin this program
by studying multiplicative derivations of a
VB-groupoid. Multiplicative multiderivations will be addressed in a
subsequent paper \cite{ETV20XX} (see also \cite[Section 2.7]{VW2017}).

Our second motivation comes from the geometric theory of PDEs.
Geometrically, a PDE is a submanifold in a jet space. Symmetries
provide important geometric invariants of a PDE. Symmetries of a PDE
form a Lie pseudogroup and the latter can be seen as a Lie subgroupoid
of a jet groupoid. There are several natural (multi)DOs and jets
attached to a jet space, in particular a jet groupoid \cite{Costa2015,
  S2013, CSS2015, Y2016}. One should expect that natural DOs/jets on
jet groupoids are also multiplicative, so falling into the class our
program aims to define and study. Additionally, zero loci of
multiplicative DOs will provide an important class of PDEs on sections
of vector bundles over Lie groupoids.

Finally, it is worth to mention that it has been recently shown that
VB-groupoids can be seen as linear atlases on vector bundles over
differentiable stacks \cite{dHO2016}, so paving the way for a theory
of vector bundles over singular spaces such as orbifolds, leaf spaces
of foliations, orbit spaces of group actions, etc. This paper can then
be seen as a study of \emph{derivations of vector bundles over
  stacks}.

The paper is divided into four main sections. Section
\ref{sec:VB-grpd_alg} is a short review of VB-groupoids and
VB-algebroids. Besides the book by Mackenzie \cite{Mackenzie2005}, our
main references for this material are \cite{bursztyn2016,
  gracia2010lie, Gracia-Saz2010}.
In Section \ref{sec:inf_aut} we define and begin our study of
\emph{multiplicative derivations}. Let $\mathcal W \rightrightarrows E$ be
a VB-groupoid over a Lie groupoid $G \rightrightarrows M$, and let $W
\Rightarrow E$ be a VB-algebroid over a Lie algebroid $A \Rightarrow
M$. We also write $(\mathcal W \rightrightarrows E; G \rightrightarrows
M)$ and $(W \Rightarrow E; A \Rightarrow M)$ for these data. We say
that a derivation of $\mathcal W \to G$ is \emph{multiplicative} if it
generates a flow by VB-groupoid automorphisms. Similarly, a derivation
of $W \to A$ is \emph{infinitesimally multiplicative} (\emph{IM}) if
it generates a flow by VB-algebroid automorphisms. It immediately
follows that if the Lie functor maps $(\mathcal W \rightrightarrows E; G
\rightrightarrows M)$ to $(W \Rightarrow E; A \Rightarrow M)$, then
it also maps multiplicative derivations of $\mathcal W \to G$ to IM
derivations of $W \to A$. Additionally, if $G \rightrightarrows M$ is
source simply connected (hence $\mathcal W \rightrightarrows E$ is source
simply connected \cite{bursztyn2016}) then the Lie functor establishes
a one-to-one correspondence between multiplicative derivations of
$\mathcal W \to G$ and IM derivations of $W \to A$.  Our main result in
this section is an algebraic description of IM derivations of a
VB-algebroid $(W \Rightarrow E; A \Rightarrow M)$ (Theorem
\ref{theor:delta})
which parallels a result by Mackenzie and Xu in \cite{Mackenzie1997}
stating that there is a one-to-one correspondence between IM vector
fields on a Lie algebroid and Lie algebroid derivations. We also prove
that multiplicative derivations (resp.~IM derivations) are
$1$-cocycles in a certain \emph{linear} subcomplex of the
\emph{deformation complex} of the Lie groupoid $\mathcal W$ (resp.~Lie
algebroid $W$) (Proposition \ref{prop:1-cocycles_grpd},
resp.~Corollary \ref{cor:1-cocycles_alg}).

%
Derivations of a vector bundle $E \to M$ are sections of the
\emph{gauge algebroid} $\der E \to M$ (see Section
\ref{sec:derivations}). Unfortunately, unlike the tangent bundle of a
groupoid (resp.~algebroid), the gauge algebroid of a VB-groupoid
(resp.~algebroid) is not a groupoid itself in general. Hence, unlike
multiplicative vector fields \cite{Mackenzie1997}, multiplicative
(resp.~IM) derivations of a VB-groupoid $(\mathcal W \rightrightarrows
E; G \rightrightarrows M)$ (resp.~algebroid $(W \Rightarrow E; A
\Rightarrow M)$) can not be characterized, in general, as groupoid
(resp.~algebroid) morphisms $G \to \der \mathcal W$ (resp.~$A \to
\der W$). In Section~\ref{sec:trivial-core}, we isolate a particular
class of VB-groupoids (resp.~algebroids) for which the associated
gauge algebroid is a VB-groupoid (resp.~algebroid) itself, namely
either \emph{trivial-core} or \emph{full-core VB-groupoids}
(resp.~\emph{algebroids}) and prove that in this case
multiplicative (resp.~IM) derivations can be characterized as morphic
sections of the \emph{gauge VB-groupoid} (resp.~\emph{algebroid})
(Theorems \ref{theor:der_VB-grpd}, \ref{theor:der_VB-alg}).
This parallels similar results for multiplicative (resp.~IM) vector
fields and the tangent VB-groupoid (resp.~algebroid)
\cite{Mackenzie1997}.
Finally, Section~\ref{sec:applications} contains an application 
of the results of Section~\ref{sec:trivial-core} to representations up to homotopy.

\section{Preliminaries}
\label{sec:VB-grpd_alg}

\subsection{VB-groupoids}

In this section, we recall basic aspects of the theory of VB-groupoids
and VB-algebroids. We refer to \cite{bursztyn2016, Mackenzie2005,
  gracia2010lie, Gracia-Saz2010} for more details. Given a Lie
groupoid $G \rightrightarrows M$, we usually denote by $s, t : G \to
M$ the source and target respectively, by $u : M \to G$ the unit, by
$G^{(2)} = G \mathbin{{}_{s} \times_t} G \subset G \times G$ the
submanifold of composable arrows, by $m : G^{(2)} \to G$ the
multiplication, and by $i : G \to G$ the inversion.  Alternatively, we
denote the unit by $x \mapsto 1_x$, the multiplication by $(g,h)
\mapsto gh$, and the inversion by $g \mapsto g^{-1}$.  For $g \in G$,
we write $g : x_1 \to x_2$ to mean that $s(g) = x_1$ and $t(g) =
x_2$. Finally, we identify $M$ with its image under $u$.  Consider a
diagram
\begin{equation}
  \begin{gathered}
    \label{diag:VB_groupoid}
  \begin{tikzcd}
    \mathcal W \arrow[d, swap, "q_{\mathcal W}"]\arrow[r, shift left] \arrow[shift right]{r}& E\arrow[d, "q"]
    \\
    G\arrow[r, shift left] \arrow[shift right]{r}& 
    M
  \end{tikzcd}
    \end{gathered}
\end{equation}
where the rows are Lie groupoids and the columns are vector bundles.
We denote by $\tilde s, \tilde t, \tilde u, \tilde m, \tilde i$ the
structure maps of the top row (source, target, unit, multiplication,
inversion respectively) and by $s, t, u, m, i$ the structure maps of
the bottom row. Additionally, we denote by $+_\mathcal W : \mathcal W
\times_G \mathcal W \to \mathcal W$ and $+_E : E \times_M E \to E$ the
additions, by $\lambda \cdot_\mathcal W : \mathcal W \to \mathcal W$
and $\lambda \cdot_E : E \to E$ the scalar multiplications by $\lambda
\in \mathbb R$, by $0_\mathcal W : G \to \mathcal W$ and $0_E : M \to
E$ the zero sections.  Notice that throughout this paper, for any
vector bundle $q:V\to N$, we will denote by $s_x$ the value of a
section $s\in\Gamma(V)$ at a point $x\in N$.
\begin{definition}[VB-groupoid] 
  \label{def:VBGroupoid}
  A \emph{VB-groupoid} is a \emph{vector bundle object in the category of Lie
  groupoids}, i.e.~a diagram like \eqref{diag:VB_groupoid} such that
  $(q_\mathcal W, q)$ and $(+_\mathcal W, +_E)$ are Lie groupoid morphisms.
  Equivalently, a VB-groupoid is a \emph{Lie groupoid object in the category
  of vector bundles}, i.e.~a diagram \eqref{diag:VB_groupoid} such that
  $(\tilde s, s)$, $(\tilde t, t)$ and $(\tilde m, m)$ are 
  vector bundle morphisms.
\end{definition} 
It follows from the definition, that $(\lambda \cdot_\mathcal W,
\lambda \cdot_E)$ and $( 0_\mathcal W, 0_E)$ are also Lie groupoid
morphisms, for all $\lambda \in \mathbb R$ and that $(\tilde u, u)$
and $(\tilde i, i)$ are also vector bundle morphisms.
When there is no risk of ambiguity, we simply denote by $(\mathcal W
\rightrightarrows E; G \rightrightarrows M)$ (or even just $\mathcal W$) a
VB-groupoid, thus understanding all structure maps.  
\begin{definition}[VB-groupoid morphism]
  \label{def:VBGroupoidMorph}
  Let $(\mathcal W_1 \rightrightarrows E_1; G_1\rightrightarrows M_1)$
  and $(\mathcal W_2\rightrightarrows E_2; G_2\rightrightarrows M_2)$
  be VB-groupoids.  A \emph{morphism of VB-groupoids}
  \begin{equation}
    \label{eq:morphism_VB_groupoids}
    (F_\mathcal W, \tilde F; F_G, F) : (\mathcal W_1\rightrightarrows E_1; G_1\rightrightarrows M_1) \to
    (\mathcal W_2\rightrightarrows E_2; G_2\rightrightarrows M_2)
  \end{equation}
  is a (commutative) diagram
  \begin{equation}
    \label{diagram:morph_VB_fibr}
    \begin{gathered}
    \begin{tikzcd}
      & \mathcal W_1 
      \arrow[rr, "F_\mathcal W", pos=0.7]
      \arrow[ld, start
        anchor={[xshift=.5ex, yshift=.5ex]}, shift left]
      \arrow[ld,
        start anchor={[xshift=.5ex, yshift=.5ex]}, shift right]
      \arrow[dd]& 
      & \mathcal W_2
      \arrow[dd]
      \arrow[ld, start
        anchor={[xshift=.5ex, yshift=.5ex]}, shift left]
      \arrow[ld,
        start anchor={[xshift=.5ex, yshift=.5ex]}, shift right]
      \\ E_1
      \arrow[rr, crossing over, "\tilde{F}",
        pos=0.7]
      \arrow[dd]&&E_2& \\ & G_1
      \arrow[rr, "F_G",
        pos=0.7]
      \arrow[ld, start anchor={[xshift=.5ex, yshift=.5ex]},
        shift left]
      \arrow[ld, start anchor={[xshift=.5ex,
            yshift=.5ex]}, shift right]&&G_2
      \arrow[ld, start
        anchor={[xshift=.5ex, yshift=.5ex]}, shift left]
      \arrow[ld,
        start anchor={[xshift=.5ex, yshift=.5ex]}, shift right] 
      \\ M_1
      \ar[rr, "F", pos=0.7]&&M_2
      \arrow[from=uu, crossing over]&
    \end{tikzcd}
    \end{gathered}
  \end{equation}
  such that the following conditions are satisfied:
  \begin{definitionlist}
  \item $(F_\mathcal W, \tilde F)$ and $(F_G,
    F)$ are morphisms of groupoids,
  \item $(F_\mathcal W, F_G)$ and $(\tilde F,
    F)$ are morphisms of vector bundles.
  \end{definitionlist}
\end{definition}
Notice that the map $F_\mathcal W$ in \eqref{eq:morphism_VB_groupoids}
determines the morphism completely. Accordingly, we also say that
$F_\mathcal W : \mathcal W_1 \to \mathcal W_2$ is a morphism of
VB-groupoids.

The (\emph{right}) \emph{core} of a VB-groupoid $(\mathcal
W\rightrightarrows E; G\rightrightarrows M)$ is the vector bundle $C
\to M$ defined by setting
\begin{equation}
  C 
  := 
  \{ \omega \in \mathcal W : \tilde s (\omega) = 0, \text{ and } q_\mathcal W (\omega) = 1 \}.
\end{equation}
In other words, $C = u^\ast (\ker \tilde s) = \ker \tilde s|_M$ is the
pull-back of the vector subbundle $\ker \tilde s \subset \mathcal W$
along the unit. There is a similar notion of left core where the
target is in place of the source. It turns out that the right and the
left cores are canonically isomorphic as vector bundles. Hence they
are manifestations of a unique vector bundle which can be safely
referred to as the \emph{core}.
  %

The dual vector bundle $\mathcal W^\ast$ is a Lie groupoid over
$C^\ast$, and, actually, $(\mathcal W^\ast\rightrightarrows C^\ast;
G\rightrightarrows M)$ is a VB-groupoid, the \emph{dual VB-groupoid}
of $(\mathcal W \rightrightarrows E; G \rightrightarrows M)$.
Additionally the VB-groupoids $\mathcal W$ and $\mathcal W^{\ast\ast}$
are canonically isomorphic. We refer to \cite{Mackenzie2005} for more
details on this construction. Here we only notice that given a
VB-groupoid isomorphism $F_\mathcal W : \mathcal W_1 \to \mathcal
W_2$, its transpose $F_\mathcal W^\ast : \mathcal W_2^\ast \to
\mathcal W_1^\ast$ is a VB-groupoid isomorphism as well. In
particular, VB-groupoid automorphisms $\mathcal W \to \mathcal W$ are
equivalent to VB-groupoid automorphisms $\mathcal W^\ast \to \mathcal
W^\ast$.

\begin{example}[Tangent VB-groupoid]
  \label{ex:tangent_VB_grpd}
  Let $G \rightrightarrows M$ be a Lie groupoid with Lie algebroid
  $A$. Then $(T G, TM; G, M)$ is a VB-groupoid, called the
  \emph{tangent VB-groupoid to $G$}, in a natural way: the structure
  maps are the tangent maps to the structure maps of $G$. The core of
  $(T G\rightrightarrows TM; G\rightrightarrows M)$ is $A$.
\end{example}

\subsection{VB-algebroids}

We now turn to the infinitesimal counterpart of VB-groupoids, the
so-called VB-algebroids.  First of all, given a Lie algebroid $A
\Rightarrow M$, we usually denote by $\rho : A \to TM$ the anchor and
by $[-,-] : \Gamma (A) \times \Gamma (A) \to \Gamma (A)$ the Lie
bracket.
Consider a diagram
\begin{equation}
  \label{diag:VB_algebroid}
  \begin{gathered}
    \begin{tikzcd}
    W \arrow[d, swap, "q_W"]\arrow[r, "\tilde p"]&E \arrow[d, "q"] \\
    A \arrow[r, "p"]& M
  \end{tikzcd}
  \end{gathered}
\end{equation}
where the rows and the columns are vector bundles. We denote by
$+_{W,A} : W \times_A W \to W$, $+_{W,E} : W \times_E W \to W$ the
additions, by $\lambda \cdot_{W,A} : W \to W$ and $\lambda \cdot_{W,E}
: W \to W$ the scalar multiplications by $\lambda \in \mathbb R$ (in
the fibers over $A$ and $E$ respectively), and by $0_{W,A} : A \to W$
and $0_{W,E} : E \to W$ the zero sections.
If the maps $(q_W, q)$ and $(+_{W,A}, +_E)$ are vector bundle
morphisms, or, equivalently $(\tilde p, p)$ and $(+_{W,E}, +_A)$ are
vector bundle morphisms, then \eqref{diag:VB_algebroid} is a
\emph{double vector bundle} \cite[Chapter 9]{Mackenzie2005}, denoted
$(W \to E; A \to M)$, and it follows that $(\lambda \cdot_{W,A},
\lambda \cdot_E)$, $( 0_{W,A}, 0_E)$, $(\lambda \cdot_{W,E}, \lambda
\cdot_A)$ and $(0_{W,E}, 0_A)$ are also vector bundle morphisms.
\begin{remark}[Core and linear sections of a double vector bundle] 
  \label{rem:linear_sections}
  Recall that the core of a double vector bundle
  \eqref{diag:VB_algebroid} is the vector bundle
  \begin{equation}
    C 
    := 
    \ker \tilde p \cap \ker q_W \to M.
  \end{equation} 
  Vector bundle $W \to E$ has two distinguished classes of sections:
  \begin{definitionlist}
  \item \emph{core sections}: for every section $\chi$ of the core $C \to
    M$, there is a section $\widehat \chi$ of $W \to E$ defined by:
    \begin{equation}
      \widehat \chi_e 
      := 
      (0_{W, E})_e +_{W,A} \chi_{q(e)}
    \end{equation}
    for all $e \in E$. 
  \item \emph{linear sections}: they are sections $\tilde a$ of 
    $W \to E$ such that $(\tilde a, a)$ is a vector bundle
    morphism for a (necessarily unique) section $a$ of $A$.
  \end{definitionlist}
  We denote by $\Gamma (W, E)$ the space of all
  sections of $W \to E$, by $\Gamma_c (W, E)$ the subspace of core sections and by
  $\Gamma_\ell (W, E)$ the subspace of linear sections.
  Linear and core sections generate $\Gamma (W, E)$ as a $C^\infty
  (E)$-module.  Actually, linear sections alone generate $\Gamma (W,
  E)$ away from the (image of) the zero section of $E$, i.e.,
  for every point $e \in E \smallsetminus 0_E (M)$, there exists an open neighborhood $\mathcal U$
  of $e$ such that the $C^\infty (\mathcal U)$-module of local sections of $W \to E$ defined on $\mathcal U$ is spanned 
  by linear sections.  This is easily  seen in local coordinates adapted to the double vector bundle  structure.
\end{remark}
From now on, we also assume that the rows of
\eqref{diag:VB_algebroid} possess Lie algebroid structures $W
\Rightarrow E$, and $A \Rightarrow M$. We denote by $\tilde\rho :
W \to TE$ the anchor.
\begin{definition}[VB-algebroid] 
  \label{def:VBAlgebroid}
  A \emph{VB-algebroid} is a \emph{vector bundle object in the
    category of Lie algebroids}, i.e.~a diagram like
  \eqref{diag:VB_algebroid} such that $(q_W, q)$ and $(+_{W,A}, +_E)$
  are Lie algebroid morphisms.
  Equivalently, a VB-algebroid is a \emph{Lie algebroid object in the
    category of vector bundles}, i.e.~a diagram
  \eqref{diag:VB_algebroid} such that $(\tilde p, p)$, $(+_{W,E},
  +_A)$ and $(\tilde\rho, \rho)$ are morphisms of vector bundles, and
  the Lie bracket in $\Gamma(W, E)$ satisfies the following
  \emph{linearity conditions}:
  \begin{definitionlist}
  \item $[\Gamma_\ell(W, E), \Gamma_\ell (W, E)] \subset \Gamma_\ell
    (W, E)$,
  \item $[\Gamma_\ell(W, E), \Gamma_c (W, E)] \subset \Gamma_c (W,
    E)$,
  \item $[\Gamma_c(W, E), \Gamma_c (W, E)] = 0$.
  \end{definitionlist}
  %
\end{definition}
We denote simply by $(W \Rightarrow E; A \Rightarrow M)$ a
VB-algebroid and when there is no risk of ambiguity we only use
$W$.
\begin{definition}[VB-algebroid morphism]
  \label{def:VBAlgebroidMorph}
  Let $(W_1\Rightarrow E_1; A_1\Rightarrow M_1)$ and 
  $(W_2\Rightarrow E_2; A_2\Rightarrow M_2)$ be VB-algebroids.
  A \emph{morphism of VB-algebroids}
  \begin{equation}\label{eq:morphism_VB_algebroids}
    (F_W, \tilde F; F_A, F) 
    : 
    (W_1\Rightarrow E_1; A_1\Rightarrow M_1) \to (W_2\Rightarrow E_2; A_2\Rightarrow M_2)
  \end{equation}
  is a (commutative) diagram
  \begin{equation}
    \label{diagram:morph_VB_algbd}
    \begin{gathered}
        \begin{tikzcd}
      &W_1\arrow[rr, "F_W", pos=.7]\arrow[ld, Rightarrow]\arrow[dd]&&W_2\arrow[dd]\arrow[dl, Rightarrow]\\
      E_1 \arrow[rr, "\tilde{F}", pos=.7, crossing over]\arrow[dd]& & E_2&\\
      &A_1\arrow[rr, "F_A", pos=.7] \arrow[dl, Rightarrow]& & A_2\arrow[dl, Rightarrow]\\
      M_1\arrow[rr, "F", pos=.7]&&M_2\arrow[from=uu, crossing over]&
    \end{tikzcd}
    \end{gathered}
  \end{equation} 
  such that the following conditions are satisfied:
  \begin{definitionlist}
  \item $(F_W, \tilde F)$ and $(F_A,
    F)$ are morphisms of Lie algebroids
  \item $(F_W, F_A)$ and $(\tilde F,
    F)$ are morphisms of vector bundles.
  \end{definitionlist}
\end{definition}
The map $F_W$ in \eqref{diagram:morph_VB_algbd} determines the
morphism completely. Accordingly, we sometimes say that $F_W : W_1 \to
W_2$ is a morphism of VB-algebroids.
%


Similarly as VB-groupoids, VB-algebroids can be dualized.  Let $(W
\Rightarrow E; A \Rightarrow M)$ be a VB-algebroid with core $C \to
M$. Then the dual $W^\ast$ of the vector bundle $W \to A$ is a
VB-algebroid and $(W^\ast \Rightarrow C^\ast; A \Rightarrow M)$ is a
VB-algebroid, called the \emph{dual} VB-algebroid of $(W \Rightarrow
E; A \Rightarrow M)$. Additionally the VB-algebroids $W$ and
$W^{\ast\ast}$ are canonically isomorphic.  We refer to
\cite{Mackenzie2005} again for more details. Here we only notice that
similarly as for VB-groupoids, given a VB-algebroid isomorphism $F_W :
W_1 \to W_2$, its transpose $F_W^\ast : W_2^\ast \to W_1^\ast$ is a
VB-algebroid isomorphism as well. Hence VB-algebroid automorphisms $W
\to W$ are equivalent to VB-algebroid automorphisms $W^\ast \to
W^\ast$.
  
\begin{example}\label{ex:tangent_VB}
  Let $E \to M$ be a vector bundle. Then $(TE\Rightarrow E;
  TM\Rightarrow M)$ is a VB-algebroid in the obvious way.
\end{example}

\begin{example}[Tangent VB-algebroid]\label{ex:tangent}
  Let $A \Rightarrow M$ be a Lie algebroid. Then $(TA\Rightarrow TM;
  A\Rightarrow M)$ is a VB-algebroid, called the \emph{tangent
    VB-algebroid to $A$}, in a natural way. The projection $TA \to TM$
  and the anchor $\tilde\rho : TA \to TTM$ are the tangent map to the
  projection $A \to M$, and the composition of the tangent map to the
  anchor $\rho : A \to TM$, followed by the canonical
  \emph{involution} $TTM \to TTM$ (see, e.g., \cite[Section
    9.6]{Mackenzie2005}). The Lie bracket
  \begin{equation}
    \label{eq:bracket_TA}
          [-,-] : \Gamma (TA ,TM) \times \Gamma(TA ,TM) \to \Gamma (TA,TM)
  \end{equation}
  is uniquely defined by its action on certain distinguished sections,
  as follows.  First notice that $(TA \to TM; A \to M)$ is a double vector
  bundle with core canonically isomorphic to $A$ itself. The
  isomorphism identifies the tangent space to a fiber of $A \to M$ at
  the origin, with the fiber itself. Now, consider the following
  distinguished classes of sections of $TA \to TM$:
  \begin{definitionlist}
  \item core sections,
  \item linear sections of the form $\D a : TM \to TA$ for some $a \in
    \Gamma (A)$.
  \end{definitionlist}
  We denote by $\widehat a_T : TM \to TA$ the core section
  corresponding to $a \in \Gamma (A)$, seen as a section of the
  core. We have
  \begin{equation}
    \label{eq:core_TA}
    (\widehat a_T)_v = (\D 0_A )(v) + a^{\uparrow}_{(0_A)_x}
  \end{equation}
  for all $v \in T_x M$, $x \in M$, where $a^{\uparrow} \in \mathfrak
  X (A)$ is the \emph{vertical lift of $a$} (see also Remark
  \ref{rem:linear_VF}).  Then bracket \eqref{eq:bracket_TA} is
  completely determined by
  \begin{equation}
    \label{eq:bracket_TA_on_sections}
    \begin{aligned}
      [\D a, \D b] & = \D [a,b], \\
      [\D a, \widehat b_T] & = \widehat{[a,b]}_T, \\
      [\widehat a_T, \widehat b_T] &= 0,
    \end{aligned}
  \end{equation}
  for all $a,b \in \Gamma(A)$. Finally, notice that the
  VB-algebroid structures on $(TA\Rightarrow A; TM\Rightarrow M)$ and
  $(TA\Rightarrow TM; A\Rightarrow M)$ are compatible in the sense
  that they provide an instance of a \emph{double Lie algebroid}
\cite{M1992, M2000}.
\end{example}

\begin{remark}[VB-groupoids/algebroids and homogeneity structures]
  \label{rem:hom_struct}
  There is an alternative definition of a VB-groupoid/algebroid that
  is often useful in practice. Namely, a vector bundle is the same as
  a manifold $E$ equipped with a \emph{regular homogeneity structure},
  i.e.~a \emph{regular} action $h: \mathbb R \times E \to E$, written
  $(\lambda, e) \mapsto h_\lambda (e)$, of the monoid $(\mathbb R,
  \cdot)$. The base manifold is then $M = h_0 (E)$ and the projection
  $E \to M$ is $h_0$.  The infinitesimal generator of
  $h_{\operatorname{exp} \lambda}$ is the \emph{Euler vector field} on
  $E$. We will not explain here the word \emph{regular}: we refer to
  \cite{GR2009} for more details (see also \cite[Definition
    2.1.1]{bursztyn2016}).  Similarly, a VB-groupoid (resp.~algebroid)
  is equivalent to a groupoid $\mathcal W \rightrightarrows E$
  (resp.~algebroid $W \to E$) equipped with a regular action $h:
  \mathbb R \times \mathcal W \to \mathcal W$ (resp.~$h: \mathbb R
  \times W \to W$) of the monoid $(\mathbb R, \cdot)$ with the
  additional property that $h_\lambda : \mathcal W \to \mathcal W$
  (resp.~$h_\lambda : W \to W$) is a Lie groupoid (resp.~algebroid)
  morphism for all $t$.  In particular, a diagram like
  \eqref{diag:VB_groupoid} is a VB-groupoid if and only if $\lambda
  \cdot_\mathcal W : \mathcal W \to \mathcal W$ (multiplication by
  $\lambda$ in the fibers of $\mathcal W \to G$) is a groupoid
  morphism for all $\lambda$, and similarly for VB-algebroids. Notice
  that using this point of view a diagram like
  (\ref{eq:morphism_VB_groupoids})
  (resp.~\eqref{eq:morphism_VB_algebroids}) is a morphism if and only
  if $F_\mathcal W$ (resp.~$F_W$) commutes with the homogeneity
  structures.
\end{remark}

\subsection{Differentiation and integration}\label{sec:diff_int}

We conclude this section discussing briefly \emph{differentiation} and
\emph{integration} issues. We refer to \cite{bursztyn2016} for
details. Denote by $\mathsf{Lie}$ the Lie functor mapping Lie
groupoids to Lie algebroids (and Lie groupoid morphisms to Lie
algebroid morphisms). Let $(\mathcal W \rightrightarrows E; G
\rightrightarrows M)$ be a VB-groupoid. We denote by $W = \mathsf{Lie}
(\mathcal W)$ the Lie algebroid of $\mathcal W$, and by $A =
\mathsf{Lie} (G)$ the Lie algebroid of $G$. Then $(W\Rightarrow E;
A\Rightarrow M)$ is a VB-algebroid in a natural (and obvious) way, and
it is called the \emph{VB-algebroid of $(\mathcal W \rightrightarrows
  E;G\rightrightarrows M)$}.  We also denote the latter by
$\mathsf{Lie}(\mathcal W \rightrightarrows E; G \rightrightarrows
M)$. Additionally,
\begin{definitionlist}
\item $(\mathcal W \rightrightarrows E; G \rightrightarrows M)$ and $(W
  \Rightarrow E; A \Rightarrow M)$ have canonically isomorphic cores
  (we denote both by $C$), and
\item $\mathsf{Lie}(\mathcal W^\ast \rightrightarrows C^\ast;
  G\rightrightarrows M) \cong (W^\ast\Rightarrow C^\ast; A\Rightarrow
  M)$ canonically.
\end{definitionlist}
Similarly, the Lie functor maps morphisms of VB-groupoids to morphisms
of VB-algebroids.
Conversely, let $(W \Rightarrow E; A \Rightarrow M)$ be a
VB-algebroid. If $W$ is integrable and $\mathcal W \rightrightarrows
E$ is its source simply connected integration, then 1) $A$ is
integrable as well, and, denoted by $G \rightrightarrows M$ its source
simply connected integration, 2) $(\mathcal W \rightrightarrows E; G
\rightrightarrows M)$ is a VB-groupoid in a natural way. Finally, if
$(\mathcal W'\rightrightarrows E'; G'\rightrightarrows M')$ is another
VB-groupoid, $(W'\Rightarrow E; A' \Rightarrow M')$ is its
VB-algebroid and $F_W : W \to W'$ is a morphism of VB-algebroids, then
$F_W$ integrates to a unique morphism $F_\mathcal W : \mathcal W \to
\mathcal W'$ of VB-groupoids.
\begin{remark}
  \label{rem:tangent_VB}
  Let $G \rightrightarrows M$ be a Lie groupoid and let $A =
  \mathsf{Lie} (G)$ be its Lie algebroid. Then the VB-algebroid of the
  tangent VB-groupoid to $G$, is (canonically isomorphic to) the
  tangent VB-algebroid to $A$: $\mathsf{Lie}(TG\rightrightarrows TM;
  G\rightrightarrows M) \cong (TA\Rightarrow TM; A\Rightarrow M)$.
\end{remark}

%
%

\section{Infinitesimal automorphisms of VB-groupoids and algebroids}
\label{sec:inf_aut}

\subsection{Derivations}
\label{sec:derivations}

Let $E\to M$ be a vector bundle, and let $x \in M$.  A
\emph{derivation} at $x$ is an $\mathbb R$-linear map $\Delta : \Gamma
(E) \to E_x$ such that
\begin{equation}
  \Delta (fe)
  =
  v(f)e_x + f(x)\Delta(e)
\end{equation}
for all $f \in \Cinfty (M)$, $e \in \Gamma (E)$ and some (necessarily
unique) tangent vector $v \in T_x M$.  We also call $v$ the symbol of
$\Delta$, and denote it by $\sigma(\Delta)$. 

\color{black} It is useful to mention that there is an alternative (more geometric) 
description of derivations at $x$. Let $\varepsilon \mapsto
\gamma (\varepsilon)$ be a smooth curve in $M$ such that $\gamma (0) =
x$ and let $\varepsilon \mapsto \gamma_E (\varepsilon) : E_x \to
E_{\gamma (\varepsilon)}$ be a smooth curve of linear isomorphisms
(covering $\gamma$) such that $\gamma_E (0) =
\operatorname{id}_{E_x}$. The \emph{velocity} of $\gamma_E$ is the
following derivation at $x$
\begin{equation}
  \left.\frac{d}{d\varepsilon}\right|_{\varepsilon = 0} \gamma_E (\varepsilon) 
  : 
  \Gamma (E) \to E_x, 
  \quad 
  e \mapsto \left(\left.\frac{d}{d\varepsilon}\right|_{\varepsilon = 0} \gamma_E (\varepsilon)\right) e 
  := \left.\frac{d}{d\varepsilon}\right|_{\varepsilon = 0}  \gamma_E (\varepsilon)^{-1} e_{\gamma (\varepsilon)}.
\end{equation}
Additionally, the symbol of
$\left.\frac{d}{d\varepsilon}\right|_{\varepsilon = 0} \gamma_E
(\varepsilon)$ is $\left.\frac{d}{d\varepsilon}\right|_{\varepsilon = 0}
\gamma (\varepsilon)$. Every derivation at $x$ arises in this way.
\color{black}

Derivations of $E$ at $x$ form a vector space denoted $\der_x E$.  If
we let $x$ vary, we get a vector bundle $\der E \to M$. Sections of
$\der E$ identify with \emph{derivations} of $E$, i.e.~$\mathbb
R$-linear operators $ \Delta: \Gamma(E) \to \Gamma (E) $ such that
\begin{equation}
  \Delta (fe)
  =
  X(f)e + f\Delta(e)
\end{equation}
for all $f \in \Cinfty (M)$, $e \in \Gamma (E)$, and some (necessarily
unique) vector field $X \in \mathfrak X (M)$.  We call $X$ the
\emph{symbol} of $\Delta$ and denote it by $\sigma(\Delta)$ or
$\sigma_\Delta$. In the literature, derivations of a vector bundle are
referred to with different names, including \emph{der-operators}
\cite{Vinogradov2009} and \emph{covariant differential operators}
\cite{Mackenzie2005}. From now on we identify sections of $\der E\to
M$ and derivations of $E$ in the obvious way. In particular, we denote
by $\Gamma(\der E)$ the space of derivations of $E$.

The vector bundle $\der E$ carries a structure of Lie algebroid: the
\emph{Atiyah} or \emph{gauge algebroid of $E$}. The Lie bracket of two
derivations is the commutator and the anchor is the symbol $\sigma :
\der E \to TM$, $\Delta \mapsto \sigma(\Delta)$.  Notice that a
\emph{representation} of a Lie algebroid $A \to M$ can be regarded as
a vector bundle $E \to M$ equipped with a Lie algebroid morphism $A
\to \der E$.
The kernel of the symbol $\sigma : \der E \to TM$ is the endomorphism
bundle of $E$. Accordingly, there is a short exact
sequence of vector bundles:
\begin{equation}
  0 
  \longrightarrow 
  \End E 
  \longrightarrow 
  \der E 
  \overset{\sigma}{\longrightarrow} 
  TM 
  \longrightarrow 
  0.
\end{equation}

The correspondence $E \mapsto \der E$ is functorial.  To see this,
first recall that a vector bundle morphism is called \emph{regular} if
it is an isomorphism on fibers.  A regular vector bundle morphism
$\phi_E : E_N \to E$ covering a smooth map $\phi : N \to M$ allows us
to \emph{pull-back} sections of $E$ to sections of
$E_N$. Specifically, let $e \in \Gamma(E)$.  Its pull-back along
$\phi_E$ is the section $\phi_E^\ast e \in \Gamma(E_N)$ defined by
\begin{equation}
  (\phi_E^\ast e)_x 
  = 
  \phi_E |_{(E_N)_x}^{-1} (e_{\phi (x)}),
\end{equation}
for all $x\in N$.  Notice that we can use a regular vector bundle
morphism $\phi_E : E_N \to E$ covering $\phi : N \to M$ to identify
$E_N$ with the pull-back bundle $\phi^\ast E = N \times_M E$ in a
canonical way. In the following, we always make use of this
identification. In particular, we denote $\phi^\ast E = E_N$, and
denote by $\phi_E : \phi^\ast E \to E$ the canonical map. Sometimes,
if there is no risk of confusion, we also denote simply by $\phi^\ast
e$ the pull-back section $\phi^\ast_E e$. Now, let $E \to M$ and $E_N
\to N$ be vector bundles, and let $\phi_E : E_N \to E$ be a
\emph{regular} vector bundle morphism, covering the smooth map $\phi :
N \to M$.  Morphism $\phi_E$ determines a (generically non regular)
vector bundle morphism
\begin{equation}
  \der \phi_E 
  : 
  \der E_N  \to \der E,
\end{equation}
covering $\phi : N \to M$ by
\begin{equation}
  \der \phi_E (\Delta) (e) 
  = 
  (\phi_E \circ \Delta) (\phi^\ast e)
\end{equation}
for all $\Delta \in \der E_N $ and $e \in
\Gamma(E)$. 
\color{black}
If 
\begin{equation}
\Delta = \left.\frac{d}{d\varepsilon}\right|_{\varepsilon = 0} \gamma_{E} (\varepsilon)
\end{equation}
for some curve $\varepsilon \mapsto \gamma_E (\varepsilon)$ of fiber
isomorphisms, then
\begin{equation}\label{eq:der_phi_vel}
  \der \phi_E (\Delta)
  = 
  \left.\frac{d}{d\varepsilon}\right|_{\varepsilon = 0} \phi_E \circ \gamma_{E} (\varepsilon) \circ \phi_E|_{(E_N)_x}^{-1},
\end{equation}
where $x$ is the base point of $\Delta$, i.e.~$\Delta \in \der_x E_N$.
\color{black}

Clearly, diagram
\begin{equation}
  \label{diag:seq_morph_DO}
  \begin{gathered}
       \begin{tikzcd}
      0 \arrow[r] &\End E_N\arrow[r]\ar[d, "\End \phi"] & \der E_N\arrow[r, "\sigma"]\arrow[d, "\der \phi_E"]	& TN\arrow[d, "d\phi"]\arrow[r] &0\\
      0\arrow[r] 
      & \End E\arrow[r]  
      & \der E \arrow[r, "\sigma"]  
      & TM \arrow[r]& 0
    \end{tikzcd}
  \end{gathered}
\end{equation}
commutes. Here $\End \phi = \phi_{\End E} : \End E \to \End E$ is the
obvious regular vector bundle morphism (covering $\phi$). Even more,
$\der \phi_E : \der E_N \to \der E$ is actually a morphism of Lie
algebroids (covering $\phi$).  We will often denote $\der \phi_E$
simply by $\der \phi $, if there is no risk of confusion.
 
It is easy to see that the rightmost square in the
diagram (\ref{diag:seq_morph_DO}) is a pull-back diagram. In
particular, $\der E_N$ is canonically isomorphic to a pull-back
bundle:
  \begin{equation}
  \der E_N \cong TN \mathbin{{}_{d\phi}\hspace*{-0.1cm}  \hspace*{-0.1cm}\times_\sigma}  \der E.
  \end{equation}
More precisely, $\der E_N$ is the pull-back Lie algebroid of $\der E$
along the smooth map $\phi : N \to M$ (which makes sense, because, as
the anchor $\sigma : \der E \to TM$ is surjective, $\phi$ is
automatically transverse to it).\color{black}

The case when $\phi : N \to M$ is the inclusion of a submanifold
is of a special interest.  In this case $E_N = E|_N$ and $\der \phi :
\der E_N \to \der E$ is an embedding whose image consists of
derivations of $E$ with symbol tangent to $N$.

There is another way to understand derivations of a vector bundle $p:
E \to M$.  They can be regarded as \emph{infinitesimal automorphisms}
of $E$.  A vector field $\tilde X$ on the total space $E$ of the
vector bundle $E \to M$ is an infinitesimal automorphism if it
generates a flow by vector bundle automorphisms.  In particular, if
$\tilde X$ is an infinitesimal automorphism of $E$, then it projects
onto a vector field $p_\ast \tilde X$ on $M$.
\begin{remark}
  \label{rem:linear_VF}
  There are several characterizations of infinitesimal automorphisms of
  the vector bundle $E \to M$ that we now discuss. First of all, let
  $h : \mathbb R \times E \to E$ be the (regular) homogeneity
  structure encoding the vector bundle structure (see
  Remark~\ref{rem:hom_struct}), and let $\mathcal E_E$ be the
  infinitesimal generator of $h_{\operatorname{exp} \lambda}$, i.e.~the
  Euler vector field on $E$. Take $\tilde X$ a vector field on
  $E$. The following conditions are equivalent:
  \begin{enumerate}
  \item $\tilde X$ is an infinitesimal automorphism;
  \item\label{ii} for every fiber-wise linear function $\varphi$ on $E$, the Lie
    derivative $\tilde X (\varphi)$ is fiber-wise linear;
  \item\label{iii} for every fiber-wise constant vector field $Y$ on
    $E$, the commutator $[\tilde X, Y]$ is fiber-wise constant;
  \item $(h_\lambda)_\ast X = X$ for all $\lambda \neq 0$;
  \item $[\mathcal E_E, X] = 0$;
  \item when regarded as a section of $TE \to E$, $\tilde X$ is a
    \emph{linear section} of (the top row) of the double
    vector bundle $(TE \to E; TM \to M)$, i.e.~$\tilde X \in
    \Gamma_\ell (TE, E)$.
  \end{enumerate} 
  Point \refitem{iii} needs some explanations. First of all, there
  is a canonical one-to-one correspondence between fiber-wise linear
  functions on $E$ and sections of the dual vector bundle $E^\ast\to
  M$. Now, let $e$ be a section of $E$, and let
  $e^{\uparrow} \in \mathfrak X (E)$ be its vertical lift,
  i.e.~$e^{\uparrow}$ is the vector field defined by
  \begin{equation}
    e^{\uparrow}_{e_0} 
    := 
    \left. \frac{d}{d\varepsilon}\right|_{\varepsilon=0} \left(e_0 + \varepsilon\, e_{p(e_0)} \right),
  \end{equation}
  for all $e_0 \in E$. Notice that $e^{\uparrow}$ is uniquely
  determined by the identity
  \begin{equation}
    e^{\uparrow} (\varphi) = \langle \varphi, e \rangle
  \end{equation}
  where $\varphi \in \Gamma (E^\ast)$ is also understood as a
  fiber-wise linear function on $E$. A \emph{fiber-wise constant}
  vector field on $E$ is a vector field of the form $e^{\uparrow}$.

  If $\tilde X$ satisfies one of the above equivalent conditions we
  also say that $\tilde X$ is \emph{linear}.

  We conclude this remark noticing that a function $f \in C^\infty
  (E)$ is fiber-wise constant if and only if $h_\lambda^\ast f = f$,
  and it is fiber-wise linear if and only if $h_\lambda^\ast f =
  \lambda \cdot f$ for all $\lambda$. Equivalently, $f$ is fiber-wise
  constant if and only if $\mathcal E_E (f) = 0$, and it is fiber-wise
  linear if and only if $\mathcal E_E (f) = f$.  Similarly, a vector
  field $\tilde X \in \mathfrak X (E)$ is fiber-wise constant if and
  only if $(h_\lambda)_\ast \tilde X = \lambda \cdot \tilde X$,
  and (as already remarked) it is linear if and only if
  $(h_\lambda)_\ast \tilde X = \tilde X$, for all $\lambda \neq
  0$. Equivalently, $\tilde X$ is fiber-wise constant if and only if
  $[\mathcal E_E, \tilde X] = -\tilde X$, and it is linear if and only
  if $[\mathcal E_E,\tilde X] = 0$.
\end{remark}

The following lemma is familiar to experts.
\begin{lemma}
  \label{lem:dual}
  Infinitesimal automorphisms of $E$ correspond bijectively to
  derivations of $E$.  
\end{lemma}  
\begin{proof}
  We report a proof for completeness. Recall that the
  gauge algebroids $\der E$ and $\der E^\ast$ are canonically
  isomorphic. The isomorphism identifies a derivation $\Delta : \Gamma
  (E) \to \Gamma (E)$ with derivation $\Delta^\ast $ of $E^\ast$
  defined by
  \begin{equation}\label{eq:dual_derivation}
    \langle \Delta^\ast \varphi , e \rangle 
    = 
    \sigma (\Delta) \langle \varphi, e \rangle - \langle \varphi, \Delta e \rangle,
  \end{equation}
  for all $\varphi \in \Gamma(E^\ast)$ and $e \in \Gamma(E)$, where
  $\langle -,-\rangle : E^\ast \otimes E \to \mathbb{R}$ is the duality
  pairing.
  Here $\mathbb{R}_M:=M\times\mathbb{R}\to M$ denotes the trivial
  line bundle.  It is easy to see that $\Delta^{\ast\ast} = \Delta$
  for all $\Delta \in \Gamma(\der E)$.  Now, it follows from point
  \refitem{ii} in Remark \ref{rem:linear_VF} that there is a bijection
  between infinitesimal automorphisms and derivations of $E$, mapping
  $\tilde X$ to $\Delta_{\tilde X} := \tilde X|_{\Gamma(E^\ast)}^\ast$
  (where we identify fiber-wise linear functions on $E$ and sections
  of $E^\ast$).
\end{proof}

Let $\tilde X$ and $\Delta_{\tilde X}$ be as in the proof of Lemma
\ref{lem:dual}. Clearly, $\sigma (\Delta_{\tilde X}) =p_\ast \tilde
X$. Additionally, bijection $\tilde X \mapsto \Delta_{\tilde X}$ is
$\Cinfty (M)$-linear and commutator-preserving. It follows that the
flow $\{ \tilde \phi_{\varepsilon} \}_{\varepsilon}$ of $\tilde X$ is
uniquely determined by the following ODE
\begin{equation}
  \Delta e
  = 
  \left. \frac{d}{d\varepsilon} \right|_{\varepsilon = 0} \tilde \phi_{\varepsilon}^\ast e
\end{equation}
In this sense, $\Delta$ generates the flow $\{ \tilde \phi_{\varepsilon}
\}_\varepsilon$ of automorphisms of $E$. We denote by $\Delta \mapsto
X_\Delta$ the inverse correspondence.

\begin{remark}
  \label{rem:vertical_lift}
  There is another way to describe the bijection $\tilde X \mapsto
  \Delta_{\tilde X}$, which is often
  useful. Namely, let $\tilde X$ be an infinitesimal automorphism of
  $E$. Then, it follows from point \textit{iii.)} in Remark
  \ref{rem:linear_VF} that the derivation $\Delta_{\tilde X}$ is
  uniquely determined by
  \begin{equation}
    (\Delta_{\tilde X} e)^{\uparrow} = [\tilde X, e^{\uparrow}]
  \end{equation}
  for all $e \in \Gamma (E)$.
\end{remark}
In view of the Lemma~\ref{lem:dual}, we can make sense of
``compatibility of two derivations with respect to a, non-necessarily regular,
vector bundle morphism''. Namely, let $E_1 \to M_1$ and $E_2 \to M_2$
be vector bundles, and let $\tilde F : E_1 \to E_2$ that be a vector bundle
morphism covering a smooth map $F : M_1 \to M_2$. If $\Delta_1$ and
$\Delta_2$ are derivations of $E_1$ and $E_2$, respectively, we say
that $\Delta_1$ and $\Delta_2$ are \emph{$\tilde F$-related} if
$X_{\Delta_1}$ and $X_{\Delta_2}$ are $\tilde F$-related.  The latter
condition can be actually restated in terms of the sole $\Delta_1,
\Delta_2$. Indeed, the pull-back along $\tilde F$ maps fiber-wise
linear functions on $E_2$ to fiber-wise linear functions on
$E_1$. Accordingly, there is a well-defined map $\tilde F{}^\ast :
\Gamma (E_2^\ast) \to \Gamma (E_1^\ast)$, and $\Delta_1$ and
$\Delta_2$ are $\tilde F$-related if and only if
\begin{gather}
  \label{eq:F_related_derivations_1}
  \Delta_1^\ast \circ \tilde F{}^\ast = \tilde F{}^\ast \circ \Delta_2^\ast,\\
  \label{eq:F_related_derivations_2}
  \sigma(\Delta_1) \circ F{}^\ast = F{}^\ast \circ \sigma(\Delta_2).
\end{gather}
If $\tilde F$ is a surjective submersion,
then~\eqref{eq:F_related_derivations_2} follows directly
from~\eqref{eq:F_related_derivations_1}, $\Delta_1$ completely
determines $\Delta_2$ via \eqref{eq:F_related_derivations_1} and we
write $\Delta_2 = \tilde F_\ast \Delta_1$. In this case
$\sigma(\Delta_2) = F_\ast \sigma (\Delta_1)$.  On another hand, if
$\tilde F : E_1 \to E_2$ is the inclusion of a subbundle (over a
possibly smaller base), $\sigma (\Delta_2)$ is tangent to $M_1$,
$\Delta_2$ completely determines $\Delta_1$ and we write $\Delta_1 =
\Delta_2 |_{E_1}$.  In this case $\sigma (\Delta_1) = \sigma
(\Delta_2)|_{M_1}$.

\subsection{Infinitesimal automorphisms and multiplicative derivations}

Recall that a vector field $X \in \mathfrak X (G)$ on a Lie groupoid
$G \rightrightarrows M$ is \emph{multiplicative} if it generates a
flow by groupoid automorphisms. We begin this section adapting this
notion to the case of VB-groupoids.
Let $(\mathcal W \rightrightarrows E; G \rightrightarrows M)$ be a
VB-groupoid as in \eqref{diag:VB_groupoid}.

\begin{definition}[Infinitesimal automorphisms of VB-groupoids]
  \label{def:inf_auto_VB-group}
  A vector field $X_\mathcal W$ on $\mathcal W$ is said to be an
  \emph{infinitesimal automorphism} of $(\mathcal W \rightrightarrows E; G
  \rightrightarrows M)$ if it generates a flow by VB-groupoid
  automorphisms.
\end{definition}

\begin{example}
  Let $G \rightrightarrows M$ be a Lie groupoid and let $X$ be a
  multiplicative vector field on $G$. Recall that the \emph{tangent
    lift} of $X$ is the vector field $\tilde X$ on $TG$ whose flow is
  the differential $\{ \D \phi_t \}$ of the flow $\{ \phi_t \}$ of $X$
  (see, e.g., \cite{YI1973}). Then, $\tilde X$ is an infinitesimal
  automorphism of the tangent VB-groupoid $(TG\rightrightarrows TM;
  G\rightrightarrows M)$.
\end{example}

\begin{remark}
  \label{rem:inf_automorphism}
  Let $X_\mathcal W \in \mathfrak X(\mathcal W)$ be an infinitesimal
  automorphism of $(\mathcal W \rightrightarrows E; G \rightrightarrows
  M)$.
  It immediately follows from the definition that $X_\mathcal W$ projects
  onto vector fields
  \begin{itemize}
  \item $X_G = (q_\mathcal W)_\ast X_\mathcal W$ on $G$,
  \item $X_E = \tilde s_\ast X_\mathcal W = \tilde t_\ast X_\mathcal W$ on $E$,
  \item $X = s_\ast X_G = t_\ast X_G = q_\ast X_E$ on $M$. 
  \end{itemize}
  Again from the definition, $X_\mathcal W$ is tangent to the core
  $C$, and we denote by $X_C = X_\mathcal W |_C$ the restriction. We
  also have that $X_C$ projects onto $X$. Additionally, $X_G$ is a
  multiplicative vector field on $G$ and $X_E$ (resp.~$X_C$) is an
  infinitesimal automorphism of the vector bundle $E \to M$ (resp.~$C
  \to M$).  From another point of view, $X_\mathcal W$ corresponds to
  a derivation $\Delta_\mathcal W$ of the vector bundle $\mathcal W
  \to G$, $X_E$ corresponds to a derivation $\Delta_E$ of the vector
  bundle $E \to M$, and $X_C$ corresponds to a derivation $\Delta_C$
  of the vector bundle $C \to M$.  Finally $\Delta_\mathcal W$ and
  $\Delta_E$ are both $\tilde s$ and $\tilde t$-related: $\Delta_E =
  \tilde s_\ast \Delta_\mathcal W = \tilde t_\ast \Delta_\mathcal
  W$. Similarly, $\Delta_C = \Delta_\mathcal W |_C$.
\end{remark}
\begin{definition}[Multiplicative derivations]
  \label{def:mult_der}
  A derivation $\Delta_\mathcal W$ of $\mathcal W \to G$ is
  \emph{multiplicative} if the corresponding linear vector field
  $X_{\Delta_\mathcal W} \in \mathfrak X(\mathcal W)$ is an infinitesimal
  automorphism of the VB-groupoid $(\mathcal W \rightrightarrows E; G
  \rightrightarrows M)$.
\end{definition}
The map $\Delta_\mathcal W \mapsto
X_{\Delta_\mathcal W}$ establishes a one-to-one correspondence between
 multiplicative derivations and infinitesimal automorphisms of $(\mathcal W
\rightrightarrows E; G \rightrightarrows M)$. Consider
 a multiplicative derivation $\Delta_\mathcal W$ of $\mathcal W \to G$.  It
follows from Remark \ref{rem:inf_automorphism} that
\begin{itemize}
\item there is a derivation $\Delta_E$ of $E \to M$ such that
  $\Delta_E = \tilde s_\ast \Delta_\mathcal W = \tilde t_\ast
  \Delta_\mathcal W$,
\item there is a derivation $\Delta_C$ of $C \to M$ such that
  $\Delta_C = \Delta_\mathcal W |_C$,
\item the symbol of $\Delta_\mathcal W$ is a multiplicative vector field
  on $G$ projecting on the symbols of both $\Delta_E$ and $\Delta_C$
  (via both $s$ and $t$).
\end{itemize}
Let $(\mathcal W \rightrightarrows E; G \rightrightarrows M)$ be a
VB-groupoid with VB-algebroid $(W \Rightarrow E; A \Rightarrow M)$.
Take a linear section $\tilde a : E \to W$ covering a section $a$ of
$A \to M$, and let $\overrightarrow{\tilde a}$ and
$\overleftarrow{\tilde a}$ be the associated right and left invariant
vector fields on $\mathcal W$.

\begin{proposition}
  \label{prop:internal} 
  Both $\overrightarrow{\tilde a}, \overleftarrow{\tilde a}$ are
  linear vector fields on $\mathcal W$.
\end{proposition}
\begin{proof}
  Consider $\overrightarrow{\tilde a}$ first.  According to
  Remark~\ref{rem:linear_VF}, point \textit{iv.)}, it is enough to
  show that $(h_\lambda)_\ast \overrightarrow{\tilde
    a}=\overrightarrow{\tilde a}$, for all $\lambda \neq 0$.  So let
  $\omega\in\mathcal W$, with $\tilde t(\omega)=e$, and compute
  \begin{equation}
    ({h_\lambda}_\ast\overrightarrow{\tilde a})_\omega
    =
    \D h_\lambda(\D R_{\lambda^{-1}\omega}(\tilde a_{\lambda^{-1}e}))
    =
    \D R_{\omega}(dh_\lambda(\tilde a_{\lambda^{-1}e}))
    =
    \D R_{\omega}(\tilde a_{e})
    =
    \overrightarrow{\tilde a}_\omega,
  \end{equation}
  where $R$ denotes right translations in $\mathcal W$, and we used the
  linearity of $\tilde m$ and $\tilde a$ in the form of the identities
  $h_\lambda\circ R_{\lambda^{-1}\omega}=R_\omega\circ h_\lambda$ and $dh_\lambda(\tilde
  a_{\lambda^{-1}e})=\tilde a_e$ respectively.  A similar argument (with
  left translations in place of right translations) shows that
  ${h_\lambda}_\ast\overleftarrow{\tilde a}=\overleftarrow{\tilde a}$, for
  all $\lambda \neq 0$, i.e.~$\overleftarrow{\tilde a}$ is a linear vector field
  as well.
\end{proof}

\begin{remark}  
  \label{rem:internal}
  It follows from Proposition~\ref{prop:internal} that
  $\overrightarrow{\tilde a}$ and $\overleftarrow{\tilde a}$
  correspond to derivations of $\mathcal W \to G$, denoted
  $\overrightarrow{\Delta}_{\tilde a}$ and
  $\overleftarrow{\Delta}_{\tilde a}$, respectively.  Moreover,
  $\overrightarrow{\tilde a} + \overleftarrow{\tilde a}$ is a linear
  vector field which is additionally multiplicative.  Hence it
  corresponds to a multiplicative derivation of $\mathcal W$, namely
  $\overrightarrow{\Delta}_{\tilde a} + \overleftarrow{\Delta}_{\tilde
    a}$, also denoted $\delta \tilde a$, and called an \emph{internal
    derivation} of $(\mathcal W\rightrightarrows E;G\rightrightarrows M)$.
  By duality, $\overrightarrow{\tilde a} + \overleftarrow{\tilde a}$
  does also correspond to a multiplicative derivation of
  $\mathcal W^\ast$.  Similarly, a linear section of $W^\ast \to C^\ast$
  determines a multiplicative derivation of both $\mathcal W$ and
  $\mathcal W^\ast$.
\end{remark}

\begin{remark}
  \label{rem:mult_vf}
  There is an important characterization of multiplicative vector
  fields on a Lie groupoid $G \rightrightarrows M$. Namely, a vector
  field $X \in \mathfrak X (G)$ is multiplicative if and only if, when
  regarded as a section $X : G \to TG$ of the tangent VB-groupoid, it
  is a morphism of Lie groupoids~\cite{Mackenzie1997}. This suggests
  the following definition. Given a VB-groupoid $(\mathcal W
  \rightrightarrows E; G \rightrightarrows M)$, a section $w$ of
  $\mathcal W \to G$ is \emph{multiplicative} if it is a Lie groupoid
  morphism. So, a vector field $X_\mathcal W \in \mathfrak X (\mathcal
  W)$ is an infinitesimal automorphism of $(\mathcal W
  \rightrightarrows E; G \rightrightarrows M)$ if and only if the
  following conditions are satisfied:
  \begin{enumerate}
   \item it is a linear vector field with respect to the vector bundle
     structure $\mathcal W \to G$,
  \item when regarded as a section $X_\mathcal W$ of $T \mathcal W \to
    \mathcal W$, it is a multiplicative section of the (tangent)
    VB-groupoid $(T\mathcal W\rightrightarrows TE; \mathcal
    W\rightrightarrows E)$.
  \end{enumerate}
  Yet in other words, infinitesimal automorphisms of $(\mathcal W
  \rightrightarrows E; G \rightrightarrows M)$ are the same as
  morphisms between the VB-groupoids $(\mathcal W \rightrightarrows E;
  G \rightrightarrows M)$ and $(T\mathcal W \rightrightarrows TE;
  TG\rightrightarrows TM)$ inverting the projection $T\mathcal W \to
  \mathcal W$ on the right.
 \end{remark}
 
 \begin{remark}\label{rem:mult_der}
  A similar characterization as for multiplicative vector fields (and
  infinitesimal automorphism of VB-groupoids, see Remark
  \ref{rem:mult_vf}) cannot take place, in general, for multiplicative
  derivations. The reason is that the gauge algebroid $\der \mathcal
  W$ is \emph{not} a VB-groupoid itself, if $(\mathcal W
  \rightrightarrows E; G \rightrightarrows M)$ is a \emph{generic}
  VB-groupoid. In Section \ref{sec:trivial-core} we discuss two simple
  but interesting classes of VB-groupoids, those with either trivial
  core or trivial side vector bundle $0_M \to M$ (we call the latter
  \emph{full-core} VB-groupoids). Trivial-core VB-groupoids
  $(E_G\rightrightarrows E; G\rightrightarrows M)$ come from
  representations $E$ of Lie groupoids $G$ and possess the nice
  property that their gauge algebroid $\der E_G$ fits in a VB-groupoid
  $(\der E_G\rightrightarrows \der E; G\rightrightarrows M)$
  itself. It turns out that multiplicative derivations of a
  trivial-core VB-groupoid $(E_G\rightrightarrows E;
  G\rightrightarrows M)$ are the same as multiplicative sections of
  $\der E_G \to G$ (see Theorem \ref{theor:der_VB-grpd}). As duality
  establishes a one-to-one correspondence between trivial-core and
  full-core VB-groupoids, similar considerations hold true for
  multiplicative derivation of full-core VB-groupoids.
\end{remark}

\begin{remark}\label{rem:answer}
  We already have a simple and conceptually satisfactory description
  of infinitesimal automorphisms of a VB-groupoid $(\mathcal W
  \rightrightarrows E; G \rightrightarrows M)$ as morphisms between
  the VB-groupoids $(\mathcal W \rightrightarrows E; G
  \rightrightarrows M)$ and $(T \mathcal W\rightrightarrows TE;
  TG\rightrightarrows TM)$ (Remark \ref{rem:mult_vf}), so the reader
  may wonder: \emph{why do we care so much about multiplicative
    derivations} (Remark \ref{rem:mult_der})?  There are at least two
  reasons for that. First of all, the target space $T \mathcal W$ of
  an infinitesimal automorphism is an example of a triple structure:
  it possesses a groupoid structure $T \mathcal W \rightrightarrows
  TE$, and two compatible vector bundle structures $T\mathcal W \to
  \mathcal W$ and $T\mathcal W \to TG$. All these structures play a
  role when dealing with infinitesimal automorphisms of $(\mathcal W
  \rightrightarrows E; G \rightrightarrows M)$. This could potentially
  lead to complications, especially in view of the fact that a
  detailed theory of triple structures like $T\mathcal W$ is still
  unavailable (and developing it would take us too far from our
  original scopes). On the other hand, derivations provide a simpler,
  more algebraic, and somehow minimal description of the same
  objects. Second, and, probably, more important, the reader should
  remember that this paper represents just the first step in a program
  aiming at studying (higher order) DOs on groupoids and
  algebroids. In this respect, derivations are an instance of first
  order DOs. Now, derivations can be seen as linear vector fields, but
  there is no similar description for higher order DOs.  So working
  with derivations directly is a better training to learn how to work
  with generic DOs.
\end{remark}

We now pass to VB-algebroids. First recall that a vector field $X \in
\mathfrak X (A)$ on a Lie algebroid $A \Rightarrow M$ is
\emph{infinitesimally multiplicative} (IM) if it generates a flow by
Lie algebroid automorphisms.  Let $(W\Rightarrow E;A \Rightarrow M)$
be a VB-algebroid.
As in \eqref{diag:VB_algebroid}, we denote by $\tilde p : W \to E$,
and $p : A \to M$ the projections.

\begin{definition}[Infinitesimal automorphisms of VB-algebroids]
  \label{def:inf_auto_VB-alg}
  A vector field $X_W$ on $W$ is an \emph{infinitesimal automorphism}
  of $(W\Rightarrow E; A\Rightarrow M)$ if it generates a flow by
  VB-algebroid automorphisms.
\end{definition}

\begin{remark}
  \label{rem:inf_automorphism_alg}
  Let $X_W \in \mathfrak X(W)$ be an infinitesimal automorphism of $(W
  \Rightarrow E; A \Rightarrow M)$. It immediately follows from the
  definition that $X_W$ projects onto vector fields
  \begin{itemize}
  \item $X_A = (q_W)_\ast X_W$ on $A$,
  \item $X_E = \tilde p_\ast X_W$ on $E$,
  \item $X = p_\ast X_A = q_\ast X_E$ on $M$.
  \end{itemize}
  Again from the definition, $X_W$ is tangent to the core $C$ and we
  denote by $X_C = X_W|_C$ the restriction.  We also have that $X_C$
  projects onto $X$.  Additionally, $X_A$ is an IM vector field on
  $A$, and $X_E$ (resp.~$X_C$) is an infinitesimal automorphism of the
  vector bundle $E \to M$ (resp.~$C \to M$).  From another point of
  view, $X_W$ corresponds both to a derivation $\Delta_{W, A}$ of the
  vector bundle $W \to A$, and a derivation $\Delta_{W,E}$ of $W \to
  E$. In the same way, $X_E$ corresponds to a derivation $\Delta_E$ of
  the vector bundle $E \to M$, $X_C$ corresponds to a derivation
  $\Delta_C$ of the vector bundle $C \to M$, and $X_A$ corresponds to
  a derivation $\Delta_A$ of $A$ (sharing the same symbol as
  $\Delta_E$ and $\Delta_C$).  Finally $\Delta_{W,A}$ and $\Delta_E$
  are $\tilde p$-related, $\Delta_C$ agrees with the restrictions of
  both $\Delta_{W,A}$ and $\Delta_{W,E}$ to $C \to M$, and
  $\Delta_{W,E}$ and $\Delta_A$ are $q_W$-related: $\Delta_E = \tilde
  p_\ast \Delta_{W,A}$, $\Delta_C = \Delta_{W,A}|_C = \Delta_{W,E}|_C$
  and $\Delta_A = (q_W)_\ast \Delta_{W,E}$.  We call $\Delta_{W, A}$
  the \emph{vertical derivation} and $\Delta_{W, E}$ the
  \emph{horizontal derivation} corresponding to $X_W$.
  In the following we provide a \emph{more algebraic} description of
  both $\Delta_E$ and $\Delta_C$ in terms of $\Delta_{W,E}$ (see
  Theorem \ref{theor:delta}).
\end{remark}
\begin{definition}[IM derivations]
  \label{def:IM_der}
  A derivation $\Delta_{W,A}$ of $W \to A$ is an \emph{IM derivation}
  if the corresponding linear vector field $X_{W} \in \mathfrak X(W)$
  is an infinitesimal automorphism of the VB-algebroid $(W \Rightarrow
  E; A \Rightarrow M)$.
\end{definition}
Recall from \cite{Mackenzie1997} that if $A \Rightarrow M$ is a Lie
algebroid, a \emph{Lie algebroid derivation} of $A \Rightarrow M$ is a
derivation $\Delta$ of $A \to M$ such that
\begin{equation}
  \label{eq:0-diff}
  \Delta [a,b] 
  = 
  [\Delta a, b] + [a, \Delta b],
\end{equation}
for all $a,b \in \Gamma (A)$. It follows from \eqref{eq:0-diff}
that $\Delta$ is also compatible with the anchor, i.e.
\begin{equation}
  \label{eq:RhoDelta}
  \rho (\Delta a) 
  = 
  [\sigma (\Delta), \rho (a)],
\end{equation}
for all $a \in \Gamma (A)$. The commutator of two Lie algebroid
derivations is clearly a Lie algebroid derivation as well. Now, the
usual map $X \mapsto \Delta_X$ establishes a one-to-one,
commutator-preserving, correspondence between IM vector fields on $A$
and Lie algebroid derivations of $A$ (see~\cite{Mackenzie1997}). So,
Lie algebroid derivations provide an algebraic description of IM
vector fields.  Notice that Lie algebroid derivations of $A$ are the
same as $1$-differentials of $A$ (cf.~\cite{Iglesias-Ponte2012}).

Let $(W \Rightarrow E; A \Rightarrow M)$ be a VB-algebroid. Take
a linear section $\tilde a$ of $W \to E$ and consider the adjoint
operator $[\tilde a, -]$. It is a derivation of $W \to
E$. Additionally, it is a Lie algebroid derivation. Thus, from the
above discussion it follows that $X_{[\tilde a, -]}$ is an
infinitesimal automorphism of $W \Rightarrow E$.

\begin{proposition}
  \label{prop:internal_alg}
  A vector field $X_{[\tilde a, -]}$ is an infinitesimal automorphism of
  the VB-algebroid \linebreak $(W \Rightarrow E; A \Rightarrow M)$.
  \end{proposition}
Before proving Proposition~\ref{prop:internal_alg} we need some
general remarks which will be useful in the next
Section~\ref{sec:linear_def} as well.  Let $\mathcal E_{W,A}$ and
$\mathcal E_{W, E}$ be the Euler vector fields of vector bundles $W
\to A$ and $W \to E$, respectively.  As $W$ is a double vector bundle,
$\mathcal E_{W,A}$ and $\mathcal E_{W, E}$ commute.  In particular, it
follows from Remark~\ref{rem:linear_VF}, point \textit{v.)}, that
$\mathcal E_{W,A}$ is an infinitesimal automorphism of the vector
bundle $W \to E$.  Hence, it corresponds to a derivation
$\Delta_{\mathcal E}$ of $W \to E$.  The symbol of $\Delta_{\mathcal
  E}$ is $\tilde{p}_\ast \mathcal{E}_{W,A}$ and coincides with
$\mathcal E_E$, the Euler vector field of $E$.  The following lemma
parallels the characterization of fiber-wise constant (resp.~linear)
functions and vector fields on a vector bundle in terms of the Euler
vector field (as discussed in Remark \ref{rem:linear_VF}) and can be
easily proved in local coordinates adapted to the double vector bundle
structure of $W$ (see e.g. \cite{GR2009,gracia2010lie}).
\begin{lemma}
  \label{lem:linear_Euler}
  A section $w$ of $W \to E$ is a core section if and only if
  $\Delta_{\mathcal E } w = -w$ and it is linear if and only if
  $\Delta_{\mathcal E} w = 0$. \color{black} Additionally, there are no
  non-trivial sections $w$ of $W \to E$ such that $\Delta_{\mathcal E}
  w = - i w$ for some integer $i > 1$.  \color{black}
 \end{lemma}

\begin{proof}[of Proposition \ref{prop:internal_alg}]
  As the vector field $X_{[\tilde a, -]}$ is already known to be an
  infinitesimal automorphism of the Lie algebroid $W \Rightarrow E$,
  it remains to check that it is linear with respect to the vector
  bundle structure $W \to A$. Now a vector field on $W$ is linear with
  respect to $W \to A$ if and only if it commutes with $\mathcal
  E_{W,A}$ (cf.~Remark~\ref{rem:linear_VF}, point \textit{v.)}).  In
  order to see that $X_{[\tilde a, -]}$ commutes with $\mathcal
  E_{W,A}$ it is enough to check that the operator $[\tilde a, -]$
  commutes with $\Delta_{\mathcal E}$ (cf.~Lemma~\ref{lem:dual}). In
  other words we have to show that the commutator $\square$ of
  $[\tilde a, -]$ and $\Delta_{\mathcal E}$ vanishes. To do this it is
  enough to check that
  \begin{enumerate}
  \item the symbol of $\square$ vanishes,
  \item $\square$ vanishes on generators (over $C^\infty (E)$).
  \end{enumerate}
  Now, for any linear section $\tilde a$ of $W \to E$, the symbol of
  $[\tilde a, -]$ is $\tilde \rho (\tilde a)$ and it is a linear
  vector field on $E$. On the other hand the symbol of $\Delta$ is the
  Euler vector field $\mathcal E_E$ of $E$. So $\sigma (\square) =
  [\tilde \rho (\tilde a), \mathcal E_E] = 0$. Finally, from
  Definition~\ref{def:VBAlgebroid} and Lemma~\ref{lem:linear_Euler},
  $\square$ vanishes on both core and linear sections.
\end{proof}

\begin{remark}
  \label{rem:internal_alg}
  It follows from Proposition \ref{prop:internal_alg} that the vector
  field $X_{[\tilde a, -]}$ corresponds to an IM derivation of $W \to
  A$, denoted $\delta \tilde a$ and called an \emph{internal
    derivation}, and by duality to an IM derivation of $W^\ast \to
  A$. Similarly, a linear section of $W^\ast \to C^\ast$ determines an
  IM derivation of both $W$ and $W^\ast$.
\end{remark}

Note that if $(W\Rightarrow E;A \Rightarrow M)$ is the VB-algebroid
of a VB-groupoid $(\mathcal W \rightrightarrows W;G \rightrightarrows
M)$, then for any $\tilde a\in\Gamma_\ell(W,E)$ the vector field
$X_{[\tilde a, -]} \in \mathfrak X (W)$ is the image of
$\overrightarrow{\tilde a} + \overleftarrow{\tilde a}$ under the Lie
functor. We use the same symbol $\delta \tilde a$ to denote the
derivations corresponding to both. It will be always clear from the
context if we are referring to one or the other.

\begin{remark}
  \label{rem:IM_vf}
  Let $A \to M$ be a Lie algebroid.  As proved
  in~\cite{Mackenzie1997}, a vector field $X \in \mathfrak X (A)$ is
  IM if and only if, when regarded as a section $X : A \to TA$ of the
  tangent VB-algebroid, it is a morphism of Lie algebroids.  This
  suggests the following definition. Given a VB-algebroid $(W
  \Rightarrow E; A \Rightarrow M)$ a section $s : A \to W$ of $W \to
  A$ is an \emph{IM section} if it is a Lie algebroid morphism. Thus,
  similarly as for infinitesimal automorphisms of VB-groupoids
  (Remark~\ref{rem:mult_vf}), a vector field $X_W \in \mathfrak X (W)$
  is an infinitesimal automorphism of $(W \Rightarrow E; A \Rightarrow
  M)$ if and only if it is a linear vector field with respect to the
  vector bundle structure $W \to A$, and when regarded as a section of
  $TW \to W$, it is an IM section of the tangent VB-algebroid
  $(TW\Rightarrow TE; W\Rightarrow E)$. In other words, infinitesimal
  automorphisms of $(W\Rightarrow E; A\Rightarrow M)$ are the same as
  morphisms between the VB-algebroids $(W \Rightarrow E; A \Rightarrow
  M)$ and $(TW\Rightarrow TE; TA\Rightarrow TM)$ inverting the
  projection $TW \to W$ on the right.
  \end{remark}

\begin{remark}\label{rem:IM_vf_2}
  A similar characterization as for IM vector fields cannot take
  place, in general, for IM derivations because the gauge algebroid
  $\der W$ of $W \to A$ is \emph{not} a VB-algebroid itself, when
  $(W\Rightarrow E; A\Rightarrow M)$ is a \emph{generic} VB-algebroid.
  However, it turns out that the gauge algebroid $\der E_A$ of a
  \emph{trivial-core} VB-algebroid $(E_A\Rightarrow E; A\Rightarrow
  M)$ (see Section \ref{sec:trivial-core}) fits in a VB-algebroid
  $(\der E_A, \Rightarrow \der E; A \Rightarrow M)$ itself, and IM
  derivations of $(E_A\Rightarrow E; A\Rightarrow M)$ are the same as
  IM sections of $\der E_A \to A$ (Theorem \ref{theor:der_VB-alg}).
  Similar considerations hold true for IM derivations of trivial-core
  VB-algebroids.
\end{remark}
The following proposition proves that IM derivations of VB-algebroids
are the infinitesimal counterpart of multiplicative derivations of
VB-groupoids.  It follows immediately from Definitions
\ref{def:inf_auto_VB-group}, \ref{def:mult_der},
\ref{def:inf_auto_VB-alg}, \ref{def:IM_der} and the general discussion
in Subsection \ref{sec:diff_int} (see \cite{bursztyn2016} for more
details).

\begin{proposition}
  Let $(\mathcal W \rightrightarrows E; G \rightrightarrows M)$ be a
  VB-groupoid with VB-algebroid $(W \Rightarrow E; A \Rightarrow
  M)$. The Lie functor maps infinitesimal automorphisms of $\mathcal
  W$ to infinitesimal automorphisms of $W$.  If, additionally,
  $\mathcal W$ is source simply connected, then any infinitesimal
  automorphism of $W$ integrates to a unique infinitesimal
  automorphism of $\mathcal W$.  In particular, there is a one-to-one
  correspondence between multiplicative derivations of $\mathcal W $
  and IM derivations of $W $.
\end{proposition}
We conclude this subsection presenting an alternative description of
IM derivations of a VB-algebroid in the same spirit as the algebraic
description of IM vector fields on a Lie algebroid provided by Lie
algebroid derivations. 
\begin{remark}
  \label{rem:delta_formula}
  Let $G \rightrightarrows M$ be a Lie groupoid with Lie algebroid $A
  \Rightarrow M$. The Lie functor maps a multiplicative vector field
  $X_G$ on $G$ to an IM vector field $X_A$ on $A$. In its turn $X_A$
  corresponds to a Lie algebroid derivation $\Delta_A$ of $A$. It has
  been shown in~\cite[Theorem 3.9]{Mackenzie1997} that for every
  right invariant vector field $Y$ on $G$, $[X_G, Y]$ is a right
  invariant vector field and, more importantly, $X_G$ and $\Delta_A$
  are related by the following formula:
  \begin{equation}
    \overrightarrow{\Delta_A a} = [X_G, \overrightarrow a],
  \end{equation}
  for all $a\in\Gamma(A)$.
\end{remark}
Now, let $(W \Rightarrow E; A \Rightarrow M)$ be a VB-algebroid, and
let $X_W$ be an infinitesimal automorphism of it. The horizontal
derivation $\Delta_{W,E}$ corresponding to $X_W$ is a Lie algebroid
derivation of $W \Rightarrow E$. We want to provide an alternative
description of those Lie algebroid derivations of $W \Rightarrow E$
that arise in this way.  In order to do this, we need to recall some
extra features of VB-algebroids.

The space $\Gamma_{\ell}(W,E)$ of linear sections of $W \to E$ (see
Remark \ref{rem:linear_sections}) is a $C^\infty (M)$-module in an
obvious way. Actually, it is the module of sections of a vector bundle
$\widehat A \to M$. Specifically, a point of $\widehat A$ over a point
$x \in M$ is a pair $(a, h)$ consisting of a point $a$ of $A_x$ and a
linear map $h : E_x \to W_a$ which inverts the projection $W_a\to E_x$
on the right.  As linear sections project onto sections of $A \to M$,
there is a vector bundle epimorphism $\pi : \widehat A \to A$.  In
view of the canonical splitting $W|_M=C\oplus E$, linear sections of
$W\to E$ taking their values in the kernel of $\pi:\widehat A\to A$
identify with vector bundle morphisms $E\to C$.  So there is a short
exact sequence
\begin{equation}
  0 
  \longrightarrow 
  \Hom (E, C) 
  \longrightarrow 
  \widehat 
  A 
  \overset{\pi}{\longrightarrow} 
  A 
  \longrightarrow 
  0,
\end{equation}
and in what follows we understand the inclusion $\Hom (E, C)
\hookrightarrow \widehat A$.  In this regard notice that for any
$\chi\in\Gamma(C)$ and $\varphi\in\Gamma(E^\ast)$ the latter
inclusion identifies $\varphi\otimes\chi\in\Gamma(E^\ast\otimes C)$
with the linear section $\varphi\widehat{\chi}\in\Gamma_{\ell}(W,E)$
given by the product of the fiber-wise linear function $\varphi$ and
the core section $\widehat \chi$.

The Lie bracket on $\Gamma (W,E)$ restricts to linear sections. As a
consequence, there is a Lie algebroid structure $\widehat A
\Rightarrow M$ with anchor given by $\rho \circ \pi$. This Lie
algebroid is the so-called \emph{fat algebroid} of $(W \Rightarrow E;
A \Rightarrow M)$ (see \cite{gracia2010lie}). Vector bundle $E \to M$
carries a canonical representation of the fat algebroid, called the
\emph{side representation} and denoted by $\psi^s$. Let $\tilde a$ be
a section of $\widehat A$ or, equivalently, a linear section of $W \to
E$ projecting onto a section $a$ of $A \to M$. The side representation
is implicitly given by
\begin{equation}
  \label{eq:SideImpl}
  \langle \varphi, \psi^s_{\tilde a} e \rangle 
  = 
  \rho (a) \langle \varphi, e \rangle - \langle \tilde \rho (\tilde a) (\varphi),e \rangle,
\end{equation}
for all $\varphi \in \Gamma(E^\ast)$, and $e \in \Gamma (E)$.  Here,
as usual, $\varphi$ is also regarded as a fiber-wise linear function
on $E$.  Recall that given a linear section $\tilde a$ of $W \to E$,
vector field $\tilde \rho (\tilde a)$ is an infinitesimal automorphism
of the vector bundle $E$, hence its Lie derivative preserves vertical
lifts of sections of $E$ (see Remark \ref{rem:vertical_lift}) and
\begin{equation}
  \label{eq:side_vertical_lift}
  \left(\psi^s_{\tilde a} e \right)^{\uparrow} 
  = 
  [ \tilde \rho (\tilde a), e^{\uparrow}],
\end{equation}
for all $e \in \Gamma (E)$. Yet in other words, $\psi^s_{\tilde a}$ is
nothing but the derivation of $E$ corresponding to the infinitesimal
automorphism $\tilde \rho (\tilde a)$ of $E$ (Lemma~\ref{lem:dual}). 
The core $C \to M$ does also carry a representation of the fat
algebroid, called the \emph{core representation} and denoted by
$\psi^c$.  The core representation is implicitely given by
\begin{equation}
  \label{eq:CoreRepres}
  \widehat{\psi^c_{\tilde a} \chi} 
  = 
  [\tilde a, \widehat \chi]
\end{equation}
for all $\tilde a \in \Gamma(\widehat A)$ and $\chi \in \Gamma (C)$,
where $\widehat \chi$ denotes the core section of $W \to E$
corresponding to $\chi$.
Finally, a distinguished vector bundle morphism $\partial : C
\to E$, called the \emph{core-anchor}, is defined by
\begin{equation}
  \label{eq:AlphaRho}
  \langle \partial \chi , \varphi \rangle 
  = 
  \tilde \rho (\widehat \chi) (\varphi),
\end{equation}
for all $\chi \in \Gamma (C)$ and $\varphi \in \Gamma (E^\ast)$.
Notice that our convention on the core-anchor differs by a sign from
that of \cite{gracia2010lie}.  There is an equivalent way to describe
the core-anchor.  Actually, in view of
Definition~\ref{def:VBAlgebroid} and Lemma~\ref{lem:linear_Euler}, for
any $\chi\in\Gamma(C)$, vector field $\tilde \rho
(\widehat\chi)\in\mathfrak{X}(E)$ is the vertical lift of a
(necessarily unique) section of $E$ (see
Remark~\ref{rem:vertical_lift}) and
\begin{equation}
  \label{eq:core-anchor_vertical_lift}
  \left(\partial\chi\right)^{\uparrow}
  =
  \tilde\rho(\widehat\chi).
\end{equation}

\begin{example}
  Let $E \to M$ be a vector bundle, and recall that its \emph{tangent
    double vector bundle} $(TE \to E; TM \to M)$ is actually a
  VB-algebroid $(TE \Rightarrow E; TM \Rightarrow M)$. It follows from
  Remark \ref{rem:linear_VF}, point \textit{vi.)}, and
  Lemma~\ref{lem:dual}, that its fat algebroid is precisely the gauge
  algebroid $\der E$ of $E$. Additionally, both the side vector bundle
  and the core of $(TE\Rightarrow E; TM\Rightarrow M)$ are
  (canonically isomorphic to) $E$ (see also Example \ref{ex:tangent})
  and the core/side representations both agree with the tautological
  action of $\der E$ on $E$. Finally the core-anchor is just the
  identity $\mathrm{id} : E \to E$.
\end{example}
We now show that infinitesimal automorphisms of $(W \Rightarrow E; A
\Rightarrow M)$ or equivalently IM derivations of $W \to A$ can be
characterized in terms of certain algebraic data, in a similar way as
IM vector fields on a Lie algebroid can be characterized in terms of
Lie algebroid derivations. So, let $X_W$ be an infinitesimal
automorphism of $(W \Rightarrow E; A \Rightarrow M)$ and let
$\Delta_{W, A}$ and $\Delta_{W,E}$ be its vertical and horizontal
derivations respectively.  Since $X_W$ is, in particular, an IM vector
field on $W$, then $\Delta_{W,E}$ is a Lie algebroid derivation of the
Lie algebroid $W \Rightarrow E$.  Additionally, $X_W$ is also an
infinitesimal automorphism of the vector bundle $W \to A$, i.e.~it
commutes with $\mathcal{E}_{W,A}$ (cf.~Remark~\ref{rem:linear_VF},
point \textit{v.)}).  It follows that $\Delta_{W,E}$ commutes with
$\Delta_{\mathcal E}$ (see Lemma~\ref{lem:dual}).  Hence, in view of
Lemma~\ref{lem:linear_Euler}, $\Delta_{W,E}$ preserves linear sections
of $W \to E$ and it restricts to a derivation $\Delta_{\widehat A} :
\Gamma (\widehat A) \to \Gamma (\widehat A)$ of the fat
algebroid. Obviously, $\Delta_{\widehat A}$ is a Lie algebroid
derivation.  For the same reason, $\Delta_{W,E}$ preserves core
sections of $W \to E$ and determines a derivation $\Delta_C : \Gamma
(C) \to \Gamma (C)$ of the core.  Consider also the derivations
$\Delta_A,\Delta_E$ of $A,E$, respectively, as in Remark
\ref{rem:inf_automorphism_alg}. We are now ready to state the main
result in this section.
\begin{theorem}
  \label{theor:delta}
  Let $(W \Rightarrow E; A \Rightarrow M)$ be a VB-algebroid. The
  assignments $X_W \mapsto \Delta_{W, A}$ and $X_W \mapsto
  (\Delta_{\widehat A}, \Delta_E, \Delta_C)$ establish one-to-one
  correspondences between
  \begin{itemize}
  \item infinitesimal automorphisms of $(W \Rightarrow E; A
    \Rightarrow M)$,
  \item IM derivations of $W \to A$,
  \item triples $(\Delta_{\widehat A}, \Delta_E, \Delta_C)$ consisting
    of
    \begin{definitionlist}
    \item a Lie algebroid derivation $\Delta_{\widehat A}$ of the fat
      algebroid $\widehat A$,
    \item a derivation $\Delta_E$ of $E$, 
    \item a derivation $\Delta_C$ of $C$,
    \end{definitionlist}
    such that
    \begin{equation} 
      \label{eq:sigma_delta}
      \sigma (\Delta_{\widehat A}) = \sigma (\Delta_E ) = \sigma (\Delta_C), 
    \end{equation}
    \begin{equation} \label{eq:delta_hom}
      \Delta_{\widehat A} \Phi
      = 
      \Delta_C \circ \Phi - \Phi \circ \Delta_E,
    \end{equation}
    for all vector bundle morphisms $\Phi:E\to C$ (also regarded as
    sections of $\widehat A$ via the inclusion $\Hom (E, C)
    \hookrightarrow\widehat A$), and
    \begin{equation}
      \label{eq:core-anchor}
      \partial \circ \Delta_C
      = 
      \Delta_E \circ \partial,
    \end{equation}
    \begin{equation}
      \label{eq:delta_psi}
            [\Delta_E, \psi^s_{\tilde a}] = \psi^s_{\Delta_{\widehat A} \tilde a},
    \end{equation}
    \begin{equation}
      \label{eq:delta_psi_c}
            [\Delta_C, \psi^c_{\tilde a}] = \psi^c_{\Delta_{\widehat A} \tilde a},
    \end{equation}
    for all $\tilde a \in \Gamma (\widehat A)$.
  \end{itemize}
\end{theorem}

\begin{proof}
  The correspondence $X_W \mapsto \Delta_{W, A}$ between infinitesimal
  automorphisms and IM derivations is one-to-one by definition. Now,
  begin with $X_W$ and consider the associated data $(\Delta_{\widehat
    A}, \Delta_E, \Delta_C)$. The discussion immediately before the
  statement shows that $\Delta_{\widehat A}$ is a Lie algebroid
  derivation. Property~\eqref{eq:sigma_delta} follows from the fact
  that all terms agree with the projection $X$ of $X_W$ on $M$.  For
  property~\eqref{eq:delta_hom}, it is enough to consider $\Phi$ of
  the form $\Phi=\varphi\otimes\chi$ (or equivalently
  $\Phi=\varphi\widehat{\chi}$), where $\varphi\in\Gamma(E^\ast)$ and
  $\chi\in\Gamma(C)$.  From the Leibniz rule for $\Delta_{W,E}$ it
  follows immediately that
  \begin{equation}
    \Delta_{\widehat A}\Phi
    =
    (\Delta_E^\ast\varphi)\otimes\chi+\varphi\otimes\Delta_C\chi\in\Gamma(\Hom(E,C)).
  \end{equation}
  Hence $\Delta_{\widehat A}$ restricts to a derivation of
  $\Hom(E,C)$.  Additionally, from~\eqref{eq:sigma_delta} we get also
  \begin{equation}
    \langle\Delta_{\widehat A}\Phi,e\rangle
    =
    \langle\Delta_E^\ast\varphi,e\rangle\widehat\chi+\langle\varphi,e\rangle\widehat{\Delta_C\chi}
    =
    -\langle\varphi,\Delta_Ee\rangle\widehat\chi+\Delta_C(\langle\varphi,e\rangle\widehat{\chi})
    =
    \langle\Delta_C\circ\Phi-\Phi\circ\Delta_E,e\rangle,
  \end{equation}
  for all $e\in\Gamma(E)$, i.e.~\eqref{eq:delta_hom} holds.  For
  \eqref{eq:core-anchor}, fix $\chi \in \Gamma (C)$ and compute
  \begin{equation}
    (\partial \Delta_C \chi)^{\uparrow}
    =
    \tilde \rho \left( \Delta_{W,E} \widehat \chi \right)
    =
    \left[\sigma (\Delta_{W,E}), \tilde \rho \left( \widehat \chi \right) \right]
    =
    [X_{\Delta_E},(\partial\chi)^{\uparrow}]=(\Delta_E\partial\chi)^{\uparrow},
  \end{equation}
  where we used the fact that $\Delta_{W,E}$ is a Lie algebroid
  derivation of $W$ whose symbol is the linear vector field on $E$
  corresponding to $\Delta_E$.  For the same reason, for all
  $a\in\Gamma(\widehat A)$ and $e\in\Gamma(E)$
  \begin{equation}
    (\psi^s_{\Delta_{\widehat A} \tilde a} e)^{\uparrow}
    =
    [\tilde\rho(\Delta_{W,E}\tilde a),e^{\uparrow}]
    =
    [[\sigma(\Delta_{W,E}),\tilde\rho(\tilde a)],e^{\uparrow}]=([\Delta_E,\psi^s_{\tilde a}]e)^{\uparrow},
  \end{equation}
  which proves \eqref{eq:delta_psi} and for all $\chi \in \Gamma
  (C)$,
  \begin{equation}
    \widehat{\psi^c_{\Delta_{\widehat A} \tilde a} \chi}
    =
    [\Delta_{W,E}\tilde a,\widehat\chi]
    =
    \Delta_{W,E}[\tilde a,\widehat\chi]-[\tilde a,\Delta_{W,E}\widehat \chi]
    =
    \widehat{[\Delta_C, \psi^c_{\tilde a}] \chi},
  \end{equation}
  which proves \eqref{eq:delta_psi_c}.
  Conversely, let $(\Delta_{\widehat A}, \Delta_E, \Delta_C)$ be as in
  the statement.  From~\eqref{eq:sigma_delta}, 
  and \eqref{eq:delta_hom}, $(\Delta_{\widehat A}, \Delta_E,
  \Delta_C)$ extends uniquely to an infinitesimal automorphism $X_W$
  of the double vector bundle $(W \Rightarrow E; A \Rightarrow M)$.
  From~\eqref{eq:core-anchor}, \eqref{eq:delta_psi},
  \eqref{eq:delta_psi_c} and the fact that $\Delta_{\widehat A}$ is a
  Lie algebroid derivation, $X_W$ is an infinitesimal automorphism of
  the VB-algebroid structure as well.  The details can be easily
  checked in local coordinates.
\end{proof}  
Let $(\Delta_{\widehat A}, \Delta_E, \Delta_C)$ be a triple
as in the above theorem. Then, derivations $\Delta_E, \Delta_C$ agree with
those in Remark \ref{rem:inf_automorphism_alg}.  Notice that it
follows from \eqref{eq:delta_hom} that both pairs $(\Delta_{\widehat
 A}, \Delta_E)$ and $(\Delta_{\widehat A}, \Delta_C)$ determine the
whole triple. Additionally $\Delta_{\widehat A}$ and derivation
$\Delta_A$ from Remark \ref{rem:inf_automorphism_alg}, are
$\pi$-related.
 \begin{remark}
   \label{rem:dvb_vs_vb_alg}
   Equations (\ref{eq:sigma_delta}) and
   (\ref{eq:delta_hom}) express compatibility with the double vector
   bundle structure, while Equations (\ref{eq:core-anchor}),
   (\ref{eq:delta_psi}), (\ref{eq:delta_psi_c}) express (the residual)
   compatibility with the algebroid structure.
 \end{remark}
From Proposition~\ref{prop:diff_int}, Remark~\ref{rem:delta_formula} and 
Theorem \ref{theor:delta} the theorem stated below easily follows.
\begin{theorem}\label{theor:inf_aut_grpd}
  Let $(\mathcal W \rightrightarrows E; G \rightrightarrows M)$ be a
  VB-groupoid with VB-algebroid $(W \Rightarrow E; A \Rightarrow M)$.
  Every infinitesimal automorphism $X_\mathcal W$ of $\mathcal W$ determines a
  triple $(\Delta_{\widehat A}, \Delta_E, \Delta_C)$ as in Theorem
  \ref{theor:delta}, via the following formulas
  \begin{equation} 
    \overrightarrow{\Delta_{\hat A} \tilde a} 
    = 
    [X_\mathcal W, \overrightarrow{\tilde a}],
    \quad 
    X_{\Delta_E}  
    = 
    \tilde s_\ast X_\mathcal W 
    = \tilde t_\ast X_\mathcal W
    \quad \text{and} \quad 
    X_{\Delta_C} 
    = 
    X_\mathcal W |_C,  
  \end{equation}
  or, equivalently,
  \begin{equation}
    \overrightarrow{\Delta}_{{\Delta_{\hat A} \tilde a}} 
    = 
    [\Delta_{\mathcal W}, \overrightarrow{\Delta}_{\tilde a}], \quad\Delta_E 
    = 
    \tilde s_\ast \Delta_\mathcal W 
    = 
    \tilde t_\ast \Delta_\mathcal W,
      \quad \text{and} \quad
    \Delta_C 
    = 
    \Delta_\mathcal W |_C, 
  \end{equation}
  for all $\tilde a \in \Gamma (\widehat A)$.  If $\mathcal W
  \rightrightarrows E$ is source simply connected, this correspondence
  is one-to-one.
\end{theorem}
Consider $X_\mathcal W$ to be an infinitesimal automorphism of the
VB-groupoid $(\mathcal W \rightrightarrows E; G \rightrightarrows M)$
and $X_W$ the infinitesimal automorphism of $(W \Rightarrow E; A
\Rightarrow M) = \mathsf{Lie} (\mathcal W \rightrightarrows E; G
\rightrightarrows M)$ associated to $X_\mathcal W$ by the Lie
functor. Finally, let $(\Delta_{\widehat A}, \Delta_E, \Delta_C)$ be
the algebraic data corresponding to $X_W$ via Theorem
\ref{theor:delta}. The above theorem then states that derivation
$\Delta_E$ (resp.~$\Delta_C$) agrees with that of Remark
\ref{rem:inf_automorphism}.

\subsection{The linear deformation complex of a VB-groupoid/algebroid}
\label{sec:linear_def}

Multiplicative vector fields on a Lie group $G$ can be seen as
$1$-cocycles in the Lie group cohomology complex of $G$ with
coefficients in the adjoint representation.  More generally,
multiplicative vector fields on a Lie groupoid $G \rightrightarrows M$
are $1$-cocycles in the \emph{deformation complex} of $G$
\cite{crainic2015deformations}.  At the infinitesimal level, IM vector
fields on a Lie algebroid $A \Rightarrow M$ are $1$-cocycles in the
\emph{deformation complex} of $A$ \cite{Crainic2008}.  Additionally,
the deformation complex of a Lie groupoid and the deformation complex
of the associated Lie algebroid are intertwined by a Van Est map
\cite{crainic2015deformations}, which is yet another manifestation of
the Lie functor.  In this section we show that there is a natural
cochain complex $C^\bullet_{\mathrm{def}, \mathrm{lin}}(\mathcal W)$
(resp.~$C^\bullet_{\mathrm{def}, \mathrm{lin}}(W)$) attached to every
VB-groupoid $(\mathcal W \rightrightarrows E; G \rightrightarrows M)$
(resp.~VB-algebroid $(W \Rightarrow E; A \Rightarrow M)$) in such a
way that multiplicative (resp.~IM) derivations are the same as
$1$-cocycles, while internal derivations are equivalent to
$1$-coboundaries.  Specifically, $C^\bullet_{\mathrm{def},
  \mathrm{lin}} (\mathcal W)$ (resp.~$C^\bullet_{\mathrm{def},
  \mathrm{lin}} (W)$) is a subcomplex of the deformation complex of
the Lie groupoid $\mathcal W \rightrightarrows E$ (resp.~Lie algebroid
$W \to E$) consisting of cochains that are \emph{linear} in a suitable
sense (see below).  Additionally we show that the Van Est map
restricts to a cochain map intertwining $C^\bullet_{\mathrm{def},
  \mathrm{lin}} (\mathcal W)$ and $C^\bullet_{\mathrm{def},
  \mathrm{lin}} (W)$.  Notice that this is similar to what Cabrera and
Drummond do in \cite{CD2017}. There is one main
difference though. While they work with the Lie groupoid/algebroid
cohomology of the total space of a VB-groupoid/algebroid, we work with
the deformation complex. In this respect we recall that the
deformation complex $C_{\mathrm{def}}^\bullet (\mathcal W)$ of a Lie
groupoid $\mathcal W \rightrightarrows E$ with Lie algebroid $W
\Rightarrow E$, can be seen as a subcomplex in the Lie groupoid
complex $C^\bullet (T^\ast \mathcal W)$ of the cotangent groupoid
$T^\ast \mathcal W \rightrightarrows W^\ast_E$
\cite{crainic2015deformations} (where $W^\ast_E \to E$ is the dual of
the vector bundle $W \to E$).  Specifically, $C_{\mathrm{def}}^\bullet
(\mathcal W)$ is the subcomplex consisting of cochains that are
\emph{linear} with respect to the vector bundle structure $T^\ast
\mathcal W \to \mathcal W$, plus an additional projectability
condition. If $\mathcal W$ is the total space of a VB-groupoid, then
$T^\ast \mathcal W$ is actually a \emph{triple structure}: a double
vector bundle in the category of groupoids. This means that the
results in this sections are analogous to those in \cite{CD2017} but
for groupoids equipped with two (not just one) compatible vector
bundle structures.

Let us recall how is the deformation complex defined. 
Let $G \rightrightarrows M$ be a Lie groupoid with Lie algebroid
$A$. Following \cite{crainic2015deformations}, we denote by $\bar{m}$
the division map in $G$, i.e.
\begin{equation}
  \bar{m}(g,h) 
  = 
  gh^{-1}
\end{equation}
for all $g,h \in G$ such that $s(g) = s(h)$.  For any $k\geq 0$, we
also denote by $G^{(k)}$ the space of $k$-tuples of composable
arrows. In particular, $G^{(0)}=M$ and $G^{(1)}=G$.  The
\emph{deformation complex} of $G$ \cite{crainic2015deformations} is
the pair $(C^\bullet_{\mathrm{def}}(G), \delta)$, where, for $k >0$,
$C^k_{\mathrm{def}}(G)$ is the space of smooth maps
\begin{equation}
  c 
  : 
  G^{(k)} \longrightarrow TG, 
  \quad 
  (g_1, \ldots, g_k) \longmapsto c (g_1, \ldots, g_k) \in T_{g_1} G,
\end{equation}
which are \emph{source-projectable}, i.e.~there exists a (necessarily
unique) smooth map $\sigma_{c} : G^{(k-1)} \longrightarrow TM, \quad
(g_1, \ldots, g_{k-1}) \longmapsto\sigma_c (g_1, \ldots, g_{k-1}) \in
T_{t(g_1)}M$, called the \emph{source projection} of $c$, such that $
(\D s \circ c )(g_1, g_2, \ldots, g_k) = \sigma_{c} (g_2, \ldots,
g_k)$.  So $1$-cochains are exactly source-projectable vector fields
on $G$.  The differential $\delta c$ of $c \in C^k_{\mathrm{def}} (G)$
is defined by
\begin{align}
  (\delta c)(g_1, \ldots, g_{k+1}) 
  = 
  & - \D \bar{m} (c(g_1g_2, g_3, \ldots, g_{k+1}), c(g_2, g_3, \ldots, g_{k+1})) 
  \\
  & + \sum_{i = 2}^{k} (-)^{i} c (g_1, \ldots, g_i g_{i+1}, \ldots, g_{k+1}) + (-)^{k+1} c (g_1, \ldots, g_k).
\end{align}
For $k = 0$, $C^0_{\mathrm{def}}(G)\coloneqq \Gamma (A)$ and for $a
\in \Gamma(A)$, $\delta a\coloneqq \overrightarrow a + \overleftarrow
a$. The terminology is motivated by the fact that $(
C{}^\bullet_{\mathrm{def}}(G), \delta)$ controls deformations of $G
\rightrightarrows M$ \cite{crainic2015deformations}.  It is easy to
see that $1$-cocycles in $( C{}^\bullet_{\mathrm{def}}(G), \delta)$
are the same as multiplicative vector fields on $G$ (see
\cite[Proposition 4.3]{crainic2015deformations}).
We also consider the \emph{normalized deformation complex} of $G$ as
the subcomplex $\widehat C{}^\bullet_{\mathrm{def}}(G) \subset
C{}^\bullet_{\mathrm{def}}(G)$ consisting of \emph{normalized
  cochains}. For $k \geq 2$, a normalized $k$-cochain is a $k$-cochain
$c$ such that for all $x \in M$
\begin{equation}
  c(1_x, g_2, \ldots, g_k) 
  = 
  \sigma_c (g_2, \ldots, g_k) 
  \quad \text{and} \quad 
  c(g_1, g_2, \ldots, 1_x, \ldots, g_k) 
  = 
  0.
\end{equation}
In degree $1$, the only condition is that $c(1_x) = (\sigma_c)_x$, and
in degree $0$ there are no conditions.

Now, let $(\mathcal W \rightrightarrows E; G \rightrightarrows M)$ be a
VB-groupoid with VB-algebroid $(W\Rightarrow E; A\Rightarrow
M)$. Consider the subspace $C^\bullet_{\mathrm{def, lin}} (\mathcal W)$ of
$C^\bullet_{\mathrm{def}} (\mathcal W)$ defined as follows.  For $k > 0$,
$C^k_{\mathrm{def, lin}} (\mathcal W)$ consists of those cochains $\tilde
c\in C^k_{\mathrm{def}}(\mathcal W)$ which are \emph{linear}, i.e.~$\tilde
c : \mathcal W^{(k)} \to T \mathcal W$ is a vector bundle morphism covering a
(necessarily unique) smooth map $c : G^{(k)} \to TG$:
\begin{equation}
\begin{array}{c}
    \begin{tikzcd}
    \mathcal W^{(k)}\arrow[r, "\tilde c"]\arrow[d]& T \mathcal W \arrow[d] \\
    G^{(k)}\arrow[r, "c"]& TG
  \end{tikzcd}
  \end{array}.
\end{equation}
In particular $C^1_{\mathrm{def, lin}}(\mathcal W)$ consists of those
source-projectable vector fields on the Lie groupoid $\mathcal W$ which
are also infinitesimal automorphisms of the vector bundle $\mathcal W\to
G$.  For $k = 0$, $C^0_{\mathrm{def}, \mathrm{lin}} (\mathcal W): =
\Gamma_\ell (W, E) = \Gamma (\widehat A)$, where $\widehat
A\Rightarrow M$ is the fat algebroid of $(W \Rightarrow E; A
\Rightarrow M)$.
\begin{lemma}\label{lem:C_def_lin} \quad
  \begin{definitionlist}
  \item $C^\bullet_{\mathrm{def, lin}}(\mathcal W)$ is a subcomplex of
    $(C^\bullet_{\mathrm{def}} (\mathcal W), \delta)$.
  \item If $\tilde c : \mathcal W^{(k)} \to T \mathcal W$ belongs 
    to $C^k_{\mathrm{def, lin}}(\mathcal W)$, then its \emph{projection} 
    $c: G^{(k)} \to TG$ belongs to $C^k_{\mathrm{def}}(G)$ and 
    $\delta \tilde c$ projects to $\delta c$.
  \end{definitionlist}
\end{lemma}
\begin{proof}
  Let $\tilde a \in C^0_{\mathrm{def}, \mathrm{lin}} (\mathcal W)$ be
  a linear section of $W \to E$. Then $\delta \tilde a =
  \overrightarrow{\tilde a} + \overleftarrow{\tilde a}$ is an
  infinitesimal automorphism of $(\mathcal W \rightrightarrows E; G
  \rightrightarrows M)$ (Remark \ref{rem:internal}). In particular, it
  belongs to $C^1_{\mathrm{def, lin}}(\mathcal W)$.
  Now, for $k > 0$, let $\tilde c : \mathcal W^{(k)} \to T\mathcal W$
  belong to $ C^k_{\mathrm{def, lin}} (\mathcal W)$, and show that
  $\delta \tilde c $ belongs to $C^{k+1}_{\mathrm{def, lin}} (\mathcal
  W)$. This follows, after a straightforward computation, from the
  linearity of $\tilde c$ and the obvious fact that like the (other)
  structure maps, the division $\bar{\tilde m} = \mathcal
  W\mathbin{{}_{\tilde s}\times_{\tilde s}} \mathcal W \to \mathcal W$ in a
  VB-groupoid $(\mathcal W \rightrightarrows E; G \rightrightarrows
  M)$ is a vector bundle morphism (covering the division $\bar{m} :
  G\mathbin{{}_s \times_s} G \to G$ in $G$).
  The second part of the statement immediately follows from the axioms
  of VB-groupoids and the, easy to check, commutativity of the
  following diagram:
  \begin{equation}
  \begin{array}{c}
    \begin{tikzcd}
      & \mathcal W^{(k)} \arrow[rr, "\mathrm{d}\tilde{s}\circ\tilde{c}"]\arrow[dd]\arrow[dl] & & TE \arrow[dd]\\
      \mathcal W^{(k-1)} \arrow[dd] \arrow[urrr, swap, crossing over, "\sigma_{\tilde c}"] & & & \\
      & G^{(k)} \arrow[ld] \arrow[rr, "\mathrm{d}s\circ c"] & &TM \\
      G^{(k-1)} \arrow[urrr, swap, "\sigma_c"]& & &
    \end{tikzcd}
    \end{array}.
  \end{equation}
  Here the projection $\mathcal W^{(k)} \to \mathcal W^{(k-1)}$ consists in
  dropping the first entry (and similarly for the projection $G^{(k)} \to
  G^{(k-1)}$) and $\sigma_c = \D q \circ \sigma_{\tilde c} \circ
  0_{\mathcal W^{(k-1)}}$, where $0_{\mathcal W^{(k-1)}} : G^{(k-1)} \to
  \mathcal W^{(k-1)}$ is the zero section.
\end{proof}
\begin{proposition}
  \label{prop:1-cocycles_grpd} \quad
  \begin{propositionlist}
  \item Infinitesimal automorphisms of $(\mathcal W \rightrightarrows
    E; G \rightrightarrows M)$ are $1$-cocycles in
    $(C^\bullet_{\mathrm{def, lin}}(\mathcal W), \delta)$.
  \item Internal derivations of $(\mathcal W \rightrightarrows E; G
    \rightrightarrows M)$ are (equivalent to) $1$-coboundaries in
    $(C^\bullet_{\mathrm{def, lin}}(\mathcal W), \delta)$.
  \end{propositionlist}
\end{proposition}
\begin{proof}
  It immediately follows from Remark \ref{rem:internal} and
  \cite[Proposition 4.3]{crainic2015deformations}.
\end{proof}
We now pass to the infinitesimal picture.  Recall from
\cite{Crainic2008} that the \emph{deformation complex} of a Lie
algebroid $A \Rightarrow M$ is the pair $(C^\bullet_{\mathrm{def}}(A),
\delta)$, where, for $k > 0$, $C^k_{\mathrm{def}} (A)$ consists of
(skew-symmetric) multiderivations
\begin{equation}
  c : \Gamma (A) \times \cdots \times \Gamma (A) \longrightarrow \Gamma (A)
\end{equation}
with $k$-entries, and multilinear symbol, i.e.~$c$ is a skew-symmetric
$\mathbb R$-multilinear map, and there exists a (necessarily unique)
vector bundle morphism $\sigma_c : \wedge^{k-1}A \to TM$, called the
\emph{symbol} of $c$, such that
\begin{equation}
  c(a_1, \ldots, a_{k-1}, f a_{k}) 
  = 
  f c(a_1, \ldots, a_{k-1}, a_{k})+ \sigma_c (a_1, \ldots, a_{k-1})(f) a_{k},
\end{equation}
for all $a_1, \ldots, a_k \in \Gamma (A)$ and $f \in C^\infty (M)$.
In particular $1$-cochains are just derivations of the vector bundle
$A\to M$.  The differential $\delta c$ of $c \in C^k_{\mathrm{def}}
(A)$ is defined by
\begin{equation}
  \label{eq:partial}
  \begin{aligned}
    (\delta c)(a_1, \ldots, a_{k+1}) 
    = {}
    & \sum_{i=1}^{k+1}(-)^{i+1} [a_i, c(a_1, \ldots, \widehat a_{i}, \ldots, a_{k+1})]
    \\
    & + \sum_{i<j} (-)^{i+j}c ([a_i, a_j], a_1, \ldots, \widehat a_i, \ldots, \widehat a_j, \ldots, a_{k+1}),
  \end{aligned}
\end{equation}
for all $a_1, \ldots, a_{k+1} \in \Gamma (A)$. 
A straightforward computation then shows that
\begin{equation}
  \label{eq:symbol_partial}
  \begin{aligned}
    \sigma_{\delta c} (a_1, \ldots, a_k)={} 
    & 
    \sum_{i = 1}^k (-)^{i+1} [\rho (a_i), \sigma_c (a_1, \ldots, \widehat a_{i}, \ldots, a_{k})]
    \\
    & +
    \sum_{i < j} (-)^{i+j} \sigma_c ([a_i, a_j],a_1, \ldots, \widehat a_i, \ldots, \widehat a_j, \ldots, a_{k}) - (-)^k \rho (c(a_1, \ldots, a_k)).
  \end{aligned}
\end{equation}
for all $a_1, \ldots, a_k \in \Gamma (A)$.
For $k = 0$, $C^{0}_{\mathrm{def}}(A) \coloneqq \Gamma(A)$ and for
$a \in \Gamma (A)$, $\delta a \coloneqq [a, -]$. The terminology is
motivated by the fact that $(C^\bullet_{\mathrm{def}}(A), \delta)$
controls deformations of $A$ \cite{Crainic2008} (see also
\cite[Section 1.1]{LPV20XX}). We remark, however, that the complex
$(C^\bullet_{\mathrm{def}}(A), \delta)$ first appeared in
\cite{GGU2003} under a different name, and for different purposes (see
Theorem 6 loc.~cit.).

It is immediate that $1$-cocycles in $(C^\bullet_{\mathrm{def}}(A),
\delta)$ are the same as Lie algebroid derivations of $A$
(cf.~\cite[Section 3.1]{Crainic2008}).  Hence they are in one-to-one
correspondence with IM vector fields on $A$.
\begin{remark}
  \label{rem:SJ_bracket}
  The graded space $C^\bullet_{\mathrm{def}} (A)$ is not just a
  cochain complex.  Actually, up to a shift, it becomes a dg-Lie
  algebra when additionally equipped with the \emph{Schouten-Jacobi
    bracket} $[\![-,-]\!]$ (see \cite{LTV16, LOTV14}), originally
  named the \emph{Gerstenhaber bracket} in
  \cite{Crainic2008}. Differential $\delta$ is then $[\![ b, -]\!]$,
  where $b = [-,-]$ is the Lie bracket on $\Gamma(A)$.
\end{remark}
When $A$ is the Lie algebroid of a Lie groupoid $G \rightrightarrows
M$, then the deformation complex of $A$ and the normalized deformation
complex of $G$ are intertwined by a \emph{Van Est map}
\cite{crainic2015deformations}, i.e.~the cochain map
\begin{equation}
  \label{eq:Van_Est}
  \mathcal V : (\widehat C{}^\bullet_{\mathrm{def}}(G), \delta) 
  \longrightarrow (C{}^\bullet_{\mathrm{def}}(A), \delta) 
\end{equation}
defined by
\begin{equation}
  (\mathcal V c)(a_1, \ldots, a_k) 
  := \sum_{\tau \in S_k} (-)^\tau (R_{a_{\tau (1)}} 
  \circ \cdots \circ R_{a_{\tau (k)}})(c),
\end{equation}
for all $c\in\widehat C^k_{\mathrm{def}}(G)$ and $a_1, \ldots, a_k\in\Gamma(A)$. Here, for any $a \in \Gamma(A)$, 
\begin{equation}
R_a :
\widehat C{}^{\bullet+1}_{\mathrm{def}}(G) \to \widehat
C{}^{\bullet}_{\mathrm{def}}(G)
\end{equation}
 is the map defined as follows. For $c
\in \widehat C{}^1_{\mathrm{def}}(G)$, $R_a c \coloneqq [c,
  \overrightarrow a]|_M$ and for $c \in \widehat
C{}^{k+1}_{\mathrm{def}}(G)$, $k > 0$,
\begin{equation}
  \label{eq:R_a}
  (R_a c)(g_1, \ldots, g_{k}) :=
  (-)^k\left.\frac{d}{d\varepsilon}\right|_{\varepsilon = 0} c (g_1, \ldots,
  g_k, \phi_{\varepsilon}^a (s(g_k))^{-1}),
\end{equation}
where $\{\phi_{\varepsilon}^a\}$ is the flow of
$\overrightarrow a$.
Now, let $(W \Rightarrow E; A \Rightarrow M)$ be a
VB-algebroid. Consider the subspace $C^\bullet_{\mathrm{def},
  \mathrm{lin}} (W)$ of $C^\bullet_{\mathrm{def}} (W)$ \color{black} defined as follows.
For $k > 0$, $C^k_{\mathrm{def,lin}}(W)$ consists of
\emph{linear cochains}, i.e.~those multiderivations $\tilde c : \Gamma (W, E)
\times \cdots \times \Gamma (W, E) \to \Gamma (W,E)$ such that
 \begin{gather}
  \tilde c(\tilde a_1, \ldots, \tilde a_k)\text{ is
      a linear section}, \label{eq:c_0} \\
    \tilde c (\tilde a_1, \ldots, \tilde a_{k-1}, \widehat \chi_1) \text{ is
      a core section}, \label{eq:c_1}
    \\ 
    \tilde c (\tilde a_1, \ldots, \tilde a_{k-i}, \widehat \chi_1, \ldots, \widehat \chi_i) = 0, \label{eq:c_2}
  \end{gather}
 and, additionally 
  \begin{gather}
   \sigma_{\tilde c} (\tilde a_1, \ldots, \tilde a_{k-1}) \text{ is a linear vector field,} \label{eq:s_0} \\
    \sigma_{\tilde c} (\tilde a_1, \ldots, \tilde a_{k-2}, \widehat \chi_1) \text{ is the vertical lift of a section of $E$,} \label{eq:s_1}\\ 
      \sigma_{\tilde c}(\tilde a_1, \ldots, \tilde a_{k-i-1}, \widehat \chi_1, \ldots, \widehat \chi_i) = 0, \label{eq:s_2}
  \end{gather}
  for all linear sections $\tilde a_1, \ldots, \tilde a_k$, all
  core sections $\widehat \chi_1, \ldots, \widehat \chi_i$ of
  $W \to E$, and all $i \geq 2$.
  For $k=0$, 
we define $C^0_{\mathrm{def,lin}}(W) := \Gamma_{\ell}(W,E)=\Gamma(\widehat{A})$. 
\color{black}
\begin{lemma}
  $C^\bullet_{\mathrm{def}, \mathrm{lin}}(W)$ is a subcomplex, and a
  Lie subalgebra, in $C^\bullet_{\mathrm{def}}(W)$.
\end{lemma}
\begin{proof}
  It immediately follows from the explicit formula for $\delta$,
  Definition~\ref{def:VBAlgebroid}, and the explicit formula for the
  Schouten-Jacobi bracket (see \cite[Proposition 1]{Crainic2008}).
\end{proof}
Linear $0$-cochains in $C^\bullet_{\mathrm{def}} (W)$ are already
characterized by Lemma \ref{lem:linear_Euler}.  For higher degree
cochains the situation is illustrated by the following
\begin{lemma}
  \label{lem:linear_Euler_extended}
  For any $k\geq 0$, a cochain $\tilde c \in C^k_{\mathrm{def}} (W)$ is linear if and only if
  $[\![\Delta_{\mathcal E}, \tilde c]\!] = 0$, where $[\![-,-]\!]$ is
  the Schouten-Jacobi bracket (see Remark \ref{rem:SJ_bracket}).
\end{lemma}
\begin{proof}
  First recall that a multiderivation is completely determined by its
  own action and its symbol action on generators. From the explicit
  formula for the Schouten--Jacobi bracket \cite[Proposition
    1]{Crainic2008} it follows that
  \begin{equation}
    \label{eq:Delta_c}
          [\![ \Delta_{\mathcal E}, \tilde c]\!] (w_1, \ldots, w_k) 
          =  \Delta_{\mathcal E} (\tilde c(w_1, \ldots, w_k))  
          - \sum_{i =1}^k \tilde c (w_1, \ldots, \Delta_{\mathcal E}w_i, \ldots, w_k),
  \end{equation}
  \color{black} and
  \begin{equation}
  \label{eq:symbol_Delta_c}
    \sigma_{[\![ \Delta_{\mathcal E}, \tilde c]\!]} (w_1, \ldots, w_{k-1}) 
    = [\sigma (\Delta_\mathcal E), \sigma_{\tilde c} (w_1, \ldots, w_{k-1})] 
    - \sum_{i=1}^{k-1}\sigma_{\tilde c} (w_1, \ldots, \Delta_{\mathcal E} w_i, \ldots, w_{k-1}),
  \end{equation}
  for all $w_1, \ldots, w_{k} \in \Gamma (W, E)$. Using
  \eqref{eq:Delta_c} and Lemma~\ref{lem:linear_Euler} we immediately
  see that if $\tilde c$ is linear, then $[\![\Delta_{\mathcal E},
      \tilde c]\!]$ and its symbol vanish on linear and core
  sections. Hence $[\![\Delta_{\mathcal E}, \tilde c]\!] = 0$.
  Conversely, suppose that $[\![\Delta_{\mathcal E}, \tilde c]\!] =
  0$. Then (\ref{eq:c_0}), (\ref{eq:c_1}), and (\ref{eq:c_2}) follow
  from (\ref{eq:Delta_c}) and Lemma~\ref{lem:linear_Euler}. Similarly,
  (\ref{eq:s_0}), (\ref{eq:s_1}), and (\ref{eq:s_2}) follow from
  (\ref{eq:symbol_Delta_c}), the fact that $\sigma (\Delta_{\mathcal
    E})$ is the Euler vector field on $E$ (cf.~Remark~\ref{rem:linear_VF})
  and that there are no non-trivial vector fields $X$ on $E$
  such that $[\sigma (\Delta_{\mathcal E}), X] = -i X$ for some
  integer $i > 1$ (cf.~Lemma~\ref{lem:linear_Euler}).
\end{proof}
 \color{black}

\begin{corollary}\label{cor:1-cocycles_alg} \quad
  \begin{propositionlist}
  \item Infinitesimal automorphisms of $(W \Rightarrow E; A \Rightarrow M)$ 
    are equivalent to $1$-cocycles in $(C^\bullet_{\mathrm{def, lin}}(W), \delta)$.
  \item Internal derivations of $(W \Rightarrow E; A \Rightarrow M)$ are equivalent to
    $1$-coboundaries in $(C^\bullet_{\mathrm{def, lin}}(W), \delta)$.
  \end{propositionlist} 
\end{corollary}
\begin{proof}
  It immediately follows from Remark~\ref{rem:internal_alg},
  Lemma~\ref{lem:linear_Euler_extended} and~\cite[Section
    3.1]{Crainic2008}.
\end{proof}
Now we discuss the analogue of Theorem \ref{theor:delta} for cochains
in $C^k_{\mathrm{def}, \mathrm{lin}}(W)$. Namely, let $\tilde c \in
C^k_{\mathrm{def}, \mathrm{lin}}$, $k > 1$. By definition, $\tilde c$
preserves linear sections, hence it determines by restriction a
$k$-cochain $c_{\widehat A}$ in the fat algebroid $\widehat A$.
\textcolor{black}{Additionally,} $\sigma_{\tilde c}$ maps linear
sections to linear vector fields and we set
\begin{equation}
  c_E (\tilde a_1, \ldots, \tilde a_{k-1}) 
  := 
  \Delta_{\sigma_{\tilde c} (\tilde a_1, \ldots, \tilde a_{k-1})},
\end{equation} 
for all $\tilde a_1, \ldots, \tilde a_{k-1} \in \Gamma (\widehat A)$
(recall that $\Delta_X$ denotes the derivation corresponding to a
linear vector field $X$). \textcolor{black}{Moreover,} ${\tilde
  c}$ maps $k-1$ linear sections and a core sections to core sections
and we set
\begin{equation}
  \widehat{c_C (\tilde a_1, \ldots, \tilde a_{k-1})(\chi)} 
  := 
  {\tilde c} (\tilde a_1, \ldots, \tilde a_{k-1}, \widehat \chi),
\end{equation}
for all $\tilde a_1, \ldots, \tilde a_{k-1} \in \Gamma (\widehat A)$
and $\chi \in \Gamma (C)$ (recall that $\widehat \chi$ is the core
section of $W \to E$ corresponding to
$\chi$). \textcolor{black}{Finally,} we define $\mathrm{d}_{\widehat A}$ by
\begin{equation}
    \mathrm{d}_{\widehat A} (\tilde a_1, \ldots, \tilde a_{k-2}) (\chi)^{\uparrow} 
    := 
    \sigma_{\tilde c} (\tilde a_1, \ldots, \tilde a_{k-2}, \widehat \chi),
\end{equation}
(recall that $e^{\uparrow}$ denotes the vertical lift of a section $e
  \in \Gamma(E)$).  For $k = 1$, a cochain $\tilde c \in
  C^1_{\mathrm{def}, \mathrm{lin}}(W)$ is a derivation of $W \to E$
  with the additional property of preserving linear sections. It is
  then obvious that the corresponding linear vector field $X_{\tilde
    c}$ on $W$ is an infinitesimal automorphism of the double vector
  bundle $(W \rightarrow E; A \rightarrow M)$. In particular, we can
  associate to $X_{\tilde c}$ a triple $(\Delta_{\widehat A},
  \Delta_E, \Delta_C)$ exactly as we did in Theorem \ref{theor:delta},
  except for the fact that as $X_{\tilde c}$ is not an
  infinitesimal automorphism of the VB-algebroid in general,
  $(\Delta_{\widehat A}, \Delta_E, \Delta_C)$ do only satisfy
  Equations (\ref{eq:sigma_delta}) and (\ref{eq:delta_hom}), and do
  not satisfy (\ref{eq:core-anchor}), (\ref{eq:delta_psi}) and
  (\ref{eq:delta_psi_c}) in general (see also Remark
  \ref{rem:dvb_vs_vb_alg}). In this section we will denote by
  $(c_{\widehat A}, c_E, c_C)$ the triple $(\Delta_{\widehat A},
  \Delta_E, \Delta_C)$ to stress that it comes from a (non-necessarily
  closed) $1$-cochain in $C^\bullet_{\mathrm{def}, \mathrm{lin}}(W)$.
\begin{theorem}
  \label{theor:linear_cochains_alg} \quad
  \begin{definitionlist}
  \item For $k>1$, the assignment $\tilde c \mapsto (c_{\widehat A}, c_E,
    c_C, \mathrm{d}_{\widehat A}) $ establishes a one-to-one correspondence
    between cochains in $C^k_{\mathrm{def}, \mathrm{lin}}(W)$ and
    $4$-tuples consisting of
    \begin{itemize}
    \item a $k$-cochain $c_{\widehat A}$ in the deformation complex of
      the fat algebroid $\widehat A$,
    \item a vector bundle morphism $c_E : \wedge^{k-1} \widehat A \to
      \der E$,
      \item a vector bundle morphism $c_C : \wedge^{k-1} \widehat A \to \der C$ and
      \item a vector bundle morphism $\mathrm{d}_{\widehat A} : \wedge^{k-2} \widehat A \to \Hom (C,E)$, such that
    \end{itemize}
    \begin{equation}
      \label{eq:derivation_valued}
      \sigma \circ c_E 
      = 
      \sigma \circ c_C = \sigma_{c_{\widehat A}} 
    \end{equation}
    where $\sigma : \der E \to TM$ (resp.~$\sigma : \der C \to TM$) is the symbol map and additionally
    \begin{align}
      c_{\widehat A} (\tilde a_1, \ldots, \tilde a_{k-1}, \Phi) 
      & = 
      c_C (\tilde a_1, \ldots, \tilde a_{k-1}) \circ \Phi - \Phi \circ c_E (\tilde a_1, \ldots, \tilde a_{k-1}), \label{eq:c_A}
      \\
      c_E (\tilde a_1, \ldots, \tilde a_{k-2}, \Phi) 
      & = 
      - \mathrm{d}_{\widehat A} (\tilde a_1, \ldots, \tilde a_{k-2}) \circ \Phi, 
      \\
      c_C (\tilde a_1, \ldots, \tilde a_{k-2}, \Phi) 
      & = 
      - \Phi \circ \mathrm{d}_{\widehat A} (\tilde a_1, \ldots, \tilde a_{k-2}), \label{eq:c_C}
      \\
      \mathrm{d}_{\widehat A} (\tilde a_1, \ldots, \tilde a_{k-3}, \Phi) 
      & = 0 
      \label{eq:d},
    \end{align} 
    for all $\tilde a_1, \ldots, \tilde a_{k-1} \in \Gamma (\widehat
    A)$ and all $\Phi \in \Gamma (\Hom (E,C))$.
   
   If $\tilde c$ corresponds to $(c_{\widehat A}, c_E, c_C,
   \mathrm{d}_{\widehat A})$, then $\delta \tilde c$ corresponds to
   $(\delta c_{\widehat A}, c_E', c_C',\mathrm{d}'_{\widehat A})$ where
    \begin{equation}\label{eq:c_Eprime}
      \begin{aligned}
        c'_E (\tilde a_1, \ldots, \tilde a_k) = {} & 
        \sum_{i = 1}^k (-)^{i+1}[\psi^s_{\tilde a_i}, c_E (\tilde a_1, \ldots, \widehat{\tilde a_i}, \ldots, \tilde a_k)] 
        \\
        & + 
        \sum_{i<j}(-)^{i+j}c_E ([\tilde a_i, \tilde a_j], \tilde a_1, \ldots, \widehat{\tilde a_i}, \ldots, \widehat{\tilde a_j}, \ldots, \tilde a_k)
       - (-)^k \psi^s_{c_{\widehat A}(\tilde a_1, \ldots, \tilde a_k)},
      \end{aligned}
    \end{equation}
    and like-wise for $c'_C$ and 
    \begin{equation}\label{eq:d_Aprime}
      \begin{aligned}
        &\mathrm{d}'_{\widehat A} (\tilde a_1, \ldots, \tilde a_{k-1}) \\
        = 
        & \sum_{i = 1}^{k-1} (-)^{i+1} \left( \psi^s_{\tilde{a}_i} \circ \mathrm{d}_{\widehat A} (\tilde a_1, \ldots, \widehat{\tilde a_i}, \ldots, \tilde a_{k-1})
        - \mathrm{d}_{\widehat A} (\tilde a_1, \ldots, \widehat{\tilde a_i}, \ldots, \tilde a_{k-1}) \circ \psi^c_{\tilde{a}_i} \right) \\
        & 
        + \sum_{i<j} (-)^{i+j} \mathrm{d}_{\widehat A}([\tilde a_i, \tilde a_j], \tilde a_1, \ldots, \widehat{\tilde a_i}, \ldots,
        \widehat{\tilde a_j}, \ldots, \tilde a_{k-1}) \\
        & + (-)^k \left( c_E(\tilde a_1, \ldots, \tilde a_{k-1}) \circ \partial - \partial \circ c_C (\tilde a_1, \ldots, \tilde a_{k-1})\right),
      \end{aligned}
    \end{equation}
    for all $\tilde a_1, \ldots, \tilde a_k \in \Gamma (\widehat A)$.
  \item For $k=1$, the assignment $\tilde c \mapsto (c_{\widehat A},
    c_E, c_C)$ establishes a one-to-one correspondence between
    $1$-cochains in $C^\bullet_{\mathrm{def,lin}}(W)$ and triples
    $(c_{\widehat A},c_E,c_{C})$, where $c_{\widehat A}$, $c_E$, $c_C$
    are derivations of the vector bundles $\widehat A$, $E$, $C$
    respectively, such that
    \begin{equation}
      \sigma(c_E)=\sigma(c_C)=\sigma(c_{\widehat A})
    \end{equation}
    and 
    \begin{equation}
    c_{\widehat A}(\Phi)=c_C \circ\Phi-\Phi\circ c_E
    \end{equation}
    for all $\Phi\in\Gamma(\Hom(E,C))$.  If $\tilde c$ corresponds to
    $(c_{\widehat A},c_E,c_C)$, then $\delta \tilde c\in
    C^2_{\mathrm{def}, \mathrm{lin}}(W)$ corresponds to the $4$-tuple
    $(\delta c_{\widehat A},c'_E, c'_C,\mathrm{d}'_{\widehat A})$, where
    \begin{equation}
    \mathrm{d}'_{\widehat A}=\partial \circ c_C-c_E\circ\partial
    \end{equation}
    and for all $\tilde a\in\Gamma(\widehat{A})$,
    \begin{equation}
      c'_E(\tilde a)
      =
      [\psi^s_{\tilde a},c_E]+\psi^s_{c_{\widehat A}(\tilde a)},\qquad c'_C(\tilde a)=[\psi^c_{\tilde a},c_C]+\psi^c_{c_{\widehat A}(\tilde a)}.
    \end{equation}
  \item For $k = 0$, the $0$-cochains $\tilde c$ in
    $C^\bullet_{\mathrm{def,lin}}(W)$ are the same as sections of the
    fat algebroid $\widehat A \to M$.  If $c\in C^0_{\mathrm{def},
      \mathrm{lin}}(W)$ is given by $\tilde a \in \Gamma (\widehat
    A)$, then $\delta c$ corresponds to the triple $(\delta \tilde a,
    \psi^s_{\tilde a}, \psi^c_{\tilde a})$.
  \end{definitionlist}
\end{theorem}
\begin{proof} 
  The second and the third claims are straightforward. Let us prove
  the first one.  Let $\tilde c \in C^k_{\mathrm{def},
    \mathrm{lin}}(W)$. It is clear from the definition that
  $c_{\widehat A}$ is a multiderivation, both $c_E$ and $d_{\widehat
    A}$ are multi-linear and~\eqref{eq:derivation_valued} holds. In
  order to prove the multi-linearity of $c_C$ (in its first $k-1$
  arguments), let $f \in C^\infty (M)$ and compute
  \begin{equation}
    {\tilde c}(\tilde a_1, \ldots \tilde a_{k-2}, f \tilde a_{k-1}, \widehat \chi) 
    = 
    f {\tilde c}(\tilde a_1, \ldots \tilde a_{k-2}, \tilde a_{k-1}, \widehat \chi) 
    -
    \sigma_{\tilde c} (\tilde a_1, \ldots \tilde a_{k-2},\widehat \chi)(f) \tilde a_{k-1}.
  \end{equation}
  Since $\sigma_{\tilde c} (\tilde a_1, \ldots \tilde a_{k-2},\widehat
  \chi)$ is a vertical vector field on $E$, the last summand vanishes
  and we get $C^\infty (M)$-linearity in $\tilde a_{k-1}$.
  Now, we check \eqref{eq:c_A}--\eqref{eq:d}. It is enough to consider
  $\Phi$ of the form $\Phi = \varphi \otimes \chi$, where $\varphi \in
  \Gamma (E^\ast)$ and $\chi\in \Gamma (C)$. Hence $\Phi$ corresponds
  to the linear section $\varphi \widehat \chi$, where, as usual, we
  interpret $\varphi$ as a fiber-wise linear function on $E$, and
  \eqref{eq:c_A}--\eqref{eq:c_C} easily follow from either the Leibniz
  rule for ${\tilde c}$ or the multi-linearity of $\sigma_{\tilde c}$,
  while \eqref{eq:d} follows from \eqref{eq:s_2}.
  As already recalled, a multiderivation $c$ of a vector bundle
  is completely determined by its action and the action of $\sigma_c$
  on generators of the module of sections of the vector bundle.  As
  linear and core sections generate $\Gamma(W, E)$ as a $C^\infty
  (E)$-module, then the data $(c_{\widehat A}, c_E, c_C, \mathrm{d}_{\widehat A})$
  determine ${\tilde c} \in C^\bullet_{\mathrm{def}, \mathrm{lin}}(W)$
  completely, i.e.~the correspondence ${\tilde c} \mapsto (c_{\widehat A},
  c_E, c_C, \mathrm{d}_{\widehat A})$ is injective.  To see that every
  $4$-tuple $(c_{\widehat A}, c_E, c_C, \mathrm{d}_{\widehat A})$ as in the
  statement arises in this way, use local coordinates to show that
 ${\tilde c} \mapsto (c_{\widehat A}, c_E, c_C,
  \mathrm{d}_{\widehat A})$ is a(n injective) $C^\infty (M)$-linear map
  between modules of sections of two vector bundles over $M$ of the
  same rank.
   Finally, \eqref{eq:c_Eprime} and \eqref{eq:d_Aprime} follow from
  \eqref{eq:partial} and \eqref{eq:symbol_partial} by a
  straightforward computation.
\end{proof}

Let $(c_{\widehat A}, c_E, c_C, \mathrm{d}_{\widehat A})$ be a $4$-tuple as in
the statement of the above theorem. Then condition
\eqref{eq:derivation_valued} says that $c_E$ and $c_C$ are
\emph{derivation valued skew-symmetric $k$-forms on $\widehat A$} in
the sense of \cite[Section 2.4]{V2015}.  Additionally, it follows from
\eqref{eq:c_A} that $c_{\widehat A}$ descends to a $k$-cochain in the
deformation complex of $A$, i.e.
\begin{equation}
  c_A (\pi (-), \ldots, \pi (-)) = \pi \circ c_{\widehat A},
\end{equation}
for some $c_A \in C^k_{\mathrm{def}} (A)$, where $\pi : \widehat A \to
A$ is the projection.  Finally, from \eqref{eq:d}, $\mathrm{d}_{\widehat A}$
descends to a vector bundle morphism $\mathrm{d}_A : \wedge^{k-1} A \to \Hom
(C,E)$, i.e.
\begin{equation}
  \mathrm{d}_{\widehat A} = \mathrm{d}_A (\pi (-), \ldots, \pi (-)).
\end{equation}
We conclude this subsection showing that the deformation complex of a
VB-groupoid and the deformation complex of its VB-algebroid are
intertwined by a \emph{linear} Van Est map.
\begin{theorem}[Linear Van Est map] 
  Let $(\mathcal W \rightrightarrows E; G \rightrightarrows M)$ be a
  VB-groupoid with (associated) VB-algebroid $(W \Rightarrow E; A
  \Rightarrow M)$.  The Van Est map \eqref{eq:Van_Est} restricts to a
  cochain map
  \begin{equation}
    \mathcal{V} :  (\widehat C{}^\bullet_{\mathrm{def, lin}}(\mathcal W), \delta) \to (C^\bullet_{\mathrm{def,lin}} (W), \delta), 
  \end{equation}
  where $\widehat C{}^\bullet_{\mathrm{def, lin}}(\mathcal W)$ denotes
  \emph{linear, normalized cochains}, i.e.~$\widehat
  C{}^\bullet_{\mathrm{def, lin}}(\mathcal W) = \widehat
  C{}^\bullet_{\mathrm{def}}(\mathcal W) \cap C{}^\bullet_{\mathrm{def,
      lin}}(\mathcal W)$.
\end{theorem}

\begin{proof}
  \color{black} When the side bundle $E$ has positive rank the proof
  can be significantly simplified by noticing that in this case a
  cochain $\tilde c \in C^\bullet_{\mathrm{def}} (W)$ is linear if and
  only if it preserves linear sections. This means that if
  $\operatorname{rank} E > 0$, only condition (\ref{eq:c_0}) is needed
  in the definition of linear cochains, while conditions
  (\ref{eq:c_1})-(\ref{eq:s_2}) are redundant. This can be seen, for
  instance, in local coordinates. In the rest of this proof we assume
  $\operatorname{rank} E > 0$. A complete proof will be presented in
  \cite{LPV2019}.  \color{black}

  So, it is enough to show that if $\tilde a$ is a linear section of
  $W \to E$, then $R_{\tilde a}$ maps $\widehat
  C{}^{\bullet+1}_{\mathrm{def}, \mathrm{lin}} (\mathcal W)$ to
  $\widehat C{}^{\bullet}_{\mathrm{def}, \mathrm{lin}} (\mathcal
  W)$. But this immediately follows from \eqref{eq:R_a} and the fact
  that from linearity the right invariant vector field
  $\overrightarrow{\tilde a} \in \mathfrak{X}(\mathcal W)$ is a linear
  vector field on $\mathcal W$ (Proposition \ref{prop:internal}),
  hence it generates a flow $\{\phi_{\varepsilon}^{\tilde a}
  \}_{\varepsilon}$ by vector bundle automorphisms.
\end{proof}

\section{Multiplicative derivations of trivial-core VB-groupoids and algebroids}
\label{sec:trivial-core}

\subsection{VB-groupoids/algebroids and representations}


In this section we discuss the properties of some special class of
VB-groupoids/algebroids.  We have mentioned in the introduction that
VB-groupoids/algebroids are intrinsic models of 2-term representations
up to homotopy (see \cite{gracia2010lie,Gracia-Saz2010}, see also
Section \ref{sec:applications} for more details). For instance, the
isomorphism class of the tangent VB-groupoid/algebroid corresponds to
the isomorphism class of the adjoint representation (up to
homotopy). In this picture, 2-term representations up to homotopy
concentrated in one single degree ($-1$ or $0$) are genuine
representations and correspond to the class of VB-groupoids/algebroids
that we aim to discuss here.
\begin{definition}[Trivial-core VB-groupoid]
  \label{def:SourceTrivialGr}
  A VB-groupoid is \emph{trivial-core} if its core is trivial.
\end{definition}
Trivial-core VB-groupoids are also known as \emph{vacant}. Let
$(E_G\rightrightarrows E; G\rightrightarrows M)$ be a trivial-core
VB-groupoid. In this case, the source $\tilde s : E_G \to E$ is a
regular vector bundle morphism (covering $s : G \to M$). Hence $E_G$
is canonically isomorphic to $s^\ast E$ as a vector bundle over
$G$. In the following we will understand the latter isomorphism and
denote $E_G = s^\ast E$.

\begin{example}
  \label{ex:TrivCoreGrpd}
  Let $G \rightrightarrows M$ be a Lie groupoid and let $E \to M$ be a
  vector bundle equipped with a left representation of $G$. Then there
  is a trivial core VB-groupoid structure $(E_G = s^\ast E
  \rightrightarrows E; G\rightrightarrows M)$. The structure maps of
  the top groupoid $E_G \rightrightarrows E$ are the following:
  \begin{definitionlist}
  \item the source and target maps are given by $\tilde s
    (g, e) = e$ and $\tilde t (g, e) = g.e$, respectively,
  \item the multiplication $\tilde m$ is given by $\tilde m
    ((g, e), (g', e')) = (gg', e')$,
  \item the unit and the inversion are given by $\tilde u (e) =
    (1_{q(e)}, e)$ and $\tilde i (g, e) = (g^{-1}, g.e)$.
  \end{definitionlist} 
  In other words, $E_{G} \rightrightarrows E$ is the \emph{action
    groupoid} corresponding to the action of $G$ on the fibered
  manifold $E \to M$. It is easy to see that every trivial-core
  VB-groupoid arises in this way. In \cite{Gracia-Saz2010}
  trivial-core VB-groupoids $(E_G\rightrightarrows E;
  G\rightrightarrows M)$ correspond to representations of $G$ on $E$,
  seen as representations up to homotopy concentrated in degree 0.
\end{example}

\begin{definition}[Full-core VB-groupoid] 
  A VB-groupoid $(\mathcal W \rightrightarrows E; G \rightrightarrows
  M)$ is \emph{full-core} if its core is $C = \mathcal W |_M$.
\end{definition}
Let $(\mathcal W \rightrightarrows E; G \rightrightarrows M)$ be a
full-core VB-groupoid with core $C$. In this case, $E = 0_M$, the rank
$0$ vector bundle over $M$, and $\mathcal W$ is canonically isomorphic to
$t^\ast C$.

\begin{example}
  \label{def:FullCoreGr}
  Let $G \rightrightarrows M$ be a Lie groupoid and let $q : C \to M$
  be a vector bundle carrying a left representation of $G$. Denote by
  $0_M \to M$ the rank zero vector bundle. Then there is a full-core
  VB-groupoid structure $(t^\ast C \rightrightarrows 0_M;
  G\rightrightarrows M)$. The structure maps of the top groupoid
  $t^\ast C \rightrightarrows 0_M$ are the following:
  \begin{definitionlist}
  \item the source and the target are given by $\tilde s (g, \chi) =
    0_{s(g)}$ and $\tilde t (g, \chi) = 0_{t(g)} = 0_{q(\chi)}$,
  \item the multiplication is given by $\tilde m ((g, \chi), (g',
   \chi')) = (gg', \chi + g.\chi')$,
  \item the unit and the inversion are given by $\tilde u (0_x) =
    (1_{x}, 0_x)$, for all $x \in M$, and $\tilde i (g, \chi) =
    (g^{-1},-g^{-1}.\chi)$.
  \end{definitionlist}
  In other words, $t^\ast C \rightrightarrows 0_M = M$ is the
  \emph{representation groupoid} corresponding to the representation
  $C$ of $G$. Actually, every full-core VB-groupoid arises in this
  way. In \cite{Gracia-Saz2010} full-core VB-groupoids correspond to
  representations of $G$ seen as representations up to homotopy
  concentrated in degree $-1$.
\end{example}
Clearly, full-core VB-groupoids are dual to trivial-core VB-groupoids
and vice-versa. If $(E_G\rightrightarrows E; G\rightrightarrows M)$ is
the trivial-core VB-groupoid corresponding to a left representation of
$G$ on $E$, then the dual full-core VB-groupoid corresponds to the
transpose representation of $G$ on $C = E^\ast$.
Similar definitions exist in the realm of Lie
algebroids. Specifically, we give the following
\begin{definition}[Trivial-core and full-core VB-algebroids]
  A VB-algebroid $(W\Rightarrow E ; A\Rightarrow M)$ is
  \emph{trivial-core} if its core is trivial, and it is
  \emph{full-core} if its core is $C = W|_M$.
\end{definition}
\begin{example}
  \label{ex:st_algebroid}
  Let $A\Rightarrow M$ be a Lie algebroid, denote by $p : A \to M$ the
  projection, and let $q:E \to M$ be a vector bundle carrying a
  representation of $A$. Denote by $\nabla$ the flat $A$-connection in
  $E$ defining the representation. For every section $a \in \Gamma
  (A)$ the operator $\nabla_a$ is a derivation of $E$. Hence, it
  corresponds to a linear vector field on $E$, that we denote by
  $X_a$. The vector field $X_a$ projects onto $\rho (a)$. It is easy
  to see that there is a trivial-core VB-algebroid structure $(E_{A}
  := p^\ast E\Rightarrow E; A\Rightarrow M)$. The structure maps of
  the top Lie algebroid $E_A\Rightarrow E$ are the following:
  \begin{definitionlist}
  \item the projection $\tilde p = p_A : E_A \to E$ is the one
    induced by $p$,
  \item the anchor $\tilde\rho : E_A \to TE$ is given by
    $\tilde\rho (a_x, e) = (X_a)_e$ for all $a \in \Gamma(A)$,
    $x \in M$, and $e\in E_x$,
  \item the Lie bracket $[-,-]: \Gamma (E_A, E) \times \Gamma (E_A, E)
    \to \Gamma (E_A, E)$ is uniquely defined by the anchor and
    \begin{equation}
      [q^\ast a, q^\ast b] = q^\ast [a,b],
    \end{equation}
    for all pull-back sections $q^\ast a, q^\ast b \in \Gamma (E_A,
    E)$, with $a,b \in \Gamma(A)$.
  \end{definitionlist}
  In other words, $E_{A} \Rightarrow E$ is the \emph{action algebroid}
  corresponding to the action of $A$ on the fibered manifold $E \to
  M$. Every trivial-core VB-algebroid arises in this
  way. In~\cite{gracia2010lie} trivial-core VB-algebroids
  $(E_A\Rightarrow E; A\Rightarrow M)$ correspond to representations
  of $A$ on $E$ seen as representations up to homotopy concentrated in
  degree 0.
\end{example}

\begin{example}
  Let $A$ be as in Example \ref{ex:st_algebroid} and let $C \to M$ be
  a vector bundle with a flat $A$-connection $\nabla$. Then there is a
  full-core VB-algebroid structure $(A \times_M C\Rightarrow 0_M;
  A\Rightarrow M)$. The structure maps of the top Lie algebroid $A
  \times_M C \Rightarrow 0_M = M$ are the following:
  \begin{definitionlist}
  \item the projection $\tilde p : A \times_M C \to 0_M$ is
    the obvious one,
  \item the anchor $\tilde\rho : A \times_M C \to TM$ is given
    by $\tilde\rho (a, e) = \rho(a)$, and
  \item the Lie bracket 
  \begin{equation}
  [-,-]: \Gamma (A \times_M C,
    M) \times \Gamma (A \times_M C, M) \to \Gamma (A \times_M C, M)
    \end{equation}
    is given by
    \begin{equation}
      [(a, \chi), (a', \chi')] = ([a,a'], \nabla_a \chi' - \nabla_{a'}\chi)
    \end{equation}
    for all $a, a' \in \Gamma (A)$, and $\chi,\chi' \in \Gamma(C)$.
  \end{definitionlist}
  In other words, $A \times_M C \to M$ is the \emph{representation
  algebroid} corresponding to the representation $C$ of $A$, i.e.~the
    \emph{semidirect product of $A$ and $C$}. All full-core
    VB-algebroids arise in this way. In~\cite{gracia2010lie} they
    correspond to representations of $A$ seen as representations up to
    homotopy concentrated in degree $-1$.
\end{example}
Full-core VB-algebroids are dual to trivial-core VB-algebroids and
vice-versa. If $(E_A\Rightarrow E; A\Rightarrow M)$ is the
trivial-core VB-algebroid corresponding to a representation of $A$ on
$E$, the dual full-core VB-algebroid corresponds to the transpose
representation of $A$ on $C = E^\ast$.

The proof of the following proposition is straightforward.
\begin{proposition}
  \label{prop:st_fc_a} \quad
  \begin{definitionlist}
  \item Let $G$ be a Lie groupoid with Lie algebroid $A$, and let $E,
    C$ be representations of $G$ (hence of $A$). The VB-algebroid of
    the trivial-core (resp.~full-core) VB-groupoid determined by $G$
    and $E$ (resp.~$C$) is the trivial-core (resp.~full-core)
    VB-algebroid corresponding to $A$ and $E$ (resp.~$C$).
  \item Let $A$ be an integrable Lie algebroid with source simply
    connected integration $G$, and let $E, C$ be representations of
    $A$, hence of $G$. The trivial-core (resp.~full-core) VB-algebroid
    determined by $A$ and $E$ (resp.~$C$) is integrable and its source
    simply connected integration is the trivial-core (resp.~full-core)
    VB-groupoid determined by $G$ and $E$ (resp.~$C$).
\end{definitionlist}
\end{proposition}

\begin{example}
  A VB-groupoid $(\mathcal W \rightrightarrows E; G \rightrightarrows M)$
  (resp.~a VB-algebroid $(W\Rightarrow E; A\Rightarrow M)$) such that
  $\mathcal W$ (resp.~$W$) is a line bundle over $G$ (resp.~$A$) is
  necessarily either trivial-core or full-core.
\end{example}

\subsection{The gauge algebroid of trivial-core/full-core VB-groupoids/algebroids}\label{sec:gauge_trivial}

Let $G \rightrightarrows M$ be a Lie groupoid and let $q : E \to M$ be
a vector bundle carrying a representation of $G$. As discussed in the
previous section, starting with this data we can construct two
VB-groupoids, a trivial-core one $(E_G\rightrightarrows E;
G\rightrightarrows M)$ and a full-core one $(t^\ast E\rightrightarrows
0_M; G\rightrightarrows M)$. Actually, $E_G = s^\ast E \to G$ and
$t^\ast E \to G$ are canonically isomorphic as vector bundles over $G$
(see, e.g., \cite{VW2016}). The vector bundle isomorphism $\tau : E_G
\to t^\ast E$ is given by
\begin{equation}
    \tau (g, e) = (g, g.e),
\end{equation}
for all $(g,e) \in G \mathbin{{}_{s}  \times_q} E = E_G$. Hence the gauge
algebroids $\der E_G$ and $\der t^\ast E$ are canonically isomorphic as Lie
algebroids over $G$.

In this section, we show that the gauge algebroid $\der E_G$ fits
itself in a VB-groupoid $(\der E_G\rightrightarrows \der E;
G\rightrightarrows M)$ (Proposition \ref{prop:diff_VB-grpd}). Using
the isomorphism $\tau$ we immediately get that $\der t^\ast E$ fits in
a VB-groupoid as well. In other words, the situation is basically the
same for trivial-core and full-core VB-groupoids, exactly as one would
expect from the fact that they are both equivalent to the same object:
a representation of a Lie groupoid.
\begin{proposition}
  \label{prop:diff_VB-grpd}
  There is a natural VB-groupoid structure $(\der E_G\rightrightarrows
  \der E; G\rightrightarrows M)$ and a short exact sequence of
  VB-groupoids:
  \begin{equation}
    \label{eq:ses_DE_G}
    \begin{gathered}
      \begin{tikzcd}[row sep=1em, column sep=.7em]
        0\arrow[rrr] &&& \mathsf{End}E_G\arrow[rrr]\arrow[ddl]\arrow[dr, shift left]\arrow[dr, shift right] &&& \der E_G\arrow[rrr, "\sigma"]\arrow[ddl]\arrow[dr, shift left]\arrow[dr, shift right] &&& TG\arrow[rrr]\arrow[ddl]\arrow[dr, shift left]\arrow[dr, shift right] &&& 0 &
        \\
        & 0\arrow[rrr, crossing over] &&& \mathsf{End}E\arrow[rrr, crossing over] &&& \der E\arrow[rrr, crossing over, "\sigma", pos=.35] &&& TM\arrow[rrr, crossing over] &&& 0
        \\
        && G\arrow[rrr, -, shift left=.27ex]\arrow[rrr, -, shift right=.27ex]\arrow[dr, shift left]\arrow[dr, shift right] &&& G\arrow[rrr, -, shift left=.27ex]\arrow[rrr, -, shift right=.27ex]\arrow[dr, shift left]\arrow[dr, shift right] &&& G\arrow[dr, shift left]\arrow[dr, shift right] &&&&&
        \\
        &&& M\arrow[rrr, -, shift left=.27ex]\arrow[rrr, -, shift right=.27ex]\arrow[from=uur, crossing over] &&& M\arrow[rrr, -, shift left=.27ex]\arrow[rrr, -, shift right=.27ex]\arrow[from=uur, crossing over] &&& M\arrow[from=uur, crossing over] &&&&\\
    \end{tikzcd}
  \end{gathered}
  \end{equation}
  Here $(\mathsf{End}E_G \rightrightarrows \mathsf{End} E;
  G\rightrightarrows M)$ is the trivial-core VB-groupoid corresponding
  to the induced action of $G$ on $\mathsf{End} E$.
\end{proposition}

\begin{proof}
  \label{proof:diff_groupoid_structure}
  \color{black}
  Denote by 
  \begin{equation}
    \mathsf a : G \mathbin{{}_{s} \times_q} E = E_G \to E, \quad (g, e) \mapsto \mathsf a_g e = g.e = \tilde t (g,e),
  \end{equation}
  the action of $G$ on $E$.  We begin with a remark (for which we
  thank the referee): \emph{the action $\mathsf a$ induces an action
    $\mathcal A$ of the tangent groupoid $TG \rightrightarrows TM$ on
    $\sigma : \der E \to TM$}. To see this, let $g \in G$, let $V \in
  T_g G$, and let $\Delta \in \der_{s(g)} E$ be such that:
  \begin{equation}
    \mathrm{d} s (V) = \sigma (\Delta).
  \end{equation}
  Choose a curve $\varepsilon \mapsto \Gamma (\varepsilon)$ in $G$ such that
  $\Gamma (0) = g$ and
  \begin{equation}
    V = \left.\frac{d}{d\varepsilon}\right|_{\varepsilon = 0} \Gamma (\varepsilon),
  \end{equation}
  and denote $\gamma = s \circ \Gamma$. Choose also a curve $\varepsilon
  \mapsto \gamma_E(\varepsilon) : E_{s(g)} \to E_{\gamma (\varepsilon)}$ of
  isomorphisms such that
  \begin{equation}
    \Delta = \left.\frac{d}{d\varepsilon}\right|_{\varepsilon = 0} \gamma_E (\varepsilon).
  \end{equation}
  As the symbol of $\Delta$ is exactly $ds (V)$ this is possible. Now,
  put
  \begin{equation}
    \label{eq:A_V_Delta}
    \mathcal A (V, \Delta) 
    :=  
    \left.\frac{d}{d\varepsilon}\right|_{\varepsilon = 0} \mathsf a_{\Gamma (\varepsilon)} \circ \gamma_E (\varepsilon) \circ \mathsf a_{g^{-1}}.
  \end{equation}
  It immediately follows that since $\mathsf a$ is an action,
   $\mathcal A$ is an action as well, as claimed.
  
  Now, recall from Section \ref{sec:derivations} that the gauge
  algebroid $\der E_N$ of the pull-back vector bundle $E_N = \phi^\ast
  E \to N$ along a smooth map $\phi : N \to M$, is (canonically
  isomorphic to) the pull-back $ \der E_N \cong TN \mathbin{{}_{\D
      \phi} \hspace*{-0.1cm}\times_\sigma} \der E.  $ In particular,
  \begin{equation}
    \der E_G \cong TG \mathbin{{}_{\D s} \hspace*{-0.1cm}\times_\sigma}  \der E.
  \end{equation}
  The action of $TG \rightrightarrows \der E$ then gives $\der E_G$
  the structure of an action groupoid $\der E_G \rightrightarrows \der
  E$. It is now easy to check that with all its structure maps,
  $(\der E_G \rightrightarrows \der E; G\rightrightarrows M)$ is
  indeed a VB-groupoid fitting in the exact sequence
  \eqref{eq:ses_DE_G}. We leave details to the reader (see also Remark
  \ref{rem:D_struct}).
 \end{proof}
 
 \color{black}
 \begin{remark}\label{rem:D_struct}
   The structure maps of the groupoid $\der E_G \rightrightarrows \der
   E$ are those of the action groupoid. For instance, the source
   $s_{\der} : \der E_G \to \der E$ is just the projection of $\der
   E_G = TG \mathbin{{}_{\D s} \hspace*{-0.1cm}\times_\sigma} \der E$ onto the second
   factor. In other words $s_{\der} = \der \tilde s$. The target
   $t_{\der} : \der E_G \to \der E$ is just the action map $\mathcal
   A : \der E_G = TG \mathbin{{}_{\D s} \hspace*{-0.1cm}\times_\sigma} \der E \to \der
   E$, and a direct check exploiting (\ref{eq:der_phi_vel}) reveals
   that $t_\der = \der \tilde t$. It is easy to check that the other
   structure maps (multiplication $m_\der$, unit $u_\der$, invertion
   $i_\der$) are all induced by those of the VB-groupoid $(E_G
   \rightrightarrows E; G \rightrightarrows M)$ via application of the
   functor $\der$. The multiplication may need a little
   explanation. Consider the space $E_G^{(2)} = E_G
   \mathbin{{}_{\tilde s} \times_{\tilde t}} E_G$ of composable arrows
   of $E_G \rightrightarrows E$. It is a vector bundle over $G^{(2)}$
   in the obvious way, and the projections
   $\tilde{\operatorname{pr}}_i : E_G^{(2)} \to E_G$ over the factors
   are regular vector bundle morphisms covering the projections
   $\operatorname{pr}_i : G^{(2)} \to G$, $i = 1,2$. Similarly, the
   multiplication $\tilde m : E_G^{(2)} \to E_G$ is a regular vector
   bundle morphism covering $m : G^{(2)} \to G$. The space of
   composable arrows of $\der E_G \rightrightarrows \der E$ is
   \begin{equation}
     (\der E_G)^{(2)} = \der E_G \mathbin{{}_{\der \tilde s} \times_{\der \tilde t}}  \der E_G,
   \end{equation}
   and it is a vector bundle over $G^{(2)}$ as well. Finally, the map
   \begin{equation}
     \der (E_G^{(2)}) \to (\der E_G)^{(2)}, \quad \square
     \mapsto (\der \tilde{\operatorname{pr}}_1 (\square), \der \tilde{\operatorname{pr}}_2 (\square))
   \end{equation}
   is an isomorphism of vector bundles over $G^{(2)}$. Under this
   identification, $m_\der = \der \tilde m$. Similarly, $u_\der = \der
   \tilde u$ and $i_\der = \der \tilde i$. This shows a close analogy
   between the groupoid $\der E_G \rightrightarrows \der E$ and the
   tangent groupoid $TG \rightrightarrows TM$.
 \end{remark}
 \color{black}
\begin{remark}
  \label{rem:core_DE_G}
  The core $C$ of the VB-groupoid $(\der E_G \rightrightarrows \der E;
  G \rightrightarrows M)$ is canonically isomorphic to $A$. The
  isomorphism $A \to C$ maps $a$ to $\mathbb D_a$, where $\mathbb D$
  is the obvious flat $(\ker \D s)$-connection in $E_G$.
\end{remark}

A similar situation occurs for trivial-core/full-core VB algebroids:
let $A \Rightarrow M$ be a Lie algebroid with projection $p: A \to M$,
and let $E \to M$ be a vector bundle carrying a representation of $A$.
We denote by $\nabla$ the flat $A$-connection in $E$. Beware that
from now on $\der E_A$ will always indicate the gauge algebroid of
the vector bundle $E_A \to A$ (\emph{not} the vector bundle $E_A \to
E$).
We will now show that $\der E_A$ fits in a VB-algebroid $(\der
E_A\Rightarrow \der E; A\Rightarrow M)$ (Proposition
\ref{prop:diff_VB-alg}). Finally, if $(E_A\Rightarrow E; A\Rightarrow
M)$ is the VB-algebroid of the VB-groupoid $(E_G\rightrightarrows
E;G\rightrightarrows M)$ then $(\der E_A\Rightarrow \der E;
A\Rightarrow M)$ is the VB-algebroid of $(\der E_G\rightrightarrows
\der E; G\rightrightarrows M)$ (Proposition \ref{prop:diff_int}).
\begin{proposition}
  \label{prop:diff_VB-alg}
  There is a natural VB-algebroid structure $(\der E_A \Rightarrow
  \der E; A \Rightarrow M)$ and a short exact sequence of
  VB-algebroids:
  \begin{equation}
    \label{eq:ses_DE_A}
    \begin{gathered}
      \begin{tikzcd}[row sep=1em, column sep=.7em]
        0\arrow[rrr] &&& \mathsf{End}E_A\arrow[rrr]\arrow[ddl]\arrow[dr, Rightarrow] &&& \der E_A\arrow[rrr, "\sigma"]\arrow[ddl]\arrow[dr, Rightarrow] &&& TA\arrow[rrr]\arrow[ddl]\arrow[dr, Rightarrow] &&& 0 &
        \\
        & 0\arrow[rrr, crossing over] &&& \mathsf{End}E\arrow[rrr, crossing over] &&& \der E\arrow[rrr, crossing over, "\sigma", pos=.35] &&& TM\arrow[rrr, crossing over] &&& 0
        \\
        && A\arrow[rrr, -, shift left=.27ex]\arrow[rrr, -, shift right=.27ex]\arrow[dr, Rightarrow] &&& A\arrow[rrr, -, shift left=.27ex]\arrow[rrr, -, shift right=.27ex]\arrow[dr, Rightarrow] &&& A\arrow[dr, Rightarrow] &&&&&
        \\
        &&& M\arrow[rrr, -, shift left=.27ex]\arrow[rrr, -, shift right=.27ex]\arrow[from=uur, crossing over] &&& M\arrow[rrr, -, shift left=.27ex]\arrow[rrr, -, shift right=.27ex]\arrow[from=uur, crossing over] &&& M\arrow[from=uur, crossing over] &&&&\\
      \end{tikzcd}
    \end{gathered}
  \end{equation}
  Here $(\mathsf{End}E_A \Rightarrow \mathsf{End} E; A\Rightarrow M)$ is the trivial-core
  VB-algebroid corresponding to the induced action of $A$ on $\mathsf{End} E$.
\end{proposition}

We propose a proof in the same spirit of that of Proposition
\ref{prop:diff_VB-grpd}. We conjecture that there exists a proof
exploiting the fact that $\der E_A$ is the fat algebroid of the
VB-algebroid $(T E_A \Rightarrow E_A; T A \Rightarrow A)$. Actually,
$T E_A$ is a triple structure: it is a Lie algebroid over $E_A$, over
$TA$, and over $TE$, and all these structures are compatible. We
expect the conjectured proof to require a preliminary analysis of the
fat algebroid of a triple structure like that but this goes beyond the
scopes of the present paper.

\color{black}
\begin{proof} 
  We begin remarking that similarly as in the groupoid case
  \emph{the flat connection $\nabla$ induces an (infinitesimal) action
    \begin{equation}
      \mathfrak A : \Gamma (TA, TM) \to \mathfrak X ( \der E)
    \end{equation}
    of the tangent algebroid $TA \Rightarrow TM$ on $\sigma : \der E
    \to TM$}.  To see this, recall from Example \ref{ex:tangent}, that
  sections of $TA \to TM$ are generated by core sections $\widehat
  a_T$, and linear sections of the form $\D a$, for some $a \in
  \Gamma(A)$.  Now, every section $a$ of $A$ determines two vector
  fields on $\der E$ as follows.  Correspondence $\Delta \mapsto
  [\nabla_a,\Delta]$ is a derivation of $\der E$.  Hence, it
  corresponds to a (linear) vector field $Y_a$ on (the total space of)
  $\der E$ itself.  Section $a$ does also determine another vector
  field on $\der E$: the vertical lift $V_a$ of $\nabla_a \in
  \Gamma(\der E)$.  We put %
  \begin{equation}
    \label{eq:infinitesimal_acton}
    \mathfrak A \left(\D a \right)
    =
    Y_a \quad \text{and} \quad \mathfrak A ( \widehat a_T ) = V_a,
  \end{equation}
  for all $a \in \Gamma (A)$.  It is easy to see that this uniquely
  defines an infinitesimal action as required.

  Now, the gauge algebroid $\der E_A$ is the pull-back $\der E_A \cong
  TA \mathbin{{}_{\D p} \hspace*{-0.1cm}\times_\sigma} \der E$.  The infinitesimal
  action $\mathfrak A$ then gives $\der E_A$ the structure of an
  action algebroid $\der E_A \Rightarrow \der E$.  It is easy to check
  that $(\der E_A \Rightarrow \der E;
  A\Rightarrow M)$ with all its structure maps is indeed a VB-groupoid fitting in the exact
  sequence (\ref{eq:ses_DE_A}).  We leave details to the reader (see
  also Remark \ref{rem:der_E_A_D}).
\end{proof}

\begin{remark}\label{rem:der_E_A_D}
  The structure maps of the algebroid $\der E_A \Rightarrow \der E$
  are those of the action algebroid. For instance, the projection
  $p_\der : \der E_A \to \der E$ is just the projection of $\der E_A =
  TA \mathbin{{}_{\D p} \hspace*{-0.1cm}\times_\sigma} \der E$ onto the second
  factor. In other words $p_\der = \der p$. In order to describe the
  anchor $\rho_\der : \der E_A \to T \der E$ first notice that
  sections of $\der E_A \to \der E$ are generated by pull-backs of
  sections of $TA \to TM$. The pull-back of a core section $\widehat
  a_T$ is a core section that we denote by $\widehat a_\der$, while
  the pull-back of a linear section of the form $\D a$ is simply $\der
  a$, $a \in \Gamma (A)$. Hence the anchor $\rho_\der : \der E_A \to T
  \der E$ is uniquely determined by
  \begin{equation}
    \label{eq:anchor}
    \rho_\der \left(\der a \right) 
    = 
    Y_a \quad \text{and} \quad \rho_\der ( \widehat a_\der ) = V_a,
  \end{equation}
  for all $a \in \Gamma (A)$. Finally,
  \begin{equation}
    \label{eq:brackets}
    \begin{aligned}
      \left[\der a, \der {b} \right] & = \der {[a,b]},\\
           [ \der a, \widehat b_\der ] & = \widehat{[a,b]}_\der,\\
           [\widehat a_\der, \widehat{b}_\der ] & = 0,
    \end{aligned}
  \end{equation}
  for all $a, b \in \Gamma(A)$.
\end{remark}

\begin{remark}
  \label{rem:core_DE_A}
  The core $C$ of the VB-algebroid $(\der E_A \Rightarrow \der E; A
  \Rightarrow M)$ is canonically isomorphic to $A$. The isomorphism $A
  \to C$ maps $a$ to $\mathbb D_a$, where $\mathbb D$ is the flat
  $(\ker \D p)$-connection in $E_A$.
\end{remark}

\color{black} 

\begin{proposition}
  \label{prop:diff_int}
  Let $(E_G\rightrightarrows E; G\rightrightarrows M)$ be a
  trivial-core VB-groupoid, and let $(E_A \Rightarrow E; A\Rightarrow
  M)$ be its trivial-core VB-algebroid (equivalently, let $(t^\ast E
  \rightrightarrows 0_M; G \rightrightarrows M)$ be a full-core
  VB-groupoid and let $(A \times_M E \Rightarrow 0_M; A \Rightarrow
  M)$ be its full-core VB-algebroid). Then $(\der E_A \Rightarrow \der
  E; A \Rightarrow M)$ is canonically isomorphic to the VB-algebroid
  of the VB-groupoid $(\der E_G\rightrightarrows \der E;
  G\rightrightarrows M)$.
\end{proposition}

\begin{proof}
  \color{black} It is enough to check that the Lie functor intertwines
  the action $\mathcal A $ in the proof of Proposition
  \ref{prop:diff_VB-grpd} and the infinitesimal action $\mathfrak A$
  in the proof of Proposition \ref{prop:diff_VB-alg}. To do this,
  denote by
  \begin{equation}
  \dot{\mathcal A} : \Gamma (TA, TM) \to \mathfrak X (\der E)
  \end{equation}
  the infinitesimal action induced by $\mathcal A$. This means that
  for every section $w \in \Gamma (TA, TM)$, every $x \in M$ and
  every point $\Delta \in \der_x E$
  \begin{equation}\label{eq:inf_act}
    \dot{\mathcal A} (w)_\Delta 
    := 
    \left. \frac{d}{d \lambda} \right|_{\lambda = 0} \mathcal A (\phi^{w}_{\lambda} (v), \Delta)
  \end{equation}
  where $\{\phi^w_\lambda \}$ is the flow of $\overrightarrow w$, the
  right invariant vector field on $TG$ corresponding to $w$, $v =
  \sigma (\Delta)$, and we identify $v$ with $1_v \in TG$.

  Now, we compute (\ref{eq:inf_act}) in the cases $w = \widehat a_T$
  and $w = \D a$, for some $a \in \Gamma (A)$. We begin with the first
  one. For every $x \in M$, and every $v \in T_x M$, we have (see
  \cite{Mackenzie2005})
  \begin{equation}
  \phi^{\widehat a_T}_{\lambda} (v) = v + \lambda a_{x}. 
  \end{equation}
  Hence
  \begin{equation}
  \begin{aligned}
    \dot{\mathcal A} (\widehat a_T)_\Delta & = \left. \frac{d}{d \lambda} \right|_{\lambda = 0} \mathcal A (v + \lambda a_x, \Delta) 
    = 
    \left. \frac{d}{d \lambda} \right|_{\lambda = 0} \mathcal A (v + \lambda a_x, \Delta + 0) 
    = 
    \left. \frac{d}{d \lambda} \right|_{\lambda = 0} \mathcal A (v, \Delta) + \lambda \mathcal A (a_x, 0) 
    \\
    & = 
    \left. \frac{d}{d \lambda} \right|_{\lambda = 0}  \Delta + \lambda \mathcal A (a_x, 0),
  \end{aligned} 
  \end{equation}
  where we used the (easy to check) fact that $\mathcal A :TG
  \mathbin{{}_{\D s} \hspace*{-0.1cm}\times_\sigma} \der E \to \der E$ is a vector
  bundle morphism covering the target $t : G \to M$. Now, for every
  section $e \in \Gamma (E)$,
  \begin{equation}\label{eq:nabla_a_x_e}
    \nabla_{a_x} e = \left. \frac{d}{d \varepsilon} \right|_{\varepsilon = 0} \mathsf  a_{\phi^a_\varepsilon (x)} e,  
  \end{equation}
  where $\{ \phi^a_\varepsilon \}$ is the flow of $\overrightarrow a$,
  and a direct computation exploiting (\ref{eq:nabla_a_x_e}) and
  (\ref{eq:A_V_Delta}) then shows that $\mathcal A (a_x, 0) =
  \nabla_{a_x}$. We conclude that
  \begin{equation}
  \dot{\mathcal A} (\widehat a_T)_\Delta
  = 
  \left. \frac{d}{d \lambda} \right|_{\lambda = 0}  \Delta + \lambda \nabla_{a_x}
  = 
  (V_a)_\Delta = \mathfrak A (\widehat a_T)_\Delta.
  \end{equation}

  Finally, let $w = \D a$. We have
  \begin{equation}
    \phi_\lambda^{\D a} = \D \phi_\lambda^a,
  \end{equation}
  where, as usual, $\{ \phi^a_\lambda \}$ is the flow of $\overrightarrow a$. Hence
  \begin{equation}
    \dot{\mathcal A} (\D a)_\Delta := \left. \frac{d}{d \lambda} \right|_{\lambda = 0} \mathcal A (\D \phi^{a}_{\lambda} (v), \Delta).
  \end{equation}
  Choose a curve $\varepsilon \mapsto \gamma (\varepsilon)$ in $G$ such that
  $\gamma (0) = x$, and
  \begin{equation}
    v = \left.\frac{d}{d\varepsilon}\right|_{\varepsilon = 0} \gamma (\varepsilon),
  \end{equation}
  and a curve $\varepsilon \mapsto \gamma_E (\varepsilon) : E_{x} \to E_{\gamma
    (\varepsilon)}$ of isomorphisms such that
  \begin{equation}
    \Delta = \left.\frac{d}{d\varepsilon}\right|_{\varepsilon = 0} \gamma_E (\varepsilon).
  \end{equation}
  Then 
  \begin{equation}
    \D\phi^{a}_{\lambda} (v) = \left.\frac{d}{d\varepsilon}\right|_{\varepsilon = 0} \phi^{a}_{\lambda} (\gamma (\varepsilon)),
  \end{equation}
  and from (\ref{eq:A_V_Delta}),
  \begin{equation}
    \mathcal A (\D \phi^{a}_{\lambda} (v), \Delta) 
    = 
    \left.\frac{d}{d\varepsilon}\right|_{\varepsilon = 0} \mathsf a_{\phi^a_{\lambda} (\gamma (\varepsilon))}
  \circ \gamma_E (\varepsilon) \circ \mathsf a_{\phi^a_\lambda (x)^{-1}} = \der \Phi^a_\lambda (\Delta),
  \end{equation}
  where $\Phi^a_\lambda$ is the vector bundle isomorphism
  \begin{equation}
    \Phi^a_\lambda : E \to E, \quad e \mapsto \mathsf a_{\phi_\lambda^a (q(e))} e.
  \end{equation}
  This shows that the vector field $\dot{\mathcal A} (\D a) \in
  \mathfrak X (\der E)$ generates the flow $\{ \der \Phi^a_\lambda
  \}$. Notice that $\{ \Phi^a_\lambda \}$ is the flow of
  $\nabla_a$. It now easily follows that $\{ \der
  \Phi^a_\lambda \}$ is the flow of the adjoint operator $[\nabla_a,
    -]$. In other words
  \begin{equation}
    \dot{\mathcal A} (\D a) = Y_a = \mathfrak A (\D a).
  \end{equation}
  This concludes the proof.
\end{proof}

\subsection{Multiplicative derivations as multiplicative sections}

Recall that a vector field $X$ on a Lie groupoid $G \rightrightarrows
M$ is multiplicative if and only if, when regarded as a section $X : G
\to TG$ of $TG \to G$, it is a morphism of Lie groupoids. A similar
statement holds true for IM vector fields on Lie algebroids. In this
section we prove similar results for multiplicative derivations of
trivial-core and full-core VB-groupoids and IM derivations of
trivial-core and full-core \color{black} VB-algebroids.
%

\begin{theorem}
  \label{theor:der_VB-grpd}
  Let $(E_G\rightrightarrows E; G\rightrightarrows M)$ be a
  trivial-core VB-groupoid. A derivation of $E_G \to G$ is
  multiplicative if and only if it is a multiplicative section of
  $\der E_G \to G$.
\end{theorem}

\begin{proof}
  Let $G\rightrightarrows M$ be a Lie groupoid, and let $E\to M$ be a
  vector bundle equipped with a representation of $G$.  Then $(E_G
  \rightrightarrows E;G \rightrightarrows M)$ is a trivial-core
  VB-groupoid with structure maps $\tilde s,\tilde t, \tilde m, \tilde
  u, \tilde i$ as in Example~\ref{ex:TrivCoreGrpd}.  Similarly the
  dual representation of $G$ on $E^\ast\to M$ determines the
  trivial-core VB-groupoid $(E^\ast_G \rightrightarrows E^\ast;G
  \rightrightarrows M)$ with structure maps denoted by $\tilde
  s^\star,\tilde t^\star, \tilde m^\star, \tilde u^\star, \tilde
  i^\star$.  Applying Proposition~\ref{prop:diff_VB-grpd} to both
  $(E_G\rightrightarrows E;G\rightrightarrows M)$ and
  $(E^\ast_G\rightrightarrows E^\ast;G\rightrightarrows M)$, we get
  VB-groupoids $(\der E_G\rightrightarrows \der E;G\rightrightarrows
  M)$ and $(\der E^\ast_G\rightrightarrows \der
  E^\ast;G\rightrightarrows M)$ with structure maps denoted by
  $s_\der, t_\der, m_\der, u_\der, i_\der$ and $s^\star_\der,
  t^\star_\der, m^\star_\der, u^\star_\der, i^\star_\der$
  respectively.  It is straightforward to see that isomorphism $\der
  E_G\to\der E^\ast_G,\ \Delta\mapsto\Delta^\ast$ (see
  Lemma~\ref{lem:dual}) is actually a VB-groupoid isomorphism $(\der
  E_G\rightrightarrows \der E;G\rightrightarrows M) \to (\der
  E^\ast_G\rightrightarrows \der E^\ast;G\rightrightarrows M)$
  (covering isomorphism $\der E \to \der E^\ast$).  As a consequence,
  a derivation $\Delta\in\Gamma(\der E_G)$ is a multiplicative section
  of $\der E_G \to \der E$ if and only if its ``dual'' derivation
  $\Delta^\ast\in\Gamma(\der E^\ast_G)$ is a multiplicative section of
  $\der E^\ast_G \to \der E^\ast$.  So it remains to check that the
  following conditions are equivalent:
  \begin{definitionlist}
  \item\label{enumitem:vector_field}
    $X_\Delta$ is a multiplicative section of $T E_G \to E_G$,
  \item\label{enumitem:derivation}
    $\Delta^\ast$ is a multiplicative section of $\der E^\ast_G \to G$.
  \end{definitionlist}
  First show that $\textit{i.)}
  \Rightarrow\textit{ii.)}$. Condition~\textit{i.)} means that there
  is a (necessarily unique) vector field $\underline{\smash{X}}\in\mathfrak X(E)$
  such that
  \begin{equation}
    \label{eq:proof:theor:der_VB-grpd_1a}
    \D\tilde s (X_\Delta)_{(g,e)}=\underline{\smash{X}}_e,\qquad
    \D\tilde t (X_\Delta)_{(g,e)}=\underline{\smash{X}}_{g.e},
  \end{equation}
  for all $(g,e)\in E_G$, and, additionally, for every pair of
  composable arrows $((g,e),(g',e'))\in(E_G)^{(2)}$, we have
  \begin{equation}
    \label{eq:proof:theor:der_VB-grpd_2a}
    (X_\Delta)_{(gg',e')}=d\tilde m((X_\Delta)_{(g,e)},(X_\Delta)_{(g',e')}).
  \end{equation}
  In this case, it follows from~\eqref{eq:proof:theor:der_VB-grpd_1a},
  that $\underline{\smash{X}}$ is also a linear vector field,
  i.e.~$\underline{\smash{X}}=X_{\underline{\smash{\Delta}}}$ for some
  (necessarily unique) derivation
  $\underline{\smash{\Delta}}\in\Gamma(\der E)$.  Hence we can
  rephrase~\eqref{eq:proof:theor:der_VB-grpd_1a} in the following
  equivalent way
  \begin{equation}
    \label{eq:proof:theor:der_VB-grpd_1b}
    s^\star_\der \Delta^\ast_g=\underline{\smash{\Delta}}^\ast_{s(g)},\qquad
    t^\star_\der \Delta^\ast_g=\underline{\smash{\Delta}}^\ast_{t(g)},
  \end{equation}
  for all $g\in G$.  
  assume the first one of~\eqref{eq:proof:theor:der_VB-grpd_1a}, then
  Equation~\eqref{eq:proof:theor:der_VB-grpd_2a} holds if and only if
  its sides have the same image through both $\D \tilde s$ and the
  tangent maps to the projections $q_G:E_G\to G$. The first condition
  is guaranteed by the first one
  of~\eqref{eq:proof:theor:der_VB-grpd_1a}. As $X_\Delta$ is
  $q_G$-related to $\sigma(\Delta)=\sigma(\Delta^\ast)$, in view
  of~\eqref{eq:ses_DE_G} the second condition is equivalent to
  \begin{equation}\label{eq:proof}
    \sigma(\Delta^\ast_{gg'})=\sigma(m_\der^\star(\Delta^\ast_g,\Delta^\ast_{g'})),
  \end{equation}
  for all $(g,g')\in G^{(2)}$.  Finally,
  from the isomorphism $\der E_G = TG \mathbin{{}_{\D s}  \hspace*{-0.1cm}\times_\sigma} \der E$, Equation
  \eqref{eq:proof}, together with the first one
  of~\eqref{eq:proof:theor:der_VB-grpd_1b}, gives
  \begin{equation}
    \label{eq:proof:theor:der_VB-grpd_2b}
    \Delta^\ast_{gg'}=m^\star_\der(\Delta^\ast_{g},\Delta^\ast_{g'}).
  \end{equation}
  for all $(g,g') \in G^{(2)}$. Equations
  \eqref{eq:proof:theor:der_VB-grpd_1b}
  and~\eqref{eq:proof:theor:der_VB-grpd_2b} exactly mean that
  Condition \textit{ii.)} holds.  Implication $\textit{ii.)}
  \Rightarrow\textit{i.)}$ can be proved by retracing the above steps
  in the reverse order.
\end{proof}
\begin{theorem}
  Let $(t^\ast C \rightrightarrows 0_M; G\rightrightarrows M)$ be a
  full-core VB-groupoid. A derivation of $t^\ast C \to G$ is
  multiplicative if and only if it corresponds, via the vector bundle
  isomorphism $\tau : C_G \to t^\ast C$, to a multplicative section of
  $\der C_G \to G$.
\end{theorem}
\begin{proof}
  Let $\Delta$ be a derivation of $t^\ast C \to G$. Then $\Delta$ is
  multiplicative if and only if it generates a flow by automorphisms
  of $(t^\ast C \rightrightarrows 0_M; G \rightrightarrows M)$ if and
  only if the dual derivation $\Delta^\ast$ generates a flow by
  automorphisms of the dual VB-groupoid $(C_G^\ast\rightrightarrows
  C^\ast; G\rightrightarrows M)$ if and only if $\Delta^\ast$ is a
  multiplicative derivation of $C_G^\ast \to G$. Now
  $(C_G^\ast\rightrightarrows C^\ast; G\rightrightarrows M)$ is a
  trivial-core VB-groupoid, hence $\Delta$ is multiplicative if and
  only if $\Delta^\ast$ is a multiplicative section of $(\der
  C_G^\ast\rightrightarrows \der C^\ast; G\rightrightarrows
  M)$. Finally, recall from the proof of Theorem
  \ref{theor:der_VB-grpd}, that the correspondence $\Delta \mapsto
  \Delta^\ast$ is an isomorphism of the VB-groupoids $(\der
  C_G\rightrightarrows \der C; G\rightrightarrows M)$ and $(\der
  C_G^\ast\rightrightarrows \der C^\ast; G\rightrightarrows M)$. This
  concludes the present proof.
\end{proof}
\begin{theorem}
  \label{theor:der_VB-alg}
  Let $(E_A\Rightarrow E; A\Rightarrow M)$ be a trivial-core
  VB-algebroid. A derivation of $E_A \to A$ is IM if and only if it is
  an IM section of $\der E_A \to A$.
\end{theorem}
\begin{proof}
  Let $\Delta$ be a derivation of $E_A \to A$ and let $X_\Delta \in
  \mathfrak X(E_A)$ be the corresponding linear vector field. Regard
  $\Delta$ as a section of $\der E_A \to A$, and $X_\Delta$ as a
  section of $E_A \to TE_A$. In view of Remark \ref{rem:IM_vf} it is
  enough to show that $\Delta$ is an IM section of $\der E_A \to A$ if
  and only if $X_\Delta$ is an IM section of $E_A \to TE_A$. This can
  be checked with a long, but straightforward, computation in local
  coordinates, that we sketch below.
  Let $(x^i)$ be local coordinates on $M$, let $(e_A)$ be a local
  basis of $\Gamma (E)$, and let $(v^A)$ be the associated, linear
  fiber coordinates on $E$. Similarly, let $(\varepsilon_\alpha)$ be a
  local basis of $\Gamma (A)$, and let $(u^\alpha)$ be the associated,
  linear, fiber coordinates on $A$. So $E_A$ is coordinatized by
  $(x^i, u^\alpha, v^A)$. A derivation $\Delta : \Gamma (E) \to E_x$
  of $E \to M$ at the point $x \in M$, is locally of the form
  \begin{equation}
    \Delta (f^A e_A) 
    = 
    \left( \dot x{}^i \left. \frac{\partial}{\partial x^i}\right|_x f^B + f^A (x) v^B_A \right) e_{B,x}
  \end{equation} 
  for some real numbers $(\dot x^i, v^B_A)$ that serve as linear,
  fiber coordinates on the total space of $\der E \to M$. Note that
  $\sigma (\Delta)= \dot x{}^i \left. \frac{\partial}{\partial
    x^i}\right|_x$. Similarly, a derivation $\Delta : \Gamma (E_A) \to
  E_x$ of $E_A \to A$ at the point $a \in A$, is locally of the form
  \begin{equation}
    \Delta (f^A p^\ast e_A) = \left( \dot x{}^i \left. \frac{\partial}{\partial x^i}\right|_{a} f^B 
    + \dot u{}^\alpha \left. \frac{\partial}{\partial u^\alpha}\right|_a f^B + f^A (x) v^B_A \right) e_{B,x}
  \end{equation} 
  for some real numbers $(\dot x^i, \dot u^\alpha, v^B_A)$ that serve
  as linear, fiber coordinates on the total space of $\der E_A \to A$, and
  $\sigma (\Delta)= \dot x{}^i \left. \frac{\partial}{\partial
    x^i}\right|_a + \dot u{}^\alpha \left. \frac{\partial}{\partial
    u^\alpha}\right|_a$.
  The Lie algebroid structure in $A \to M$ is encoded by the de Rham
  differential $\D_A$ in $\Gamma (\wedge^\bullet A^\ast) =
  \wedge^\bullet \Gamma (A)^\ast$ which does locally look like:
  \begin{equation}
    \D_A 
    = 
    \rho^{i}_\alpha u^\alpha \frac{\partial}{\partial x^i} 
    - 
    \frac{1}{2}c^{\alpha}_{\beta \gamma} u^\beta \wedge u^\gamma \frac{\partial}{\partial u^\alpha},
  \end{equation}
  for some functions $(\rho^{i}_\alpha)$ encoding the anchor $\rho: A
  \to TM$, and some (structure) functions $(c^\alpha_{\beta\gamma})$
  encoding the Lie bracket in $\Gamma (A)$. Finally, the flat
  $A$-connection $\nabla$ in $E$ is locally encoded by ``Christoffel
  symbols'' $(\Gamma_\alpha{}^B_A)$ such that
  \begin{equation}
    \nabla_{\varepsilon_\alpha} e_A =  \Gamma_\alpha{}^B_A e_B.
  \end{equation}
  It is now straightforward to check that the Lie algebroid structure
  in $E_A \to E$ is encoded by a de Rham differential $\D_{E_{A}}$ in
  $\wedge^\bullet \Gamma (E_A, E)^\ast$ locally given by:
  \begin{equation}\label{eq:d_E_A}
    \D_{E_{A}}
    =  
    \rho^{i}_\alpha u^\alpha \frac{\partial}{\partial x^i} 
    - 
    \Gamma_\alpha{}^B_A v^B   u^\alpha \frac{\partial}{\partial v^B} 
    -
    \frac{1}{2}c^{\alpha}_{\beta \gamma} u^\beta \wedge u^\gamma \frac{\partial}{\partial u^\alpha}.
  \end{equation}
  The (tangent) Lie algebroid structure in $TE_A \to TE$ is then
  encoded by a differential $\D_{T E_A}$ in $\wedge^\bullet \Gamma (T
  E_A, TE)^\ast$ locally given by the ``tangent lift'' of
  \eqref{eq:d_E_A}, i.e.
  \begin{equation}
    \begin{aligned}
      \D_{TE_A} 
      = {} 
      & \rho^{i}_\alpha u^\alpha \frac{\partial}{\partial x^i}  
      + 
      \left(\frac{\partial \rho^{i}_\alpha}{\partial x^j}\dot x^j u^\alpha + \rho^i_\alpha \dot u{}^\alpha\right)\frac{\partial}{\partial \dot x{}^i}
      \\
      &- 
      \Gamma_\alpha{}^B_A v^B   u^\alpha \frac{\partial}{\partial v^B} 
      -
      \left( \frac{\partial \Gamma_\alpha{}_B^A}{\partial x^i} \dot x{}^i v^B u^\alpha 
      + 
      \Gamma_\alpha{}_B^A \dot v{}^B u^\alpha + \Gamma_\alpha{}_B^A v^B \dot u{}^\alpha \right)\frac{\partial}{\partial \dot v{}^A} 
      \\
      & 
      - 
      \frac{1}{2}c^{\alpha}_{\beta \gamma} u^\beta \wedge u^\gamma \frac{\partial}{\partial u^\alpha}
      - \left(c_{\beta \gamma}^\alpha u^\beta \wedge \dot u{}^\gamma 
      + 
      \frac{1}{2}\frac{\partial c^{\alpha}_{\beta \gamma}}{\partial x^i} \dot x{}^i u^\beta \wedge u^\gamma \right) \frac{\partial}{\partial \dot u{}^\alpha}.
    \end{aligned}
  \end{equation}
  Finally the Lie algebroid structure in $\der E_A \to \der E$ is encoded by a
  differential $\D_{\der E_A}$ in $\wedge^\bullet \Gamma (\der E_A, \der E)^\ast$
  locally given by
  \begin{equation}
    \begin{aligned}
      \D_{\der E_A} = {} & \rho^{i}_\alpha u^\alpha
      \frac{\partial}{\partial x^i} + \left(\frac{\partial
        \rho^{i}_\alpha}{\partial x^j}\dot x^j u^\alpha + \rho^i_\alpha \dot
      u{}^\alpha\right)\frac{\partial}{\partial \dot x{}^i}\\ & +
      \left( \frac{\partial \Gamma_\alpha{}^A_B}{\partial x^i} \dot
      x{}^i u^\alpha + \Gamma_\alpha{}_B^C v_C^A u{}^\alpha -
      \Gamma_\alpha{}_C^A v_B^C u^\alpha + \Gamma_\alpha{}_B^A \dot
      u{}^\alpha \right) \frac{\partial}{\partial v^A_B} \\ & -
      \frac{1}{2}c^{\alpha}_{\beta \gamma} u^\beta \wedge u^\gamma
      \frac{\partial}{\partial u^\alpha} - \left(c_{\beta
        \gamma}^\alpha u^\beta \wedge \dot u{}^\gamma +
      \frac{1}{2}\frac{\partial c^{\alpha}_{\beta \gamma}}{\partial
        x^i} \dot x{}^i u^\beta \wedge u^\gamma \right)
      \frac{\partial}{\partial \dot u{}^\alpha}.
    \end{aligned}
  \end{equation}
  Now, let $\Delta$ be a section of $\der E_A \to A$ locally given by
  \begin{equation}
    \Delta : \left\{ 
    \begin{array}{ccc}
      \dot x{}^i & = & X^i (x, u) \\
      \dot u{}^\alpha & = & U^\alpha (x,u) \\
      v^A_B & = & V^A_B (x,u) 
    \end{array}
    \right. .
  \end{equation} 
  The associated linear vector field on $E_A$ is the section $X_\Delta$
  of $TE_A \to E_A$ locally given by
  \begin{equation}
    X_\Delta : \left\{ 
    \begin{array}{ccc}
      \dot x{}^i & = & X^i (x, u) \\
      \dot u{}^\alpha & = & U^\alpha (x,u) \\
      \dot v{}^A & = & - V^A_B (x,u) v^B 
    \end{array}
    \right. .
  \end{equation} 
  By definition, $\Delta$ is an IM section if and only if
  \begin{enumerate}
  \item[$(IM1)$] $\Delta$ is a morphism of the vector bundles $A \to
    M$ and $\der E_A \to \der E$, and
  \item[$(IM2)$] $\Delta^\star \circ \D_{\der E_A} = \D_A \circ
    \Delta^\star$,
  \end{enumerate}
  where $\Delta^\star : \wedge^\bullet \Gamma (\der E_A, \der E)^\ast
  \to \wedge^\bullet \Gamma (A)^\ast$ is the obvious graded algebra
  map induced by $\Delta$ in view of $(IM1)$. Similarly, $X_{\Delta}$
  is an IM section if and only if
  \begin{enumerate}
  \item[$(IM'1)$] $X_\Delta$ is a morphism of the vector bundles $E_A \to E$ and $TE_A \to TE$, and
  \item[$(IM'2)$] $X_\Delta^\star \circ \D_{TE_A} = \D_{E_{A}}\circ X_\Delta^\star$,
  \end{enumerate}
  where $X_\Delta^\star : \wedge^\bullet \Gamma (TE_A, TE)^\ast \to
  \wedge^\bullet \Gamma (E_A, E)^\ast$ is the graded algebra map
  induced by $X_\Delta$ in view of $(IM'1)$. A direct computation
  shows that $(IM1)$ and $(IM'1)$ are both equivalent to
  \begin{equation}
    \begin{aligned}
      X^i & = X^i (x), \\
      V^A_B & = V^A_B (x), \\
      U^\alpha & = U^\alpha_\beta (x) u^\beta. 
    \end{aligned}
  \end{equation}
  Finally, $(IM'1)$ and $(IM'2)$ are both equivalent to the following
  system of PDEs:
  \begin{equation}
    \begin{aligned}
      \rho^j_\alpha \frac{\partial X^i}{\partial x^j} 
      & = 
      \frac{\partial \rho^i_\alpha}{\partial x^i} X^i + \rho^i_\beta U^\beta_\alpha, \\
      \rho^i_\alpha \frac{\partial V^A_B}{\partial x^i} 
      & = 
      \frac{\partial \Gamma_\alpha{}_B^A}{\partial x^i} X^i + \Gamma_\alpha{}_B^C V^A_C - \Gamma_\alpha{}_C^A V^C_B + \Gamma_\beta{}_A^B U^\beta_\alpha,
      \\
      \rho^i_\beta \frac{\partial U^\alpha_\gamma}{\partial x^i} - \rho^i_\gamma \frac{\partial U^\alpha_\beta}{\partial x^i}
      & = 
      c^\alpha_{\delta \beta} U^\delta_\gamma - c^\alpha_{\delta \gamma} U^\delta_\beta 
      + c^\delta_{\beta\gamma} U^\alpha_\delta  -  \frac{\partial c^\alpha_{\beta\gamma}}{\partial x^i} X^i.
    \end{aligned}
  \end{equation}
  This concludes the proof.
\end{proof}
By an analogous computation we can prove the following
\begin{theorem}
  Let $(A \times_M C\Rightarrow 0_M; A\Rightarrow M)$ be a full-core VB-algebroid. A derivation of
  $A \times_M C \to A$ is IM if and only if it is an IM section of $\der C_A \to A$.
\end{theorem}
  We showed in this section that the gauge algebroid of a trivial-core
  or a full-core VB-groupoid is a VB-groupoid itself and likewise for
  VB-algebroids. For a generic VB-groupoid $(\mathcal W
  \rightrightarrows E; G \rightrightarrows M)$ the gauge algebroid is
  \emph{not} a VB-groupoid. To see this notice that $\der \mathcal W$
  contains $\End \mathcal W = \mathcal W^\ast \otimes \mathcal W$ as a
  vector subbundle (and a Lie subalgebroid). But the tensor product of
  two VB-groupoids over $G \rightrightarrows M$ is not a VB-groupoid
  in general. In fact it is part of a \emph{higher VB-groupoid} or,
  more precisely, a \emph{simplicial vector bundle over the nerve of
    $G$} \cite{dHT20XX}. We conjecture that $\der \mathcal W$ is a
  higher VB-groupoid in the same sense.
  In the case of a VB-algebroid $(W \Rightarrow E; A \Rightarrow M)$
  we can say a little bit more. The Lie algebroid $A \Rightarrow M$ is
  equivalent to the degree $1$ $\mathbb N Q$-manifold $(A[1], \mathrm{d}_A)$
  . Similarly, as proved by T.~Voronov \cite{V2012}, $(W \Rightarrow
  E; A \Rightarrow M)$ is equivalent to the vector bundle $(W[1]_E \to
  A[1], d_W)$ in the category of degree $1$ $\mathbb N
  Q$-manifolds. Here $W[1]_E$ is obtained from $W$ shifting by one the
  degree in the fibers over $E$. Now, not only $d_W$ is a homological
  vector field, but it is a linear vector field, hence it corresponds
  to a degree $1$ derivation $\Delta_W$ of $W[1]_E \to A[1]$, such
  that $\Delta_W^2 = 0$.  Denote by $\mathcal D W[1]_E$ the module of
  all graded derivations of $W[1]_E \to A[1]$, so that $(\mathcal D
  W[1]_E, [\Delta_W, -]))$ is a DG Lie algebra. More precisely it is
  the DG Lie algebra of sections of a DG Lie algebroid $\der_A W[1]_E
  \Rightarrow A[1]$.  The cochain complex $(\mathcal D W[1]_E,
              [\Delta_W, -]))$ is canonically isomorphic to the linear
              deformation complex of $(W \Rightarrow E; A \Rightarrow
              M)$ \cite{ETV20XX, LPV20XX}, so that IM derivations of
              $W \to A$ are the same as degree zero derivations of
              $W[1]_E \to A[1]$ commuting with $\Delta_W$. In this way
              we encoded IM derivations in $\der_A W[1]_E \Rightarrow
              A[1]$. After a suitable shift, the latter is itself a
              vector bundle in the category of $\mathbb N
              Q$-manifolds, however it is not of degree $1$, but of
              degree $2$. So it is a degree $2$ Lie algebroid module
              in the sense of Va\u{\i}ntrob \cite{V1997} and
              corresponds to a 3-term representation up to homotopy of
              $A$ (see \cite{M2014}, see also next section, for more
              details on representations up to homotopy).

\section{A \texorpdfstring{$2$}{2}-term representation up to homotopy from a plain representation}
\label{sec:applications}

\subsection{VB-groupoids}

There is a nice relationship between VB-groupoids and representations
up to homotopy \cite{Gracia-Saz2010}. Let $G \rightrightarrows M$ be a
Lie groupoid. Recall that a \emph{representation up to homotopy} of
$G$ \cite{CA2013} is a graded vector bundle $q_{\mathcal E} : \mathcal
E = \bigoplus_{k} \mathcal E_k \to M$ together with a degree $1$,
$\mathbb R$-linear operator
\begin{equation}
D : C^\bullet(G, \mathcal E) \to C^\bullet (G, \mathcal E)
\end{equation}
satisfying
\begin{equation}
  D^2 = 0, \quad \text{and} \quad D(\omega \cdot \varepsilon) = \delta \omega \cdot \varepsilon + (-)^{\omega} \omega \cdot D  \varepsilon, 
\end{equation}
for all $\omega \in C^\bullet (G)$, and all $\varepsilon \in C^\bullet
(G, \mathcal E)$. Here $(C^\bullet (G), \delta)$ is the Lie groupoid
complex of $G$, and $C^p(G, \mathcal E)$ consists of $\mathcal
E$-valued groupoid $p$-cochains, i.e.~maps
\begin{equation}
\varepsilon : G^{(p)} \to \mathcal E, \quad (g_1, \ldots, g_p) \mapsto \omega (g_1, \ldots, g_p) \in \mathcal E_{t(g_1)}.
\end{equation}
Additionally, $D$ is required to preserve \emph{normalized conchains},
i.e.~cochains vanishing when one of their entry is $1$. We stress that
when we say that the degree of $D$ is $1$, we refer to the \emph{total
  degree}, i.e.~the sum of the cochain degree and the degree in
$\mathcal E$. For more details see \cite{CA2013}.  We are mainly
interested in \emph{$2$-term} representations up to homotopy, i.e.~to
the case when $\mathcal E = \mathcal E_{-1} \oplus \mathcal E_0$ is
concentrated in degree $-1, 0$.  A $2$-term representation up to
homotopy $(\mathcal E = \mathcal E_{-1} \oplus \mathcal E_0, D)$ is
actually equivalent to the following data
\begin{itemize}
\item a vector bundle map $\partial : \mathcal E_{-1} \to \mathcal
  E_0$,
\item \emph{$G$-quasi-actions} on $\mathcal E_{-1}$ and $\mathcal
  E_{0}$, i.e.~vector bundle maps $\Delta_{-1} : G \mathbin{{}_{s}
    \times_{q_{\mathcal E}}} \mathcal E_{-1} \to \mathcal E_{-1}$ and $\Delta_{0} : G
  \mathbin{{}_{s}  \times_{q_{\mathcal E}}} \mathcal E_{0} \to \mathcal E_{0}$ covering
  the target map $t : G \to M$,
\item a transformation $2$-cochain $\mathcal W \in C^2 (G, \mathcal E_{-1}
  \to \mathcal E_{0})$, i.e.~a vector bundle map
\begin{equation}
  \varpi : G^{(2)} \mathbin{{}_{s\operatorname{pr}_1}  \times_{q_{\mathcal E}}} \mathcal E_{-1} \to \mathcal E_0, \quad (g_1, g_2 ; v) \mapsto \varpi_{(g_1, g_2)} v \in E_{t(g_2)}
\end{equation}
where $s \operatorname{pr}_1 : G^{(2)} \to M$ is the composition of
the first projection $\operatorname{pr}_1 : G^{(2)}\to G$ followed by
the source map.
\end{itemize}
Additionally data $\partial, \Delta_{-1}, \Delta_0, \varpi$ have to
satisfy suitable compatibility conditions \cite[Section
  2.7]{Gracia-Saz2010}.

The main remark here is that VB-groupoids $(\mathcal W
\rightrightarrows E; G \rightrightarrows M)$ are equivalent to
$2$-term representations up to homotopy of the Lie groupoid $ G
\rightrightarrows M$ up to isomorphisms. From $(\mathcal W
\rightrightarrows E; G \rightrightarrows M)$ one can construct the
corresponding (isomorphism class of) $2$-term representation(s) up to
homotopy as follows. As usual, denote by $C$ the core. There is a
short exact sequence of vector bundles over $G$
\begin{equation}\label{eq:SESVB}
  0 \to t^\ast C \to \mathcal W \to s^\ast E \to 0.
\end{equation}
where the inclusion $t^\ast C \to \mathcal W$ is defined by $(g, \chi)
\mapsto \chi \cdot 0_{g}$, and the projection $\mathcal W \to s^\ast
E$ is the obvious one. Next we need a \emph{horizontal lift}, i.e.~a
splitting $h : s^\ast E \to \mathcal W$, $(g, e) \mapsto h_g e$ of
\eqref{eq:SESVB} with the additional property of being compatible with
the units, i.e.~$h (1_x, e) = 1_e$ for all $e \in E_x$, $x \in
M$. Horizontal lifts always exist. Now put $\mathcal E_{-1} = C[1]$
(where the ``$[1]$'' refers to the fact that sections of $C$ should be
thought as possessing degree $-1$), $\mathcal E_0 = E$ and finally
\begin{equation}\label{eq:2-rep_grpd}
  \begin{aligned}
    \partial c & = \tilde t (c) \\
    \Delta_{-1} (g, c) & = h_g (\tilde t (c)) \cdot c \cdot 0_{g^{-1}} \\
    \Delta_0 (g, e) & = \tilde t (h_g (e)) \\
    \varpi_{(g_1, g_2)} e' & 
    = \left( h_{g_1g_2} (e') - h_{g_1} (\tilde t(h_{g_2} (e'))) \cdot h_{g_2} (e') \right) \cdot 0_{(g_1g_2)^{-1}}
  \end{aligned}
\end{equation}
for all $g \in G$, $c \in C_{s(g)}$, $e \in E_{s(g)}$, $(g_1, g_2) \in
G^{(2)}$, and $e' \in E_{s(g_1)}$. Then data $(\partial, \Delta_{-1},
\Delta_0, \varpi)$ define a $2$-term representation up to homotopy
$(\mathcal E = \mathcal E_{-1}\oplus \mathcal E_0, D)$ which is
actually independent of $h$ up to isomorphisms (see
\cite{Gracia-Saz2010} for more details).

In the last part of this subsection we describe the data $(\partial ,
\Delta_{-1}, \Delta_0, \varpi)$ for the $2$-term representation up to
homotopy corresponding to the VB-groupoid $(\der E_G \rightrightarrows
\der E; G \rightrightarrows M)$ from Proposition
\ref{prop:diff_VB-grpd}. So let $G \rightrightarrows M$ be a Lie
groupoid, and let $E$ be a representation of it.
Recall from
Remark~\ref{rem:core_DE_G} that the core of $(\der E_G
\rightrightarrows \der E; G \rightrightarrows M)$ is canonically
isomorphic to the Lie algebroid $A$ of $G$. The isomorphism identifies
$a$ to $\mathbb D_a$, where $\mathbb D$ is the flat $(\ker \D
s)$-connection in $E_G$. So there is a short exact sequence of vector
bundles over $G$
\begin{equation}\label{eq:SESDEG}
0 \to t^\ast A \to \der E_G \to s^\ast \der E \to 0,
\end{equation}
where the inclusion $ t_{\der}^\ast A \to \der E_G$ is given by
\begin{equation}
(g, a) = \mathbb D_a \cdot 0_g = \mathbb D_{\overrightarrow a_g}.
\end{equation}
As $G$ acts on $E$, then $A$ acts on $E$ as well, and we denote by
$\nabla : A \to \der E$ the corresponding flat $A$-connection. Now we
choose a horizontal lift $h : s^\ast \der E \to \der E_G$, to get a
$2$-term representation up to homotopy $(\mathcal E, D)$ of $G$, with
$\mathcal E = A[1] \oplus \der E$. Specializing formulas
(\ref{eq:2-rep_grpd}) to the present situation it is easy to see that
$D$ corresponds to the data $(\partial, \Delta_{-1}, \Delta_0, \varpi)$
given by
\begin{equation}\label{eq:2-rep_DEG}
\begin{aligned}
  \partial a & = \nabla_a \\
  \Delta_{-1} (g, a) & = m_{\der} \left( h_g (\nabla_a), \mathbb D_{\overrightarrow{a}_{g^{-1}}} \right) \\
  \Delta_0 (g, \square) & = t_{\der} (h_g (\square)) \\
  \varpi_{(g_1, g_2)} \square ' 
  & = 
  \der R_{(g_1g_2)^{-1}} \left( h_{g_1g_2} (\square') - m_{\der} \left(h_{g_1} (t_{\der}(h_{g_2} (\square'))), h_{g_2} (\square') \right) \right)
\end{aligned}
\end{equation}
for all $g : x \to y \in G$, $a \in A_{s(g)}$, $\square \in
\der_{s(g)} E$, $(g_1, g_2) \in G^{(2)}$, and $\square' \in
\der_{s(g_1)} E$, where $R_{g} : s^{-1}(y) \times E_{y} \to s^{-1} (x)
\times E_x$ is the right translation.

\begin{remark}
  Let $\pi : \der E_G \to t^\ast A$, be the projection induced by the
  horizontal lift $h$. It is not hard to show that quasi-action
  $\Delta_{-1}$ is equivalently given by
  \begin{equation}\label{eq:qa_AC}
    \Delta_{-1} (g, a) = - \pi_g \left (\der L_g \circ \der i  (\mathbb D_a) \right),
  \end{equation}
  for all $g : x \to y \in G$, $a \in A_{s(g)}$, where $L_{g} :
  t^{-1}(x) \times E_{x} \to t^{-1} (y) \times E_y$ is the left
  translation, and $i : s^{-1}(x) \times E_x \to t^{-1}(x) \times E_x$
  is the inversion. Formula (\ref{eq:qa_AC}) is the ``derivation
  version'' of the $G$-quasi-action on $TM$ determined by an Ehresmann
  connection in $G$, entering the adjoint representation (up to
  homotopy) of $G$, and appearing in \cite[Definition 2.11]{CA2013}.
\end{remark}

Summarizing, we proved the following

\begin{proposition}
  Let $G \rightrightarrows M$ be a Lie groupoid and let $E$ be a
  representation of $G$. Additionally, let $h : s^\ast \der E \to
  \der E_G$ be a horizontal lift. These data define a $2$-term
  representation up to homotopy $(\mathcal E, D)$ of $G$, where
  $\mathcal E = A[1] \oplus \der E$, and $D$ corresponds to data
  $(\partial, \Delta_{-1}, \Delta_0, \varpi)$ given by
  (\ref{eq:2-rep_DEG}). The representation up to homotopy $(\mathcal
  E, D)$ does only depend on $E$, in particular it is independent of
  $h$, up to isomorphisms.
\end{proposition}

\begin{remark}
  Using the isomorphism $\der E_G \cong TG \mathbin{{}_{ds}
    \hspace*{-0.1cm}\times_\sigma \der E}$, a horizontal lift $h : s^\ast \der E \to
  \der E_G$, can be thought of as a morphism $s^\ast \der E \to TG$ of
  vector bundles over $G$, additionally covering $\sigma : \der E \to
  TM$. The \eqref{eq:2-rep_DEG} can then be re-expressed in terms of
  the action $\mathcal A$ of $TG$ on $\der E$. We leave the details to
  the reader.
\end{remark}

\begin{remark}
  The symbol, relates the VB-groupoids $(\der E_G \rightrightarrows \der E; G \rightrightarrows M)$ 
  and $(TG \rightrightarrows TM; G \rightrightarrows M)$ (Diagram \ref{eq:ses_DE_G}). 
  In turn, a horizontal lift $k : s^\ast TM \to TG$ of the tangent groupoid determines 
  a $2$-term representation up to homotopy $(\mathcal E^T, D^T)$ of $G$ where 
  $\mathcal E^T = A[1] \oplus TM$, the \emph{adjoint representation} (see \cite{CA2013} for details, 
  see also \cite{Gracia-Saz2010}). It is now natural to expect that the symbol relates $(\mathcal E, D)$ 
  and $(\mathcal E^T, D^T)$. This is indeed the case if the horizontal lifts $h : s^\ast \der E \to \der E_G$ 
  and $k : s^\ast TM \to TG$ are chosen in an appropriate 
  way. Namely, let $h, k$ be such that the diagram
  \begin{equation}
    \label{eq:h,k}
    \begin{gathered}
    \begin{tikzcd}
      s^\ast \der E \arrow[r, "h"]\arrow[d, swap, " "] & \der E_G \arrow[d, "\sigma"] \\
      s^\ast TM \arrow[r, swap, "k"]&TG
    \end{tikzcd}
    \end{gathered}
  \end{equation}
  commutes (here the left vertical arrow is that induced by the symbol $\der E \to TM$). 
  Such $h,k$ indeed exist. To see this, choose $k$ and use the isomorphism 
  $\der E_G \cong TG \mathbin{{}_{ds}\hspace*{-0.1cm} \times_\sigma \der E}$ to put
  \[
    h (g, \square) := (k(g, \sigma(\square)), \square).
  \]
  The $h,k$ have all the required properties. Finally, denote by $(\partial^T, \Delta_{-1}^T, \Delta_{0}^T, \varpi^T)$ 
  the \emph{structure data} of the adjoint representation. 
  It immediately follows from Diagrams (\ref{eq:ses_DE_G}) and (\ref{eq:h,k}) that
  \[
    \sigma \circ \partial 
    = 
    \partial^T, \quad \Delta_{-1} 
    = 
    \Delta_{-1}^T, \quad \sigma \circ \Delta_0 
    = 
    \Delta_0^T \circ (\operatorname{id}_G \times \sigma), \quad \text{and} \quad \sigma \circ \varpi = \varpi^T.
  \]
  In other words, the homomorphism of $C^\bullet (G)$-modules
  \[
    C^\bullet (G, A[1] \oplus \der E) \to C^\bullet (G, A[1] \oplus TM)
  \] 
  obtained from $\operatorname{id}_A \oplus \sigma : A[1] \oplus \der E \to A[1] \oplus TM$ 
  by extension of scalars, is a DG-module homomorphism, i.e.~it additionally 
  intertwines the differentials $D$ and $D^T$.
\end{remark}

\subsection{VB-algebroids}

We recall again from \cite{gracia2010lie} that there is a nice
relationship between VB-algebroids and representations up to
homotopy. Let $A \Rightarrow M$ be a Lie algebroid. Recall that a
\emph{representation up to homotopy} of $A$ \cite{CA2012} is a graded
vector bundle $\mathcal E \to M$ together with a degree $1$, $\mathbb
R$-linear operator
\begin{equation}
  D : \Omega(A, \mathcal E) \to \Omega (A, \mathcal E)
\end{equation}
satisfying
\begin{equation}
  D^2 = 0, 
  \quad \text{and} \quad 
  D(\omega \cdot \varepsilon) = \D_A \omega \cdot \varepsilon + (-)^{\omega} \omega \cdot D  \varepsilon, 
\end{equation}
for all $\omega \in \Omega (A)$, and all $\varepsilon \in \Omega(A,
\mathcal E)$ (i.e., $D$ is an homological derivation of the graded
$\Omega(A)$-module $\Omega(A, E)$, with symbol $\D_A$).  Here
$(\Omega(A) = \Gamma (\wedge^\bullet A^\ast), \D_A)$ is the de Rham
complex of $A$, and $\Omega (A, \mathcal E) = \Gamma (\wedge^\bullet
A^\ast \otimes \mathcal E)$. We are mainly interested in
\emph{$2$-term} representations up to homotopy, i.e.~to the case when
$\mathcal E = \mathcal E_{-1} \oplus \mathcal E_0$ is concentrated in
degree $-1, 0$.  A $2$-term representation up to homotopy $(\mathcal E
= \mathcal E_{-1} \oplus \mathcal E_0, D)$ is equivalent to the
following data
\begin{itemize}
\item a vector bundle map $\partial : \mathcal E_{-1} \to \mathcal
  E_0$,
\item an $A$-connection $\nabla_{-1}$ in $\mathcal E_{-1}$, with
  curvature $R_{-1} \in \Omega^2 (A, \End \mathcal E_{-1})$,
\item an $A$-connection $\nabla_0$ in $\mathcal E_0$, with curvature
  $R_0 \in \Omega^2 (A, \End \mathcal E_0)$ and
\item a $\Hom (\mathcal E_0, \mathcal E_{-1})$-valued $2$-form $R \in
  \Omega^2 (A, \Hom (\mathcal E_0, \mathcal E_{-1}))$,
\end{itemize}
such that
\begin{equation}
\partial \circ (\nabla_{-1})_a 
= 
(\nabla_0)_a \circ \partial, \quad R_{-1} (a, b) 
= 
R(a, b) \circ \partial, \quad  R_0 (a, b) 
= 
\partial \circ R(a, b), \quad \text{and} \quad d_{\nabla^{\Hom}} R = 0,
\end{equation}
for all $a, b \in \Gamma (A)$, where $\nabla^{\Hom}$ is the
$A$-connection in $\Hom (\mathcal E_0, \mathcal E_{-1})$ induced by
$\nabla_{-1}, \nabla_0$. The data $(\partial, \nabla_{-1}, \nabla_0, R)$
correspond to the representation up to homotopy whose differential $D$
is
\begin{equation}
D = \partial + \nabla_{-1} + \nabla_0 + R
\end{equation}
where we interpret $\partial, \nabla_{-1}, \nabla_0, R$ as degree
$1$ operators on $\Omega(A, \mathcal E_{-1} \oplus \mathcal E_0)$ in
the obvious way (see \cite{gracia2010lie} for more details).

The main remark here is that VB-algebroids $(W \Rightarrow E; A
\Rightarrow M)$ are equivalent to $2$-term representations up to
homotopy of a Lie algebroid $A \Rightarrow M$ up to isomorphisms. From
$(W \Rightarrow E; A \Rightarrow M)$ one can construct the
corresponding (isomorphism class of) $2$-term representation(s) up to
homotopy as follows. As usual, denote by $C$ the core and by $\widehat
A$ the fat algebroid, and consider the short exact sequence
\begin{equation}
0 \to \Hom (E, C) \to \widehat A \to A \to 0.
\end{equation}
Choose a splitting $\sigma : A \to \widehat A$. Now put $\mathcal
E_{-1} = C[1]$, $\mathcal E_0 = E$, let $\partial : C \to E$ be the core-anchor, and additionally
\begin{equation}
\nabla_{-1} = \psi^c \circ \sigma, \quad \nabla_0 = \psi^s \circ \sigma,
\end{equation}
where $ \psi^s, \psi^c$ are the side and the
core representation respectively. Finally, we let $R$ be given by
\begin{equation}
R(a,b) = \sigma [a, b] - [\sigma(a), \sigma (b)], \quad a,b \in \Gamma (A).
\end{equation}

In the last part of this section we describe the $2$-term
representation up to homotopy corresponding to the VB-algebroid $(\der
E_A \Rightarrow \der E; A \Rightarrow M)$ from Proposition
\ref{prop:diff_VB-alg}. So let $A \Rightarrow M$ be a Lie algebroid,
and let $E$ be a representation of it. We denote by $\nabla$ the flat
$A$-connection in $E$.  We have already seen that the core of $(\der
E_A \Rightarrow \der E; A \Rightarrow M)$ is (canonically isomorphic
to) $A$. So there is a short exact sequence of vector bundles
\begin{equation}
\label{eq:SESDEA}
0 \to \Hom (\der E, A) \to \widehat A \to A \to 0
\end{equation}
where $\widehat A$ is the fat algebroid. We begin describing $\widehat A$.
Denote by $p : A \to M$ the projection.

\begin{lemma}\label{lem:fat_alg_DEA}
  The fiber $\widehat A_x$ of the fat algebroid $\widehat A
  \Rightarrow M$ of $(\der E_A \Rightarrow \der E; A \Rightarrow M)$
  over a point $x \in M$ consists of pairs $(a, h)$, where $a \in A_x$
  and $h : \der_x E \to T_a A$ is a linear map such that $dp \circ h =
  \sigma$. The projection $\widehat A \to A$ maps $(a, h)$ to $a$.
\end{lemma}

\begin{proof}
  The claim immediately follows from the isomorphism 
  $\der E_A \cong TA \mathbin{{}_{dp} \hspace*{-0.1cm} \hspace*{-0.1cm}\times_\sigma } \der E$.
\end{proof}

\begin{remark}\label{rem:APPL_Luca}
  The map $\Gamma (A) \to \Gamma (\widehat A)$, $a \mapsto \der a$
  splits the short exact sequence
  \begin{equation}
  0 \to \Gamma (\Hom (\der E, A)) \to \Gamma (\widehat A) \to \Gamma(A) \to 0
  \end{equation}
  in the category of vector spaces. In particular, a section of
  $\widehat A \to M$ can be seen as a pair $(a, \varphi)$ where $a \in
  \Gamma (A)$ and $\varphi \in \Gamma (\Hom (\der E, A))$.  We now
  take this point of view to describe a little bit more explicitly
  (the core-anchor and) the core, the side representations and the Lie
  bracket in $\Gamma (\widehat A)$. Actually, it easily follows from
  the definition of the Lie algebroid $\der E_A \Rightarrow \der E$
  that
  \begin{itemize}
  \item the core-anchor $\partial : A \to DE$ agrees with the
    $A$-connection $\nabla$ in $E$;
  \item the side representation is given by $\psi^s_{\der a} =
    [\nabla_a, -]$ and $\psi^s_\varphi = - \nabla \circ
    \varphi$, \item the core representation is given by $\psi^c_{\der
      a} = [a, -]$ and $\psi^c_{\varphi} = - \varphi \circ \nabla$;
  \item the Lie bracket in $\Gamma (\widehat A)$ is given by
    \begin{equation}
      [\der a, \der b] 
      = 
      \der [a, b], \quad [\der a, \varphi] 
      = 
      \mathcal L_a \varphi, \quad [\varphi, \upsilon] 
      = 
      \upsilon \circ \nabla \circ \varphi - \varphi \circ \nabla \circ \upsilon;
    \end{equation}
  \end{itemize}
  for all $a, b \in \Gamma (A)$, and $\varphi, \upsilon \in \Gamma
  (\Hom (\der E, A))$. Here $\mathcal L_a \varphi$ is the bundle map $\der
  E \to A$ given by
  \begin{equation}
    (\mathcal L_a \varphi)(\Delta) 
    := 
    [a, \varphi (\Delta)] - \varphi ([\nabla_a, \Delta]), \quad \Delta \in \Gamma (\der E).
  \end{equation}
\end{remark}

Now we choose a splitting of (\ref{eq:SESDEA}). It follows from Lemma
\ref{lem:fat_alg_DEA} that such a splitting is equivalent to a $\der
E$-connection $\mathcal D$ in $A$. Namely, the connection $\mathcal D$
corresponds to the splitting $\sigma_{\mathcal D} : A \to \widehat A$
given by
\begin{equation}
  \sigma_{\mathcal D} (a) := \der a - \D_{\mathcal D} a, \quad a \in \Gamma (A),
\end{equation}
where $\D_{\mathcal D} : \Gamma (A) \to \Gamma (\Hom (\der E, A))$ is
the connection differential.  Hence, we get a $2$-term representation
up to homotopy $(\mathcal E, D)$ of $A$ by putting $\mathcal E = A[1]
\oplus \der E$ and
\begin{equation}
  D = \partial + \nabla_{-1} + \nabla_0 + R
\end{equation}
with
\begin{equation}
  \partial 
  = 
  \nabla, \quad \nabla_{-1} = \psi^s \circ \sigma_{\mathcal D}, \quad \nabla_0 = \psi^c \circ \sigma_{\mathcal D},
\end{equation}
and 
\begin{equation}
  R(a,b) 
  = 
  \sigma_{\mathcal D} [a, b] - [\sigma_{\mathcal D} (a), \sigma_{\mathcal D} (b)], \quad a,b \in \Gamma (A).
\end{equation}
A straightforward computation using Remark \ref{rem:APPL_Luca} now
shows that
\begin{equation}\label{eq:nabla-1,0}
  (\nabla_{-1})_a \Delta = [\nabla_a, \Delta] + \nabla_{\mathcal D_\Delta a}, \quad (\nabla_0)_a b = [a, b] + \mathcal D_{\nabla_b} a
\end{equation}
and
\begin{equation}\label{eq:R}
  \begin{aligned}
    R(a, b) \Delta & =  [\mathcal D_\Delta a, b] + [a, \mathcal D_\Delta b]  - \mathcal D_\Delta [a,b] \\
    & \quad - \mathcal D_{[\nabla_a, \Delta]} b + \mathcal D_{[\nabla_b, \Delta]} a + \mathcal D_{\nabla_{\mathcal D_\Delta b}} a
    - \mathcal D_{\nabla_{\mathcal D_\Delta a}} b,
  \end{aligned}
\end{equation}
for all $a, b \in \Gamma (A)$, and all $\Delta \in \Gamma (\der E)$.

Summarizing, we proved the following

\begin{proposition}
  Let $A \Rightarrow M$ be a Lie algebroid and let $E$ be a
  representation of $A$ with flat $A$-connection
  $\nabla$. Additionally, let $\mathcal D$ be a (not necessarily flat)
  $\der E$-connection in $A$. These data define a $2$-term
  representation up to homotopy $(\mathcal E, D)$ of $A$, where
  $\mathcal E = A[1] \oplus \der E$, and $D = \nabla + \nabla_{-1} +
  \nabla_{0} + R$, with $\nabla_{-1}$, $\nabla_0$, and $R$ being given
  by (\ref{eq:nabla-1,0}) and (\ref{eq:R}). The representation up to
  homotopy $(\mathcal E, D)$ does only depend on $(E, \nabla)$, in
  particular it is independent of $\mathcal D$, up to isomorphisms.
\end{proposition}

\section*{Acknowledgements}
The authors thank Iakovos Androulidakis for useful discussions. They
are also thankful to the anonymous referee for remarks that simplified
a lot various proofs.  AGT is supported by a FWO postdoctoral
fellowship and further acknowledges the partial support received,
during the preparation of this paper, by the GNSAGA (INdAM), the
Centre for Mathematics of the University of Coimbra
(UID/MAT/00324/2013), and the FWO research project G083118N (Belgium).
LV is member of the GNSAGA of INdAM, Italy.

%
%

{
  \footnotesize
  \renewcommand{\arraystretch}{0.5}

}

%
%

\end{document}